%% file: solvable-Part1-2.tex
\documentclass[10pt]{article}
\usepackage{amsmath}
\usepackage{latexsym}
\usepackage{amsthm}
\usepackage{amssymb}
\usepackage[normalem]{ulem}
\usepackage{graphicx}
\usepackage{xr}
\usepackage{color}
\usepackage{epsfig}

\newtheorem{definition}{Definition}[subsection]
\newtheorem{theorem}{Theorem}[section]
\newtheorem{proposition}{Proposition}[subsection]
\newtheorem{lemma}{Lemma}[subsection]
\newtheorem{corollary}{Corollary}[subsection]
\newtheorem{remark}{Remark}[subsection]

\renewcommand{\labelenumi}{(\roman{enumi})}
\newcommand{\gothic}{\mathfrak}
\renewcommand{\bold}[1]{\medskip \noindent {\bf #1 }\nopagebreak}
\long\def\symbolfootnote[#1]#2{\begingroup%
\def\thefootnote{\fnsymbol{footnote}}\footnote[#1]{#2}\endgroup}

\begin{document}

\title{Coarse differentiation and quasi-isometries of a class of solvable Lie groups I}
\date{}
\author{Irine Peng}
\maketitle

\begin{abstract}
This is the first of two papers (the other one being \cite{Pg}) which aim to understand quasi-isometries of a subclass of unimodular split solvable Lie groups.
In the present paper, we show that locally (in a coarse sense), a quasi-isometry between
two groups in this subclass is close to a map that respects their group structures.  \end{abstract}

\tableofcontents

\section{Introduction}
A \emph{$(\kappa, C)$ quasi-isometry} $f$ between metric spaces $X$ and $Y$ is a map $f: X \rightarrow Y$ satisfying
\[ \frac{1}{\kappa} d(p,q) -C \leq d(f(p), f(q)) \leq \kappa d(p,q) + C \] \noindent with the additional property
that there is a number $D$ such that $Y$ is the $D$ neighborhood of $f(X)$.  Two quasi-isometries $f, g$ are
considered to be equivalent if there is a number $E>0$ such that $d(f(p), g(p)) \leq E$ for all $p \in
X$. \smallskip

From \cite{Auslander}, any solvable Lie group $\mathcal{L}$ has the form

\[ 1 \rightarrow \mathcal{U} \rightarrow \mathcal{L} \rightarrow \mathbb{R}^{s}
\rightarrow 1 \]

\noindent where $\mathcal{U}$ largest connected normal nilpotent subgroup of
$\mathcal{L}$, called its \emph{nilradical}, and $\mathbb{R}^{s}$ is the abelianization of
its Cartan subgroup.

In a group $G$, an element $x \in G$ is called \emph{exponentially distorted} if there are numbers $c, \epsilon$ such that for all
$n \in \mathbb{Z}$,
\[ \frac{1}{c} \log(|n| + 1) - \epsilon \leq \| x^{n} \|_{G} \leq c \log(|n|+1) + \epsilon
\]
\noindent where $\|x^{n} \|_{G}$ is the distance between the identity and $x^{n}$ in $G$.

In the case of a connected, simply connected solvable Lie group $G$, Osin showed in \cite{Osin} that the set of
exponentially distorted elements forms a normal subgroup $R_{\emph{exp}}(G)$ inside of
the nilradical of $G$.

Motivated by the Gromov program of classifying groups up to quasi-isometries, we consider, in this two-part paper, quasi-isometries
between connected, simply-connected unimodular solvable Lie group $G$ whose exponential radical coincides with its nilradical and
is a semidirect product between its abelian Cartan subgroup and
its abelian nilradical that is 'irreducible' in some sense. (For example, is not a
direct product with abelian factors).  By applying the techniques introduced by Eskin-Fisher-Whyte in \cite{EFW0},
\cite{EFW1}, and \cite{EFW2}, we are able to show that \medskip

\noindent \textit{(Theorem \ref{Andre-behaviors of phi} in \cite{Pg}(abridged))}
\textit{Let $G$, $G'$ be non-degenerate, unimodular, split abelian-by-abelian solvable Lie groups, and $\phi: G \rightarrow G'$ a $\kappa, C$
quasi-isometry.  Then $\phi$ is bounded distance from a composition of a left translation and a standard
map.}  \smallskip

\noindent Here a standard map is one that respect the factors in the semidirect product
and their group structures. (See definition \ref{standard map}). \newline

Consequently, we are able to see that \newline
\textit{(Corollary \ref{Andre-QIgroup} in \cite{Pg} ) }
\[ \mathcal{QI}(G) = \left( \prod_{[\alpha]} Bilip(V_{[\alpha]}) \right) \rtimes \mbox{Sym}(G)
\]

\noindent Here $V_{[\alpha]}$'s are subspaces of the nilradical, and $Sym(G)$ is a finite group,
analogous to the Weyl group in reductive Lie groups. It reflects the symmetries of $G$. (See section \ref{geometry of G})
\newline

Writing a non-degenerate, unimodular, split abelian-by-abelian solvable group as $G=\mathbf{H} \rtimes_{\varphi} \mathbf{A}$,
where $\mathbf{H}$ is the abelian nilradical and $\mathbf{A}$ an abelian Cartan subgroup.  We can also distinguish groups depending on whether the
action of the Cartan subgroup on the nilradical (via $\varphi$) is diaonalizable or not.
\smallskip

\noindent \textit{(Corollary \ref{Andre-diagORnot} in \cite{Pg})}\textit{ Let $G$, $G'$ be non-degenerate, unimodular, split abelian-by-abelian solvable
Lie groups where actions of their Cartan subgroups on the nilradicals are $\varphi$ and
$\varphi'$ respectively.  If $\varphi$ is diagonalizable and $\varphi'$ isn't, then there is no quasi-isometry
between them.} \newline

When $\varphi$ is diagonalizable, as an application the work by Dymarz \cite{Dy} on quasi-conformal maps on the boundary of
$G$, and a theorem of Mostow that says polycyclic groups are virtually lattices in a
connected, simply connected solvable Lie group, we have \smallskip

\noindent \textit{(Corollary \ref{Andre-showing polycyclic}, \ref{Andre-showing lattice in almost
me} in \cite{Pg})} \textit{In the case that $\varphi$ is diagonalizable, if $\Gamma$ is a finitely generated group quasi-isometric to
$G=\mathbf{H} \rtimes_{\varphi} \mathbf{A}$, then $\Gamma$ is virtually polycyclic, and is virtually a lattice in a unimodular
semidirect product of $\mathbf{H}$ and $\mathbf{A}$. } \medskip

\noindent Note that in the statement above we are not able to determine if the target semidirect product of
$\mathbf{H}$ and $\mathbf{A}$ is actually $G$ because the latter is a semidirect product
of the same factors with some additional conditions, which we are not able to detect at
this stage. \\

All the argument in this paper are local in nature and below is a description of the main result.
\\

Let $G= \mathbf{H} \rtimes_{\varphi} \mathbf{A}$, $G'=\mathbf{H'} \rtimes_{\varphi'} \mathbf{A'}$ be connected,
simply connected non-degenerate unimodular split solvable groups (See section \ref{geometry of G} for definitions).  We say a map
from $G$ to $G'$ is \emph{standard}, if it splits as a product map that respects $\varphi$ and $\varphi'$
(See definition \ref{standard map}). \smallskip

A compact convex set $\Omega \subset \mathbb{R}^{n}$ determines a bounded set $\mathbf{B}(\Omega)$ in $G$ (See section
2.2).  Writing $\rho \Omega$ for the compact convex set obtained by scaling $\Omega$ by $\rho$ from the barycenter of
$\Omega$, we show in this paper that \smallskip

\begin{theorem} \label{exisence of standard maps in small boxes}
Let $G$, $G'$ be non-degenerate, unimodular, split abelian-by-abelian Lie groups, and
$\phi:G \rightarrow G'$ be a $(\kappa, C)$ quasi-isometry.  Given $0< \delta, \eta < \tilde{\eta} < 1$, there exist numbers $L_{0}$, $m > 1$, $\varrho, \hat{\eta} < 1$ depending on $\delta$,
$\eta$, $\tilde{\eta}$ and $\kappa, C$ with the following properties: \smallskip

If $\Omega \subset \mathbf{A}$ is a product of intervals of equal size at least $mL_{0}$,
then a tiling of $\mathbf{B}(\Omega)$ by isometric copies of $\mathbf{B}(\varrho \Omega)$

\[ \mathbf{B}(\Omega)= \bigsqcup_{i \in \mathbf{I}} \mathbf{B}(\omega_{i}) \sqcup \Upsilon \]

\noindent contains a subset $\mathbf{I}_{0}$ of $\mathbf{I}$ with relative measure at
least $1-\nu$ such that

\begin{enumerate}

\item  For every $i \in \mathbf{I}_{0}$, there is a subset $\mathcal{P}^{0}(\omega_{i})$
of $\mathbf{B}(\omega_{i})$ of relative measure at least $1-\nu'$

\item The restriction $\phi|_{\mathcal{P}^{0}(\omega_{i})}$ is within $\hat{\eta}
diam(\mathbf{B}(\omega_{i}))$ Hausdorff neighborhood of a standard map $g_{i} \times
f_{i}$. \end{enumerate}

Here, $\nu$, $\nu'$ and $\hat{\eta}$ all approach zero as $\tilde{\eta}$, $\delta$ go to
zero.  The measure of set $\Upsilon$ is at most $\delta'$ proportion of measure of
$\mathbf{B}(\Omega)$, where $\delta'$ depends on $\delta$ and goes to zero as the latter
approaches zero. \end{theorem}

\subsection{Proof outline}
The idea of the proof is as follows.  We employ the technique of `coarse differentiation'
to images of a particular family of geodesics (which fills up the set $\mathbf{B}(\Omega)$) in
$\mathbf{B}(\Omega)$ to obtain the scale $\rho$ on which those quasi-geodesics behave
like certain simple geodesics.  We are also able to obtain a tiling because the group $G$
is unimodular and $\mathbf{B}(\Omega)$ have small boundary area compared to its volume.
We then use the properties of the groups being non-degenerate, unimodular and split
abelian-by-abelian to reach the conclusion on those smaller tiles.


\bold{Acknowledgement} I would like to thank Alex Eskin for his patience and guidance.  I
also owe much to David Fisher for his help and support.

\section{Preliminaries}
In this section, we first describe the geometry of the subclass of unimodular solvable
Lie group mentioned in Introduction, followed by a list of notations that will be used in
the remaining of this paper.

\subsection{Geometry of a certain class of solvable Lie groups} \label{geometry of G}
\bold{Non-degenerate, split abelian-by-abelian solvable Lie groups }
Let $\gothic{g}$ be a (real) solvable Lie algebra, and $\gothic{a}$ be a Cartan
subalgebra.  Then there are finitely many non-zero linear functionals $\alpha_{i}: \gothic{a}
\rightarrow \mathbb{C}$ called \emph{roots}, such that

\[ \gothic{g} = \gothic{a} \oplus \bigoplus_{\alpha_{i}} \gothic{g}_{\alpha_{i}} \]

\noindent where $\gothic{g}_{\alpha_{i}}=\{ x \in \gothic{g}: \forall t \in \gothic{a}, \exists n, \mbox{ such that }
(ad(t)-\alpha_{i}(t)Id)^{n}(x) = 0 \}$, $Id$ is the identity map on $\gothic{g}$, and
$ad: \gothic{g} \rightarrow Der_{\mathbb{R}}(\gothic{g})$ is the adjoint representation.

We say $\gothic{g}$ is \emph{split abelian-by-abelian} if $\gothic{g}$ is a semidirect product
of $\gothic{a}$ and $\bigoplus_{i} \gothic{g}_{\alpha_{i}}$, and both are abelian Lie algebras;
\emph{unimodular} if the the roots sum up to zero; and \emph{non-degenerate} if the roots span $\gothic{a}^{*}$.  In particular,
non-degenerate means that each $\alpha_{i}$ is real-valued, and the number of roots is at least the dimension of
$\gothic{a}$.  Being unimodular is the same as saying that for every $t \in \gothic{a}$, the trace of
$ad(t)$ is zero.  We extend these definitions to a Lie group if its Lie algebra has these
properties. \smallskip

Therefore a connected, simply connected solvable Lie group $G$ that is non-degenerate, split abelian-by-abelian necessary takes the form
$G= \mathbf{H} \rtimes_{\varphi} \mathbf{A}$ such that  \begin{enumerate}
\item both $\mathbf{A}$ and $\mathbf{H}$ are abelian Lie groups.

\item the homomorphism $\varphi: \mathbf{A} \rightarrow Aut(\mathbf{H})$ is injective

\item there are finitely many $\alpha_{i} \in \mathbf{A}^{*} \backslash 0$ which together span
$\mathbf{A}^{*}$, and a decomposition of $\mathbf{H}=\oplus_{i} V_{\alpha_{i}}$

\item there is a basis $\mathcal{B}$ of $\mathbf{H}$ whose intersection with each of
$V_{\alpha_{i}}$ constitute a basis of $V_{\alpha_{i}}$, such that for each $\mathbf{t} \in \mathbf{A}$, $\varphi(\mathbf{t})$ with respect
to $\mathcal{B}$ is a matrix consists of blocks, one for each $V_{\alpha_{i}}$, of the form $e^{\alpha_{i}(\mathbf{t})} N(\alpha_{i}(t))$, where
$N(\alpha_{i}(t))$ is an upper triangular with 1's on the diagonal and whose off-diagonal entries are polynomials of
$\alpha_{i}(t)$. \smallskip
If in addition, $G$ is unimodular, then $\varphi(\mathbf{t})$ has determinant 1 for all
$\mathbf{t} \in \mathbf{A}$.  \end{enumerate}
\noindent The \emph{rank} of a non-degenerate, split abelian-by-abelian group $G$ is
defined to be the dimension of $\mathbf{A}$, and by a result of Cornulier \cite{Cornulier}, if two such groups are quasi-isometric,
then they have the same rank. \smallskip

Let $\triangle$ denotes the roots of $G$.  For each $\alpha \in \triangle$, choose a
basis $\{ e^{\alpha}_{1}, e^{\alpha}_{2}, \cdots e^{\alpha}_{n_{\alpha}} \}$ in $V_{\alpha}$ such that
$\varphi(\mathbf{t})|_{V_{\alpha_{i}}}$ is upper triangular for all $\mathbf{t} \in \mathbf{A}$.  Also fix a basis $\{ E_{j} \}$ in
$\mathbf{A}$ (for example, the duals of a subset of roots), and for each $\mathbf{t} \in \mathbf{A}$, write $\mathbf{t}_{j}$ for its
$E_{j}$ coordinate.  We coordinatize a point $(\sum_{\alpha \in \triangle}
\sum_{j=1}^{n_{\alpha}} x_{j,\alpha} e^{\alpha}_{j})(\mathbf{t}) \in \mathbf{H}
\rtimes_{\varphi} \mathbf{A}$ by the $dim(G)$-tuple of numbers
$((\mathbf{x}_{\alpha})_{\alpha}, (\mathbf{t}_{j})) \in \mathbb{R}^{dim(G)}$, where
$\mathbf{x}_{\alpha}=(x_{1,\alpha}, x_{2, \alpha}, \cdots, x_{dim(V_{\alpha}),\alpha} )$.
 In this coordinate system, a left invariant Riemannian metric at $((\mathbf{x}_{\alpha})_{\alpha}, (\mathbf{t}_{j})_{j})$ is

\[ \sum_{j} d ( \mathbf{t}_{j} )^{2} + \sum_{\alpha \in \triangle}
e^{-2\alpha(\mathbf{t})} \sum_{i=1}^{\alpha_{i}} \left( d x_{i,\alpha}
+ \sum_{\iota=i+1}^{n_{\alpha}} P^{\alpha}_{i,\iota}(\alpha(-\mathbf{t})) dx_{\iota,\alpha} \right)^{2} \]

\noindent where $P^{\alpha}_{i, \iota}$ is a polynomial with no constant term.  We see
that the above Riemannian metric is bilipschitz to the following Finsler metric:

\[ |d \mathbf{t} |  + \sum_{\alpha \in \triangle}
e^{-\alpha(\mathbf{t})} \sum_{i=1}^{n_{\alpha}}  [ 1 + Q_{i,\alpha}(\alpha(-\mathbf{t}))] |d x_{i,\alpha}|  \]

\noindent where $|d \mathbf{t}|$ means $\sum_{j} | d \mathbf{t}_{j} |$, and $Q_{i,\alpha}$ is sum of absolute values of polynomials
with no constant terms.

\begin{remark}
Since we defined our metric to be left-invariant, left multiplication by an element of $G$ is an isometry.  On the other hand, right
multiplication typically distorts distance.  For example, for points $p,q \in
\mathbf{H}$, $\mathbf{t} \in \mathbf{A}$, $d(\mathbf{t}p, \mathbf{t}q)=d(p,q)$, but $d(p
\mathbf{t}, q \mathbf{t})$ usually is some exponential-polynomial multiple of $d(p,q)$.  \end{remark}

Let $\mathit{H}_{s+1}=\mathbb{R}^{s} \rtimes_{\psi} \mathbb{R}$ be a non-unimodular solvable Lie group
such that with respect to bases $\{ e_{i} \}$, $\{ E \}$ of $\mathbb{R}^{s}$ and
$\mathbb{R}$ respectively, we have $\psi(tE)=e^{at} N(t)$, for all $t \in \mathbb{R}$.  Here $a>0$ and $N(t)$ is unipotent matrix
(upper triangular with 1's on the diagonal) with polynomial entries.  By giving a point $(\sum x_{i}e_{i})(tE) \in \mathit{H}^{s+1}$
the coordinate of $(x_{1},x_{2}, \cdots x_{s}, t)$, and argue as above we see that a left-invariant Finsler metric bilipschitz to
a left-invariant Riemannian metric can be given as
\begin{equation} \label{rank1 Finsler}
 | dt | + e^{-a t} \sum_{i} [1 + P_{i}(at)] |dx_{i}| \end{equation}

\noindent where $P_{i}$ is the sum of absolute values of polynomials with no constant terms.
\newline

The following consequence is immediate.

\begin{lemma} \label{QI embedding}
If $G$ is non-degenerate, split abelian-by-abelian, then it can be QI embedded into $\prod_{\alpha \in \triangle}
\mathit{H}_{\mbox{dim}(V_{\alpha})+1}$.  \end{lemma}

\begin{remark} When $\psi(t)$ is diagonal, $\mathit{H}_{s+1}$ is just the usual hyperbolic space. \end{remark}

\begin{proof}
Because of the following relation

\begin{eqnarray*}
&& \frac{1}{|\triangle|} \sum_{\alpha \in \triangle}  \left( |d \alpha(\mathbf{t})|
+ e^{-\alpha(\mathbf{t})} \sum_{i=1}^{n_{\alpha}}  [ 1 + Q_{i,\alpha}(\alpha(-\mathbf{t}))] |d x_{i,\alpha}| \right) \leq
|d \mathbf{t} | + \sum_{\alpha \in \triangle} e^{-\alpha(\mathbf{t})} \sum_{i=1}^{n_{\alpha}}
[ 1 + Q_{i,\alpha}(\alpha(-\mathbf{t}))] |d x_{i,\alpha}| \\
& \leq & \sum_{\alpha \in \triangle} \left(  |d \alpha(\mathbf{t})|
+ e^{-\alpha(\mathbf{t})} \sum_{i=1}^{n_{\alpha}}  [ 1 + Q_{i,\alpha}(\alpha(-\mathbf{t}))] |d x_{i,\alpha}| \right)
\end{eqnarray*}  \end{proof}

%
%

To understand the geometry of $\mathit{H}_{s+1}$ better, we can assume without loss of generality that $a=1$,
and note that the Finsler metric in equation (\ref{rank1 Finsler}) is quasi-isometric to one given by
$dt + e^{-t} Q(t) d\mathbf{x}$ for some polynomial $Q(t)$.  Since exponential grows
faster than polynomials, for any large positive number $x$, there is a $t_{0}$ such that $e^{-t}Q(t)x \leq 1$
for all $t \geq t_{0}$, and we see that a function q.i. to the metric on $\mathit{H}_{s+1}$ is the following
\begin{equation}\label{distance}
d((\mathbf{x}_{1},t_{1}), (\mathbf{x}_{2},t_{2})) = \left \{ \begin{array}{ll}
                                     |t_{1}-t_{2}| & \mbox{ if $e^{-t_{i}}Q(t_{i}) |\mathbf{x}_{1}-\mathbf{x}_{2}| \leq 1 $ for some $i=1,2$};\\
                               U_{Q}(|\mathbf{x}_{1}-\mathbf{x}_{2}|) - (t_{1}+ t_{2}) & \mbox{ otherwise} \end{array} \right. \end{equation}
\noindent where $U_{Q}(|\mathbf{x}_{1}-\mathbf{x}_{2}|)=t_{0}$ satisfies
\[ e^{-t_{0}} Q(t_{0}) |\mathbf{x}_{1}-\mathbf{x}_{2}| = 1 \]

Furthermore, the following relation
\[ \frac{1}{e^{t}} < \frac{Q(t)}{e^{t}} < \frac{e^{1/2t}}{e^{t}} \mbox{     }  \mbox{ for $t$ sufficiently large } \]

\noindent and the fact that both $e^{-t}$ and $Q(t)e^{-t}$ are decreasing functions when $t$ becomes big enough means that we have
the following inequalities for their inverses:
\begin{equation} \label{property of U}
\ln(x) - C_{Q} \leq    U_{Q}(x)    \leq    2 \ln(x) + C_{Q} \mbox{     } \mbox{ for $x > 1$ } \end{equation}

\noindent for some constant $C$ depends only on the polynomial $Q$. \medskip

Back to the description of $G$, we declare two roots equivalent if they are positive multiples of each other, and write $[\Xi]$ for the equivalence
class containing $\Xi \in \triangle$.  A left translate of $V_{[\Xi]}=\oplus_{\sigma \in [\Xi]} V_{\sigma}$ will be called a
\emph{horocycle of root class $[\Xi]$}.

A left translate of $\mathbf{H}$, or a subset of it, is called a \emph{flat}.  For two
points $p,q \in \mathbf{H}$ with coordinates $(\mathbf{x}_{\alpha})_{\alpha \in
\triangle}$ and $(\mathbf{y}_{\alpha})_{\alpha \in \triangle}$, we compute subsets of $p \mathbf{H}$ and $q \mathbf{H}$ that are within
distance 1 of each other according to the embedded metric in Lemma \ref{QI embedding}, as the $p$ and $q$ translate of the
subset of $\mathbf{A}$:
\[  \bigcap_{\alpha \in \triangle : \ln(|\mathbf{x}_{\alpha}-\mathbf{y}_{\alpha}|) \geq 1}  \alpha^{-1}
[U_{\alpha}(|\mathbf{x}_{\alpha}-\mathbf{y}_{\alpha}|), \infty ]  \]

As the roots sum up to zero in a non-degenerate, unimodular, split
abelian-by-abelian group, the set where two flats come together can be empty, i.e. the two flats have
no intersection.  If it is not empty, then the equation above says that it is an unbounded convex subset of $\mathbf{A}$ bounded by hyperplanes
parallel to root kernels.

\begin{definition} \label{standard map}
Let $G$, $G'$ be non-degenerate, split abelian-by-abelian Lie groups.  A map from $G$ to $G'$ or a subset of them, is called
\emph{standard} if it takes the form $f \times g$, where $g: \mathbf{H} \rightarrow
\mathbf{H}'$ sends foliation by root class horocycles of $G$ to that of $G'$, and $f: \mathbf{A} \rightarrow
\mathbf{A}$ sends foliations by root kernels of $G$ to that of $G'$.\end{definition}

\begin{remark}
Note that when $G$ has at least $rank(G)+1$ many root kernels,
the condition on $f$ means that $f$ is affine, and when $G$ is rank 1, the condition on $f$ is
empty. \end{remark}


\subsection{Notations}

\subsubsection{General remarks about paths, neighborhoods}
\bold{Division of a curve }
The word 'scale' shall mean a number $\rho \in (0,1]$.  We will often examine a quasi-geodesic on different 'scales', and
see if the quasi-geodesic 'on that scale' satisfies certain properties. This roughly means that we subdivide the quasi-geodesic
into subsegments whose lengths are $\rho$ times the length of the original one, and see if each one of them satisfies certain
properties. \smallskip  \newline In practice, however, instead of dealing with 'length', we use 'distance between end points' of a
curve. More precisely, let $\zeta:[a,b] \rightarrow Y$ be  rectifiable curve.
\begin{itemize}
\item Given $r>0$, we can divide $\zeta$ into subsegments whose end points are $r$ apart. \newline More precisely,
$\hat{\mathcal{S}}(\zeta,r) = \{ q_{i} \}_{i=1}^{n_{r}}$, is the set of the dividing points on $\zeta$, where $q_{0}=\zeta(a)$,
$q_{n_{r}}=\zeta(b)$, and
\begin{equation*}
\zeta^{-1}(q_{i+1})= \min \{ t \geq \zeta^{-1}(q_{i}) \mbox{    } |\mbox{     }   d(\zeta(t),q_{i}) = r \}
\end{equation*}

\item Given two points $p, q \in \zeta$, we write $\zeta_{[p,q]}$ for the part of $\zeta$ between $p$ and $q$.
Define $\mathcal{S}(\zeta,r) = \{ \zeta_{[q_{i}, q_{i+1}]} \}$, to be the set of subsegments after division.

\item Let $\mathbf{P}$ be a statement. Define $\mathcal{S}(\zeta,r,\mathbf{P} )=\{\zeta^{i} \in
\mathcal{S}(\zeta,r) | \zeta^{i} \mbox{ satisfies } \mathbf{P} \}$ to be those subsegments satisfying statement $\mathbf{P}$.

\item We write $| \zeta |$ for the distance between end points of $\zeta$, and $\| \zeta \|$ denotes for the length of $\zeta$.
\end{itemize}

\bold{Neighborhoods of a set }
We write $B(p,r)$ for the ball centered at $p$ of radius $r$, and $N_{c}(A)$ for the $c$ neighborhood of the set
$A$.  We also write $d_{H}(A,B)$ for the Hausdorff distance between two sets $A$ and $B$.  If $\Omega \subset \mathbb{R}^{k}$ is a
bounded compact set, and $r \in \mathbb{R}$, we write $r \Omega$ for the bounded compact set that is scaled from $\Omega$ with respect
to the barycenter of $\Omega$. \smallskip

Given a set $X$, a point $x_{0} \in X$, the \emph{$(\eta, C)$ linear neighborhood of $X$ with respect to
$x_{0}$} is the set $\{ y, s.t. \exists \hat{x} \in X, d(y,\hat{x})=d(y,X) \leq \eta d(\hat{x},x_{0}) + C$.
Equivalently it is the set $\bigcup_{x \in X}B(x,\eta d(x,x_{0})+C)$.  By $(\eta, C)$ linear neighborhood of a set $X$, we mean
the $(\eta, C)$ linear neighborhood of $X$ with respect to some $x_{0}\in X$. \smallskip

If a quasi-geodesic $\lambda$ is within $(\eta, C)$ linear (or just $\eta$-linear) neighborhood of a geodesic
segment $\gamma$, where $\eta \ll 1$ and $C \ll \eta |\lambda|$, then we say that $\lambda$
\emph{admits a geodesic approximation} by $\gamma$.

\subsubsection{Notations used in split abelian-by-abelian groups}
Let $G=\mathbf{H} \rtimes \mathbf{A}$ stands for a non-degenerate, split abelian-by-abelian group, and $\triangle$
denotes for its roots.  Fix a point $p \in G$.  We define the following:
\begin{itemize}
\item For $\alpha \in \triangle$ a root, we write $\vec{v}_{\alpha}$ for the dual of
$\alpha_{i}$ of norm 1 with respect to the usual Euclidean metric.  (This is really a function
on root classes.)

\item Given $\vec{v} \in \mathbf{A}$, we define \begin{eqnarray*}
W_{\vec{v}}^{+} &=& \oplus_{\Xi(\vec{v}) >0} V_{\Xi} \\
W_{\vec{v}}^{-} &=& \oplus_{\Xi(\vec{v}) <0} V_{\Xi} \\
W_{\vec{v}}^{0} &=& \oplus_{\Xi(\vec{v})=0} V_{\Xi} \end{eqnarray*}



\item Let $\ell \in \mathbf{A}^{*}$, we define $W_{\ell}^{+}$, $W_{\ell}^{-}$,
$W_{\ell}^{0}$, as $W_{\vec{v}_ {\ell}}^{+}$, $W_{\vec{v}_ {\ell}}^{-}$, $W_{\vec{v}_
{\ell}}^{0}$ respectively, where $\vec{v}_{\ell} \in \mathbf{A}$ is the dual of $\ell$.

\item By \emph{the walls based at p}, we mean the set $p \bigcup_{\Xi} ker(\Xi)$.

\item By a \emph{geodesic segment through p}, we mean a set $p \overline{AB}$,
where $\overline{AB}$ is a directed line segment in $\mathbf{A}$.  By direction of a directed line segment in Euclidean space, we mean
a unit vector with respect to the usual Euclidean metric, and by direction of $p \overline{AB}$ we mean the direction of
$\overline{AB}$.

\item For $i=2,3,.. n-1$, by \emph{$i$-hyperplane through p}, we mean a set
$p S$, where $S \subset \mathbf{A}$ is an $i$-dimensional linear
subspace or an intersection between an $i$-dimensional linear subspace with a convex set.

\item Let $\pi_{A}:G= \mathbf{H}  \rtimes \mathbf{A} \longrightarrow \mathbf{A}$ be the projection
onto the $\mathbf{A}$ factor as $( \mathbf{x}, \mathbf{t} ) \mapsto \mathbf{t}$


\item For each root $\alpha_{i}$, define $\pi_{\alpha_{i}}: G \longrightarrow  V_{\alpha_{i}} \rtimes \langle \vec{v}_{\alpha_{i}} \rangle $ as
$(\mathbf{x}_{1}, \mathbf{x}_{2}, \cdots \mathbf{x}_{|\triangle|}) \mathbf{t} \mapsto (\mathbf{x}_{i},
\alpha_{i}(\mathbf{t})\vec{v}_{\alpha_{i}})$.  We refer to negatively curved spaces
$V_{\alpha_{i}} \rtimes \langle \vec{v}_{\alpha_{i}} \rangle$ or $V_{[\alpha]} \rtimes \langle \vec{v}_{\alpha} \rangle$
as weight (or root) hyperbolic spaces. \end{itemize} \smallskip

Now assume in addition that $G$ is unimodular.  Fix a net $\gothic{n}$ in $G$.  For $\alpha \in \triangle$, let $b(r) \subset V_{\alpha}$
be maximal product of intervals of size $r$, $[0,r]^{\mbox{dim}(V_{\alpha})}$, and since
$\mathbf{H}$ is the direct sum of those $V_{\alpha}$'s, we write $\prod_{\alpha \in \triangle}
b(r_{\alpha})$ for the product of those $b(r_{\alpha})$'s as $\alpha$ ranging over all
roots.  In other words, $\prod_{\alpha \in \triangle} b(r_{\alpha})$ is just product of
intervals in $\mathbf{H}$ where interval length is $r_{\alpha}$ in $V_{\alpha}$.
\smallskip

Let $\Omega \subset \mathbf{A}$ be a convex compact set with non-empty interior, e.g. a
product of intervals or a convex polyhedra.  Without loss of generality assume its barycenter
is the identity of $\mathbf{A}$.  We define the \emph{box associated to $\Omega$}, $\mathbf{B}(\Omega)$,
as the set $\left(  \prod_{j=1}^{|\triangle|} b(e^{\max(\alpha_{j}(\Omega))}) \right) \Omega$.

\begin{remark}
A box $\mathbf{B}(\Omega)$ as defined above is just a union of left translates of $\Omega \subset \mathbf{A}$ by a subset of
$\mathbf{H}$ (product of intervals) whose size is determined by $\Omega$.  The size of the intervals were chosen
so that a large proportion of points in the box $\mathbf{B}(\Omega)$ lie on a quadrilateral (see
Definition \ref{quadrilateral}).  In the definition above we have defined this subset of $\mathbf{H}$ as a
product of intervals, but this is just a choice of convenience so that it is simple to
describe the size of this subset in $\mathbf{H}$ in terms of $\Omega$.\end{remark}

Associate to the box $\mathbf{B}(\Omega)$, we use the following notations:

\begin{itemize}
\item $\mathcal{L}(\Omega)[m]$ (or $\mathcal{L}(\mathbf{B}(\Omega))[m]$) for the set of geodesics in $\mathbf{B}(\Omega)$
whose $\pi_{A}$ images begin and end at points of $\partial \Omega$
such that the ratio between its length and the diameter of $\Omega$
lies in the interval $[1/m, m]$.


\item For $i=2,3,\cdots ,n$, write $\mathcal{L}_{i}(\Omega)[m_{i}]$ (or $\mathcal{L}_{i}(\mathbf{B}(\Omega))[m_{i}]$) for the set of $i$
dimensional hyperplanes in $\mathbf{B}(\Omega)$ such that the ratio between its diameter and the diameter of $\Omega$ lies in the
interval $[1/m_{i}, m_{i}]$.


\item $\mathcal{P}(\Omega)$ (or $\mathcal{P}(\mathbf{B}(\Omega))$) for the set of points in $\mathbf{B}(\Omega)$.


\item  Let $S$ be an element of   $\bigcup_{i=2}^{n} \mathcal{L}_{i}(\Omega) \bigcup \mathcal{L}(\Omega) \bigcup
\mathcal{P}(\Omega)$.  We write $L(S)$, $L_{i}(S)$ for subset of $\mathcal{L}(\Omega)$, $\mathcal{L}_{i}(\Omega)$ contained or
containing $S$, and $P(S)$ for the subset of $\mathcal{P}(\Omega)$ contained in $S$. \end{itemize}

\begin{remark}
As we are interested in a given quasi-isometry $\phi: G \rightarrow G'$ which implicitly implies particular choices of
nets $\mathit{n} \subset G$, $\mathit{n'} \subset G'$, we will primarily consider $\phi$
as a map from $\mathit{n}$ to $\mathit{n'}$\symbolfootnote[3]{But then any two nets are bounded distance apart,
and a bounded modification does not change the quasi-isometry class of $\phi$, so whatever argument we make for $\mathit{n}$ and $\mathit{n'}$
are valid for other choices of nets as well.}.  Let $\hat{p}:G \rightarrow \mathit{n}$ that assigns $x \in G$, a closest net point.  In this way we tend to
think of a set $K \subset G$ not so much as a subset of the Lie group $G$, but as a subset of $\mathit{n}$ via the identification of $K$ and
$\hat{p}(K)$. \smallskip

In particular, the set of hyperplanes and points associated to a box as defined above
would be considered finite sets for us. \end{remark}


We now use boxes to produce a sequence of F\"{o}lner sets.
\begin{lemma} \label{boxes are folner}
Let $G=\mathbf{H} \rtimes \mathbf{A}$ be a non-degenerate, unimodular, split abelian-by-abelian Lie group.  Let $\Omega \subset \mathbf{A}$ be
compact convex with non-empty interior.  Then, $\mathbf{B}(r \Omega)$, $r \rightarrow \infty$ is a F\"{o}lner sequence. The volume ratio between $N_{\epsilon}(\partial (\mathbf{B}(r \Omega)))$
and $\mathbf{B}(r \Omega)$ is $O(\epsilon/ diam( \mathbf{B}(r \Omega))$\symbolfootnote[2]{because
the ratio of volumes of $\partial \Omega$ to $\Omega$ is roughly $\frac{1}{diam(\Omega)}$, and $r diam(\Omega)=diam(r
\Omega)$}.\end{lemma}

\begin{proof}
For each root $\alpha_{j}$, write $\alpha_{j}(\Omega)=[b_{j}, a_{j}]$. \smallskip   Since the sum of roots is zero,
the volume element is $\wedge_{j} d \mathbf{x}_{j} \wedge d \mathbf{t}$.
Therefore vol$(\mathbf{B}(r \Omega))=\left( \prod_{j} e^{ra_{j}} \right)r^{n} |\Omega|$. On the other hand, the area of the
boundary is  \begin{equation*}
\left| \partial \left( \prod_{j}[0,e^{ra_{j}}] (r \Omega) \right) \right|
= \underbrace{ \left| \partial \left( \prod_{j}[0,e^{ra_{j}}]
\right) (r \Omega) \right| }_{(1)} + \underbrace{ \left| \left(
\prod_{j}[0,e^{ra_{j}}] \right) \partial \left( r \Omega \right) \right|}_{(2)} \end{equation*}
\noindent We estimate the size of each term:
\begin{equation*}
(2): \left| \left( \prod_{j}[0,e^{ra_{j}}] \right)
\partial \left( r \Omega \right) \right| =
\left( \prod_{j} e^{ra_{j}} \right) r^{n-1} \left| \partial \Omega
\right| \end{equation*}

\begin{eqnarray*}
(1): \left| \partial \left( \prod_{j}[0,e^{ra_{j}}] \right) (r
\Omega) \right| & = & 2 \sum_{j}  \int_{\mathbf{t} \in r\Omega}
\underbrace{ \int_{ \mathbf{x}_{1} \in b(e^{ra_{1}}), \cdots
\mathbf{x}_{i} \in b^{e^{ra_{i}}} } e^{-\alpha_{1}(\mathbf{t})}
d \mathbf{x}_{1} \cdots e^{-\alpha_{i}(\mathbf{t})} d
\mathbf{x}_{i} \cdots }_{i \not= j}  d \mathbf{t} \\
&=& 2 \sum_{j}  \left( \prod_{i \not= j} e^{ra_{i}}
\int_{\mathbf{t} \in r\Omega} e^{\alpha_{j}(\mathbf{t})}
d\mathbf{t} \right)  \\
& \leq &  2 \sum_{j}  \left( \prod_{i \not= j}
e^{ra_{i}} (e^{ra_{j}} - e^{rb_{j}}) \left| Proj_{ker(\alpha_{j})}(r \Omega) \right| \right) \\
&=& 2 \left( \prod_{i} e^{ra_{i}} \right) r^{n-1} \left( \sum_{j}\left| Proj_{ker(\alpha_{j})} (\Omega) \right|( 1-e^{-(ra_{j}-rb_{j})}) \right)\\
& \leq& 2 |\triangle| \left( \prod_{i} e^{ra_{i}} \right) r^{n-1} \max_{j} \left| Proj_{ker(\alpha_{j})}(\Omega) \right| \end{eqnarray*} \end{proof}

\begin{remark} \label{general folner}
The same calculation as above shows that for any set $\tilde{B}$ of the form $\Lambda \rtimes \Omega$,
where $\Lambda \subset \mathbf{H}$, $\Omega \subset \mathbf{A}$, the ratio of volumes of
$N_{\epsilon}(\partial \tilde{B})$ and that of $\tilde{B}$ is $O(\epsilon/
diam(\tilde{B}))$.   \end{remark}

\section{Quasi-geodesics}
The purpose of this section is to prove

\begin{theorem} \label{existence of good boxes}
Let $G, G'$ be non-degenerate, unimodular, split abelian-by-abelian Lie groups, and $\phi: G \rightarrow G'$ a $(\kappa, C)$ quasi-isometry.
Given $0< \delta, \eta< \tilde{\eta} <1 $, there are numbers $L_{0}$, $m> 1$ and $0 < \rho < 1$ depending on
$\delta, \eta$, $\kappa, C$ with the following properties: \smallskip


If $\Omega \subset \mathbf{A}$ is a product of intervals of equal size at least $mL_{0}$,
then a tiling of $\mathbf{B}(\Omega)$ by isometric copies of $\mathbf{B}(\rho \Omega)$
\[ \mathbf{B}(\Omega)=\bigsqcup_{j \in \mathbf{J}} \mathbf{B}(\Omega_{j}) \sqcup \Upsilon  \]

\noindent contains a subset $\mathbf{J}_{0}$ whose measure is at least $1-\vartheta$
times that of $\mathbf{J}$ such that:

\begin{enumerate}
\item For all $j \in \mathbf{J}_{0}$, there is a subset $\mathcal{L}^{0}_{j} \subset
\mathcal{L}(\Omega_{j})[m]$, whose measure is at least $1-\varkappa$ times that of $\mathcal{L}(\Omega_{j})$

\item If $\zeta \in \mathcal{L}^{0}_{j}$, then $\phi(\zeta)$ is within $\eta$-linear neighborhood of a geodesic
segment which makes an angle at least $\sin^{-1}(\tilde{\eta})$ with root kernels. \end{enumerate}

Here $\vartheta$, $\varkappa$ approach zero as $\tilde{\eta} \rightarrow 0$.  The measure of set $\Upsilon$ is at most $\delta'$ proportion of
measure of $\mathbf{B}(\Omega)$, where $\delta'$ depends on $\delta$ and goes to zero as the latter approaches zero. \end{theorem}

\subsection{Some facts about non-degenerate, split abelian-by-abelian groups }
In this subsection, $G$ denotes for a non-degenerate, split abelian-by-abelian group.  By Lemma \ref{QI embedding}, we can use the
embedded metric on $G$.  We will use the metric property of those $H_{s+1}$ spaces to obtain the following proposition, which basically says
that if a quasi-geodesic in $G$ is long, then its projection in $\mathbf{A}$ has to be long as well.

\begin{proposition} \label{can't move far in R2R3} 
Let $\zeta:[0,L] \rightarrow G$ be a $(\kappa, C)$ quasi-geodesic segment.
Suppose $\{ \pi_{A}(\zeta(t)) \}$ lies in a ball of diameter $s$. Then for
any $p,q \in \zeta$, $d(p,q) \leq \hbar s$, where $\hbar$ is a
constant that depends only on the number of roots. \end{proposition}

\begin{corollary} \label{can't move far}(assumptions as in Proposition \ref{can't move far in R2R3})
If there are two points $p,q$ on $\zeta$ such that $d(p,q) > \hbar
s$, then there must be a point $r \in [\zeta^{-1}(p),\zeta^{-1}(q)]$
such that $d(\pi_{A}(p),\pi_{A}(\zeta(r)))>s$. \end{corollary}

To prove Proposition \ref{can't move far in R2R3}, we need the following two lemmas whose verifications can be found in the
Appendix.\\



\noindent In $\mathit{H}_{n'+1}=\mathbb{R}^{n'} \rtimes \mathbb{R}$, we write $h$ for the projection onto the $\mathbb{R}$ factor.

\begin{lemma} \label{can't move far in H2}
Let $\eta:[a,b] \rightarrow \mathit{H}_{n'+1}$ be a continuous path such that \begin{itemize}
\item The image of $h \circ \eta$ is contained in an interval of length no bigger than
$s$, where $s > \kappa(C_{\mathit{H}_{n'+1}})^{2} ( > 2) $.  Here $C_{\mathit{H}_{n'+1}}$ is a constant depending only on
$\mathit{H}_{n'+1}$ (as in equation (\ref{property of U})).

\item whenever $i_{1} \leq i_{2} \leq ... i_{n} \in [a,b]$,
\begin{equation*} \frac{\sum_{j} d(\eta(i_{j}), \eta(i_{j+1}))}{d(\eta(i_{1}),\eta(i_{n}))} \leq 2 \kappa \end{equation*}  \end{itemize}
\noindent Then, for any two points $p,q \in \eta([a,b])$, $d(p,q)\leq \hat{C}(2\kappa) s$, where $\hat{C}$ depends only on
$C_{\mathit{H}_{n'+1}}$. \end{lemma} \begin{proof} see Appendix \end{proof}

\begin{lemma} \label{mixing} Let $a,b \geq 0$, $A,B >0$.  Suppose $\frac{a+b}{A+B}= c_{\alpha}\frac{a}{A} + c_{\beta}\frac{b}{B}$,
with $c_{\alpha}+ c_{\beta}=1$.  Suppose $c_{\alpha} \geq c_{\beta}$, then $A \geq B$.
\end{lemma} \begin{proof} see Appendix \end{proof}

\begin{proof} \textit{of Proposition} \ref{can't move far in R2R3}


We proceed by induction on the number of roots.  The base step where there is just one
root is Lemma \ref{can't move far in H2}.  Since $\zeta$ is a $(\kappa,c)$ quasi-geodesic, for any $i_{0} \leq i_{1} \leq i_{2} \leq i_{3} .... i_{n} \in [0,L]$, we must have
\begin{equation} \label{weak efficient}
\frac{\sum_{j} d(\zeta(i_{j}),\zeta(i_{j+1}))}{d(\zeta(i_{0}),\zeta(i_{n}))} \leq 2 \kappa \end{equation}

\noindent We recall from Lemma \ref{QI embedding} that $d(\cdot,\cdot)= \sum_{l=1}^{|\triangle|}d^{\alpha_{l}}(\pi_{\alpha_{l}}(\cdot),\pi_{\alpha_{l}}(\cdot))$,
and proceed to simplify equation(\ref{weak efficient}) by writing $d^{\alpha_{l}}(\pi_{\alpha_{l}}(\zeta(i_{j})),\pi_{\alpha_{l}} \zeta(i_{j+1}))$ as $d^{\alpha_{l}}_{j}$, and
$d^{\alpha_{l}}(\pi_{\alpha_{l}}(\zeta(i_{0})),\pi_{\alpha_{l}}(\zeta(i_{n})))$ as $d^{\alpha_{l}}$.

Now equation (\ref{weak efficient}) becomes

\begin{equation*}
\frac{\sum_{j} \left( d^{\alpha_{1}}_{j} + d^{\alpha_{2}}_{j} +
\cdots d^{\alpha_{|\triangle|}}_{j} \right) }{ d^{\alpha_{1}} +
d^{\alpha_{2}}+ \cdots d^{\alpha_{|\triangle|}} } \leq 2 \kappa
\end{equation*}

\begin{itemize}
\item Suppose for some weight, let's say $\alpha_{1}$, we have

\begin{equation} \label{alpha big}
\frac{\sum_{j}d^{\alpha_{1}}_{j} + \sum_{j} ( d^{\alpha_{2}}_{j}+ d^{\alpha_{3}}_{j} \cdots + d^{\alpha_{|\triangle|}}_{j} ) }
{d^{\alpha_{1}} + ( d^{\alpha_{2}} + d^{\alpha_{3}}+ \cdots + d^{\alpha_{|\triangle|}}) }= c_{\alpha}
\frac{\sum_{j}d^{\alpha_{1}}_{j}}{d^{\alpha_{1}}} + c_{\beta}
\frac{\sum_{j} \left( d^{\alpha_{2}}_{j}+ d^{\alpha_{3}}_{j} \cdots
+ d^{\alpha_{|\triangle|}}_{j} \right) }{ \left(  d^{\alpha_{2}} +
d^{\alpha_{3}}+ \cdots + d^{\alpha_{|\triangle|}} \right)}
\end{equation}

\noindent with $c_{\alpha}+c_{\beta}=1$, and $c_{\alpha} \geq c_{\beta}$.  Therefore
\begin{itemize}
\item $c_{\alpha} \geq 1/2$.  Since equation (\ref{alpha big}) is bounded above by $2\kappa$, we now have an upper bound for
the first term:

\begin{equation*} \frac{1}{2} \frac{\sum_{j}d^{\alpha_{1}}_{j}}{d^{\alpha_{1}}} \leq 2 \kappa \end{equation*}
That is,  $\{ \pi_{\alpha}(\zeta(i_{j})) \}$ are points whose heights in the $\alpha$ weight hyperbolic space lie in an interval
of width no bigger than $s$ (because $\pi_{A}(\zeta(i_{j}))$ lies in a ball of diameter $s$), and
\begin{equation*}
\frac{\sum_{j} d^{\alpha_{1}}(\pi_{\alpha_{1}}(\zeta(i_{j})),\pi_{\alpha_{1}}(\zeta(i_{j+1})))}
{d^{\alpha_{1}}(\pi_{\alpha}(\zeta(i_{0})),\pi_{\alpha}(\zeta(i_{n})))} \leq 4 \kappa \end{equation*}

By Lemma \ref{can't move far in H2}, $d^{\alpha}(\pi_{\alpha}(\zeta(i_{0})),\pi_{\alpha}(\zeta(i_{n}))) \leq \hat{C} (4\kappa)s$
\smallskip

\item Since $c_{\alpha} \geq c_{\beta}$, Lemma \ref{mixing} says $d^{\alpha_{1}} \geq \sum_{l=2}^{|\triangle|} d^{\alpha_{l}}$, which
makes $d(\zeta(i_{0}),\zeta(i_{n})) =d^{\alpha_{1}} + \sum_{l=2}^{|\triangle|} d^{\alpha_{l}} \leq 2 \hat{C} (4\kappa)s = 2^{2} \hat{C} (2\kappa) s$
\end{itemize}

\item If the first possibility doesn't occur, then for every weight
$\alpha_{i'}$, we must have
\begin{equation}\label{beta is big}
\frac{\sum_{j}(d^{\alpha_{1}}_{j}+ d^{\alpha_{2}}_{j} \cdots d^{\alpha^{|\triangle|}}_{j}) }{d^{\alpha_{1}} + d^{\alpha_{2}} + \cdots
d^{\alpha_{|\triangle|}}}= c_{\alpha_{i'}} \frac{\sum_{j}d^{\alpha_{i'}}_{j}}{d^{\alpha_{i'}}} + c_{\beta_{i'}}
\frac{\sum_{j} \left( d^{\alpha_{1}}_{j} + d^{\alpha_{2}}_{j} \cdots d^{\alpha_{i'-1}}_{j} + d^{\alpha_{i'+1}}_{j} \cdots
d^{\alpha_{|\triangle|}}_{j} \right) }{\sum_{l \not= i'} d^{\alpha_{l}} } \end{equation}

\noindent with $c_{\alpha_{i'}}, c_{\beta_{i'}} \geq 0$, $c_{\alpha_{i'}}+c_{\beta_{i'}}=1$,
BUT $c_{\alpha_{i'}} \leq c_{\beta_{i'}}$.  We fix such an $i'$.  Then
\begin{itemize}
\item $c_{\beta_{i'}} \geq 1/2$.  Since the equation (\ref{beta is big}) is bounded above
by $2\kappa$, we obtain an upper bound for the second term on the
right hand side:
\begin{equation*}
\frac{\sum_{j} (d^{\alpha_{1}}_{j} + d^{\alpha_{2}}_{j} \cdots
d^{\alpha_{i'-1}}_{j} + d^{\alpha_{i'+1}}_{j} \cdots
d^{\alpha_{|\triangle|}}_{j}) }{\sum_{l \not= i'} d^{\alpha_{l}}} \leq 4
\kappa \end{equation*}

By inductive hypothesis, \begin{equation*} \sum_{l \not= i'}
d^{\alpha_{l}}(\pi_{\alpha_{l}}(\zeta(i_{0})), \pi_{\alpha_{l}}(\zeta(i_{n}))) =\sum_{l \not= i'}
d^{\alpha_{l}} \leq 2^{2(|\triangle|-2)} \hat{C} (4 \kappa) s \end{equation*}

\item Finally, since $c_{\alpha_{i'}} \leq c_{\beta_{i'}}$, Lemma \ref{mixing} says
$d^{\alpha_{i'}} \leq \sum_{l \not= i'} d^{\alpha_{l}}$ which means $d(\zeta(i_{0}),\zeta(i_{n}))=d^{\alpha_{i'}}+ \sum{l \not= i'}
d^{\alpha_{l}} \leq 2 2^{2(|\triangle|-2)} \hat{C} (4\kappa)s = 2^{2(|\triangle|-1)} \hat{C} (2\kappa)s$  \end{itemize}
\end{itemize} \end{proof}

\subsection{Efficient scale}
This subsection is based on definition 4.5 and lemma 4.6 in \cite{EFW0},
where $\epsilon$-efficiency was defined.  Here we note the consequence of an efficient segment in
a non-degenerate, split abelian-by-abelian group.

\begin{definition} \emph{($\epsilon$-efficient at scale $\tilde{r}$)}
Let $Y$ be a metric space, and $\lambda:[0,L] \rightarrow Y$ a rectifiable curve. We say that $\lambda$ is $\epsilon$-efficient at scale $\tilde{r}$,
$0< \tilde{r} \leq 1$ if   \begin{equation*}
\sum_{j} d(p_{j}, p_{j+1}) \leq (1+ \epsilon) d(\lambda(0),\lambda(L)), \mbox{  where   } \{p_{j}\}
=\hat{\mathcal{S}}(\lambda,\tilde{r}d(\lambda(0),\lambda(L)))  \end{equation*} \end{definition}

\begin{remark} Note that being efficient at scale $r$ does necessarily not imply efficient at all sales $\tilde{r} < r$. \end{remark}

Efficiency provides with us the closest description of being `straight' in $\mathbb{R}^{n}$, whose meaning is made precise by the following lemma.

\begin{lemma} \label{efficient in Rn}
If $\lambda:[a,b] \rightarrow \mathbb{R}^{n}$ is $\epsilon$-efficient at scale $r$, then $d_{H}(\lambda,
\overline{\lambda(a)\lambda(b)}) \leq  (r+1.5\epsilon^{1/4}) d(\lambda(a),\lambda(b))$ \end{lemma}

\begin{proof}
Let $m=r d(\lambda(a), \lambda(b))$, and $\{p_{j}\}_{j=0}^{N}=\mathcal{S}(\lambda,m)$ so
that $d(p_{0},p_{N})=d(\lambda(a),\lambda(b))=L$.  Let $h_{\overline{p_{0}p_{N}}}$ be the orthogonal projection of
$\lambda$ onto $\overline{p_{0}p_{N}}$, $\tilde{p}_{i} =h_{\overline{pq}}(p_{i})$, so $d(\tilde{p}_{j},\tilde{p}_{j+1})
\leq d(p_{j},p_{j+1})=m$. Since $\tilde{p}_{0}=p_{0}$, $\tilde{p}_{N}=p_{N}$, $\bigcup_{i=0}^{N-1}
\overline{\tilde{p}_{i}\tilde{p}_{i+1}}=\overline{p_{0}p_{N}}$, and Lemma \ref{triangle in Rn} in the Appendix gives that $d(p_{j},\tilde{p}_{j}) \leq
1.5 \epsilon^{1/4}L$. So if $\dot{p} \in \lambda$, let $p_{j}$ be the closest point in $\mathcal{S}(\lambda,m)$, we then have
$d(\dot{p}, \overline{p_{0}p_{N}}) \leq d(\dot{p},p_{j}) + d(p_{j},\overline{p_{0}p_{N}}) \leq m + 1.5 \epsilon^{1/4}L$.
Similarly for $\ddot{p} \in \overline{p_{0}p_{N}}$, there is a $j$ such that $\ddot{p} \in \overline{\tilde{p}_{j}\tilde{p}_{j+1}}$,
with $d(\ddot{p},\tilde{p}_{j}) \leq d(\ddot{p},\tilde{p}_{j+1}$, then $d(\ddot{p},\lambda) \leq d(\ddot{p},\tilde{p}_{j}) +
d(\tilde{p}_{j},\lambda) \leq \frac{1}{2}m + 1.5 \epsilon^{1/4}L$. \end{proof}

The purpose of this subsection is to prove the following lemma which roughly says that given a $\epsilon$, if
a path is sufficiently long, then it is $\epsilon$-efficient on some scale.

\begin{lemma} \label{efficiency scale}
Let $G$ be a non-degenerate, split abelian-by-abelian group.  Take any $N \gg 2$, $L_{stop} \geq (2\kappa)C$,
$0 <\epsilon < 1$.  If $\tilde{\lambda}:[0,L] \rightarrow G$ is $(\kappa, C)$ quasi-geodesic
satisfying\footnote{this long expressions really just says that $L$ has to be
sufficiently big with respect to given $L_{stop}$, $\epsilon$ and $N$.}
\[ \frac{L_{stop}}{\left( \frac{1}{2}\epsilon^{1/4} \right)^{\frac{ \hbar (2\kappa)^{2} N +
\epsilon}{\epsilon}}}  \leq 2\kappa L  \] \noindent then there is a
scale $0< \rho_{J} \leq 1$ such that
\begin{equation*}
\frac{| \mathcal{S}(\lambda, \rho_{J}|\lambda|, \mbox{ not }
\epsilon \mbox{efficient at scale } \frac{1}{2}\epsilon^{1/4}) |}{|
\mathcal{S}(\lambda, \rho_{J}|\lambda|) |} \leq \frac{1}{N}
\end{equation*} \noindent where $\lambda=\pi_{A}(\tilde{\lambda})$, and $\frac{1}{2}\epsilon^{1/4}
\rho_{J}|\lambda| \geq L_{stop}$. \end{lemma}

\begin{proof}
The idea of the proof is as follows: if a segment is not efficient, then by subdividing and adding up the distance between consecutive
pairs of points in the subdivision, the sum exceeds the distance between end points of the original segment by a fixed proportion.  In other words, lack of
efficiency increases length. However this cannot happen at every scale (bigger than $\frac{C}{d(\lambda(0), \lambda(L))}$, where $C$
is the additive constant of the quasi-geodesics), because to every subdivision, the sum of distance between successive pairs of points is bounded above
by the length of the curve.  We now proceed with the proof. \medskip

First note that the condition on $L$ in relation to $\epsilon$, $L$
and $N$ is the same as

\begin{equation} \label{Lstop}
\frac{\ln(L_{stop})-\ln(2\kappa L)} {\ln \left(\frac{\epsilon^{1/4}}{2} \right)} -1 \geq \frac{\hbar (2\kappa)^{2}}{\epsilon}N \end{equation}

If $\lambda$ $\epsilon$-efficient at scale $\frac{1}{2} \epsilon^{1/4} |\lambda|$, we can take $r_{J}=|\lambda|$,
$\rho_{J}=\frac{r_{J}}{|\lambda|}=1$ and we are done. Otherwise, let $\{\tilde{p}^{0}_{j}\}_{j=0}^{n_{0}} \subset \{\tilde{p}^{1}_{j}
\}_{j=0}^{n_{1}} \subset \{\tilde{p}^{2}_{j} \}_{j=0}^{n_{2}} \cdots \subset \{ \tilde{p}^{D}_{j} \}_{j=0}^{n_{D}}$ be an increasing sets
of points on $\tilde{\lambda}$ such that

\begin{enumerate}
\item $r_{0}=\frac{1}{2}\epsilon^{1/4}|\lambda|$,  $r_{b}=\frac{1}{2} \epsilon^{1/4} r_{b-1}$, $r_{D}=L_{stop}$
\item $\{p^{b}_{j} \}=\hat{\mathcal{S}}(\lambda,r_{b})$, where $p^{b}_{j}=\pi_{A}(\tilde{p}^{b}_{j})$  \end{enumerate} \medskip

We note here that for each $b$ between $0$ and $D$, $\lambda_{[p^{b}_{j}, p^{b}_{j+1}]}$ lies in a ball of diameter $r_{b}$, we must have
$d(\tilde{p}^{b}_{j}, \tilde{ p}^{b}_{j+1}) \leq \bar{h} r_{b}$ by Proposition \ref{can't move far in R2R3}. Therefore

\begin{equation*}
\left| \tilde{\lambda}^{-1}(\tilde{p}^{b}_{j}) - \tilde{\lambda}^{-1}(\tilde{p}^{b}_{j+1}) \right| \leq (\hbar r_{b}) (2 \kappa) \end{equation*}

\noindent and

\begin{equation*}
\left|  \{ [\tilde{\lambda}^{-1}(\tilde{p}^{b}_{j}), \tilde{\lambda}^{-1}(\tilde{p}^{b}_{j+1}) ] \} \right| \geq \frac{L}{( \hbar r_{b})(2\kappa)} \end{equation*}

\noindent Thus if we denote $\sum_{j=0}^{n_{b}-1} d(p^{b}_{j},p^{b}_{j+1})$ by $L_{b}$,

\begin{equation}\label{lower bound}
\frac{L_{b}}{|\lambda|}=\frac{\sum_{j} d(p^{b}_{j},p^{b}_{j+1})}{|\lambda|} \geq \frac{L}{( \hbar
r_{b})(2\kappa)} r_{b} \frac{1}{|\tilde{\lambda}|} \geq \frac{L}{\hbar 2 \kappa} \frac{1}{2 \kappa L} = \frac{1}{\hbar
(2\kappa)^{2}}=\hat{c} \end{equation}  \noindent which we note is a lower bound that depends only on $\kappa$ and the group $G$.
\medskip

Let $E$ be an integer between $0$ and $D$. By construction, for any
$p^{E}_{j}$, $p^{E}_{j+1}$, there are $s_{1}$, $s_{2}$ such that
$p^{E+1}_{s1}=p^{E}_{j}$, $p^{E+1}_{s2}=p^{E}_{j+1}$. Then
\begin{equation*} \sum_{i=s1}^{s2-1} d(p^{E+1}_{i},p^{E+1}_{i+1})
\geq d(p^{E}_{j},p^{E}_{j+1}) \end{equation*}

If however the segment of $\lambda_{[p^{E}_{j},p^{E}_{j+1}]}$ is not
$\epsilon$-efficient on scale $\frac{1}{2}\epsilon^{1/4}$, then
\begin{equation*}
\sum_{i=s1}^{s2-1} d(p^{E+1}_{i},p^{E+1}_{i+1}) \geq (1+ \epsilon)
d(p^{E}_{j},p^{E}_{j+1}) = d(p^{E}_{j},p^{E}_{j+1}) + \epsilon
d(p^{E}_{j},p^{E}_{j+1}) \end{equation*} This means
\begin{equation*}
\sum_{i=0}^{n_{E+1}-1} d(p^{E+1}_{i},p^{E+1}_{i+1}) \geq
\sum_{j=0}^{n_{E}-1} d(p^{E}_{j},p^{E}_{j+1}) + \epsilon \sum_{l \in
B_{E}} d(p^{E}_{l},p^{E}_{l+1}) \end{equation*} where $B_{E}$ are
those integer $j$ between $0$ and $n_{E}-1$ such that
$\lambda_{[p^{E}_{j},p^{E}_{j+1}]}$ is not $\epsilon$-efficient on
scale $\frac{1}{2}\epsilon^{1/4}$. Denote $\sum_{l \in B_{E}}
d(p^{E}_{l},p^{E}_{l+1})$ by $\Omega_{E}$ the above says
\begin{equation*}
L_{E+1} \geq L_{E} + (\epsilon) \Omega_{E}
\end{equation*}

Hence
\begin{eqnarray*}
|\lambda| \geq L_{D} &\geq& L_{D-1} +(\epsilon) \Omega_{D-1} \\
& \geq&  L_{D-2} +(\epsilon)\Omega_{D-2} + (\epsilon) \Omega_{D-1}\\
& \geq& L_{D-3} + (\epsilon)\Omega_{D-3} +
(\epsilon)\Omega_{D-2} + (\epsilon) \Omega_{D-1} \\
& \geq& \cdots \\
&\geq& L_{0} + (\epsilon)\sum_{j=1}^{D} \Omega_{D-j} \geq
(\epsilon)\sum_{j=1}^{D} \Omega_{D-j}
\end{eqnarray*}


Dividing both sides by $|\lambda|$, and let $\delta_{i}(\lambda)=
\frac{\Omega_{i}}{L_{i}}$ be the proportion of elements in $S(\lambda, r_{j})$ that are
not $\epsilon$-efficient at scale $1/2 \epsilon^{1/4}$, we have by equation (\ref{lower bound})
\begin{equation}\label{efficient scale sum}
\frac{1}{\epsilon} \geq \sum_{i=0}^{D-1} \frac{\Omega_{i}}{|\lambda|} = \sum_{i=0}^{D-1}
\frac{\Omega_{i}}{L_{i}} \frac{L_{i}}{|\lambda|} \geq \hat{c} \sum_{i=0}^{D-1} \delta_{i}(\lambda)  \end{equation}

\noindent Since we stop at $r_{D}=L_{stop}$,
\begin{eqnarray*}
\left( \frac{1}{2} \epsilon^{1/4} \right)^{D+1} |\lambda| & =& L_{stop} \\
D = \frac{\ln(L_{stop}) - \ln(|\lambda|)}{\ln \left( \frac{1}{2}\epsilon^{1/4} \right)}-1 & \geq & \frac{\ln(L_{stop}) -
\ln(2\kappa L)}{\ln \left( \frac{1}{2}\epsilon^{1/4}\right)} \geq \frac{1}{\epsilon \hat{c} }N \end{eqnarray*}

\noindent where we used equations (\ref{Lstop}) and (\ref{lower bound}) in the last inequality. \smallskip

\noindent The right hand side of (\ref{efficient scale sum}) has at least $\frac{1}{\epsilon  \hat{c} }N$ terms, so for some $0 \leq J
\leq D$, $\delta_{J}(\lambda) \leq \frac{1}{N}$, which means the proportion of segments in $\mathcal{S}(\lambda, r_{J})$ that are not
$\epsilon$-efficient at scale $\frac{1}{2} \epsilon^{1/4}$ is at most $\frac{1}{N}$. The desired
$\rho_{J}=\frac{r_{J}}{|\lambda|}=\left(\frac{1}{2}\epsilon^{1/4} \right)^{J+1}$ \end{proof}

\begin{corollary} \label{efficiency scale multiple}
Let $G$ be a non-degenerate, split abelian-by-abelian group.  Take any $2 \ll N_{0} < N$, $L_{stop} \geq (2\kappa)C$,
$0 <\epsilon < 1$, and let $\mathcal{F}=\{ \tilde{\lambda}_{i} \}$ be a finite set of
$(\kappa, C)$ quasi-geodesics.  If every element of $\mathcal{F}$, $\tilde{\lambda}_{i}:[0,L_{i}] \rightarrow G$ satisfies
\[ \frac{L_{stop}}{\left( \frac{1}{2}\epsilon^{1/4} \right)^{\frac{ \hbar (2\kappa)^{2} N +
\epsilon}{\epsilon}}}  \leq 2\kappa L_{i}  \]
\noindent then there is a scale $0< \rho_{J} \leq 1$ and a subset $\mathcal{F}_{0}$ such
that  \begin{enumerate}
\item $|\mathcal{F}_{0}| \geq (1-\frac{N_{0}}{N}) |\mathcal{F}|$

\item for every $\tilde{\lambda}_{i} \in \mathcal{F}_{0}$,
\begin{equation*}
\frac{| \mathcal{S}(\lambda_{i}, \rho_{J}|\lambda_{i}|, \mbox{ not } \epsilon \mbox{efficient at scale }
\frac{1}{2}\epsilon^{1/4}) |}{| \mathcal{S}(\lambda_{i}, \rho_{J}|\lambda_{i}|) |} \leq \frac{1}{N_{0}} \end{equation*}
\noindent where $\lambda_{i}=\pi_{A}(\tilde{\lambda}_{i})$, and
$\frac{1}{2}\epsilon^{1/4} \rho_{J}|\lambda_{i}| \geq L_{stop}$.  \end{enumerate} \end{corollary}

\begin{proof}
We apply Lemma \ref{efficiency scale} to each element of $\mathcal{F}$ and stop at
equation (\ref{efficient scale sum}).  That is, for every $\tilde{\lambda}_{j} \in
\mathcal{F}$, we have
\begin{equation*}
\frac{1}{\epsilon} \geq \hat{c} \sum_{i=0}^{D-1} \delta_{i}(\lambda_{j})  \end{equation*} \noindent therefore
\begin{equation}\label{efficien scale sum multiple}
\frac{1}{\epsilon} = \frac{1}{|\mathcal{F}|} \sum_{\tilde{\lambda}_{j} \in \mathcal{F}}
\frac{1}{\epsilon} \geq  \frac{\hat{c}}{|\mathcal{F}|} \sum_{\tilde{\lambda}_{j} \in \mathcal{F}} \sum_{i=0}^{D-1}
\delta_{i}(\lambda_{j}) = \hat{c} \sum_{i=0}^{D-1}  \frac{1}{|\mathcal{F}|} \sum_{\tilde{\lambda}_{j} \in \mathcal{F}}
\delta_{i}(\lambda_{j})  \end{equation}

\noindent For the same reason as in Lemma \ref{efficiency scale}, the right hand side of
equation (\ref{efficien scale sum multiple}) has at least $\frac{1}{\epsilon  \hat{c} }N$ terms,
so for some $0 \leq J \leq D$

\[ \frac{1}{N}  \geq \frac{1}{|\mathcal{F}|} \sum_{\tilde{\lambda}_{j} \in \mathcal{F}} \delta_{J}(\lambda_{j})  \]
\noindent Let $\mathcal{F}_{b}$ be those $\tilde{\lambda}_{j} \in \mathcal{F}$ whose
$\delta_{J}$ value is more than $\frac{1}{N_{0}}$.  Applying Chebyshev inequality we see
that
\[ \frac{1}{N} \geq \left| \mathcal{F}_{b} \right| \frac{1}{N_{0}} \frac{1}{|\mathcal{F}|}  \]

\noindent the claim is obtained by setting $\mathcal{F}_{0}$ as the complement of $\mathcal{F}_{b}$.  \end{proof}

The following lemma says that given an efficient segment, most subsegments of length
sufficiently larger than the efficient scale are efficient.

\begin{lemma}\label{lots of efficient segments}
Let $\lambda$ be a rectifiable curve in a metric space Y whose end
points are $L$ apart. Suppose $\lambda$ is $\epsilon$-efficient at
scale $\frac{1}{2} \epsilon^{1/4}$. Let $\{q_{i}\}$ be a subdivision
of $\lambda$ such that for some $r_{s}, r_{b} \in [\epsilon^{1/4}L,
L]$, the distance between successive subdivision points satisfies
$r_{s} \leq d_{Y}(q_{i},q_{i+1}) \leq r_{b}$. Then provided
$\frac{r_{b}}{r_{s}} \epsilon^{1/2} \ll 1$, at least $\epsilon^{1/2}
\frac{r_{b}}{r_{s}}$ proportion of the subsegments $\{
\lambda_{[q_{i},q_{i+1}]} \}$ are $\epsilon^{1/2}$ efficient at
scale $\frac{1}{2} \epsilon^{1/4}$. \end{lemma}

\begin{proof}
Let $\{ p_{j} \}=\bigcup_{i}
\hat{\mathcal{S}}(\lambda_{[q_{i},q_{i+1}]},\frac{1}{2}\epsilon^{1/4}d(q_{i},q_{i+1}))$.
Then $\{ q_{j}\} \subset \{p_{j} \}$. Write $q_{0}=p_{n_{0}}$,
$q_{1}=p_{n_{0}+n_{1}}$, $q_{2}=p_{n_{0}+n_{1}+n_{2}}$ etc. With
this notation we can write

\begin{equation*}
\sum_{i=0}^{N-1} d(p_{i},p_{i+1})= \sum_{j} \sum_{i=n_{0}+n_{1}
\cdots + n_{j}}^{n_{0}+n_{1} \cdots +n_{j+1} -1} d(p_{i},p_{i+1})
\end{equation*}

For each $j$, write $q_{j}=p_{s1}$, $q_{j+1}=p_{s2}$.  Then \textbf{EITHER}
\begin{equation*}
d(q_{j}, q_{j+1}) \leq \sum_{i=0}^{s2-1} d(p_{i},p_{i+1}) \leq (1+
\epsilon^{1/2}) d(p_{s1},p_{s2})=(1+ \epsilon^{1/2})
d(q_{j},q_{j+1}) \end{equation*} \textbf{OR}
\begin{equation*} \sum_{i=s1}^{s2-1} d(p_{i},p_{i+1}) > (1+ \epsilon^{1/2})
d(p_{s1},p_{s2})= (1+ \epsilon^{1/2}) d(q_{j},q_{j+1})
\end{equation*} in which case we denote the set of all such $q_{j}$'s as
$\mathcal{B}$. Note that the cardinality of the coarser division
points $|\{ q_{i} \}| \geq \frac{L}{r_{b}}$.

$\lambda$ being efficient means
\begin{equation*}
\epsilon d(p_{0},p_{N}) \geq  \left( \sum_{j} \sum_{i=n_{0}+n_{1} \cdots + n_{j}}^{n_{0}+n_{1}
\cdots +n_{j+1} -1} d(p_{i},p_{i+1}) \right) - d(p_{0},p_{N})  \end{equation*} Hence

\begin{eqnarray*}
\epsilon L= \epsilon d(p_{0},p_{N})  &\geq &  \left( \sum_{j} \sum_{i=n_{0}+n_{1} \cdots + n_{j}}^{n_{0}+n_{1}
\cdots +n_{j+1} -1} d(p_{i},p_{i+1}) \right) - d(p_{0},p_{N}) \\
&\geq & \sum_{j} \left( \sum_{i=n_{0}+n_{1} \cdots + n_{j}}^{n_{0}+n_{1} \cdots +n_{j+1} -1} d(p_{i},p_{i+1}) - d(q_{j},q_{j+1}) \right) \\
& \geq& \sum_{q_{j} \in \mathcal{B}} \epsilon^{1/2} d(q_{j},q_{j+1}) \geq \epsilon^{1/2} |\mathcal{B}|r_{s} \end{eqnarray*}
\noindent therefore $ | \mathcal{B} | \leq \epsilon^{1/2} \frac{L}{r_{s}} $, giving us a bound on $\flat_{r}$, the
proportion of $\mathcal{B}$, as \begin{equation*}
\flat_{r}=\frac{|\mathcal{B}|}{ |\{q_{j} \}|} \leq \frac{\epsilon^{1/2}\frac{L}{r_{s}}}{\frac{L}{r_{b}}}
=\epsilon^{1/2}\frac{r_{b}}{r_{s}}\end{equation*} \end{proof}

\subsection{Monotone scale}
\begin{definition} \emph{($\delta$-monotone)}
Let $G$ be a split abelian-by-abelian group, and $\zeta:[0,L] \rightarrow G$ a $(\kappa,C)$ quasi-geodesic segment such that there exists a line segment
$\overline{AB} \in \mathbf{A}$ satisfying $d_{H}(\pi_{A}(\zeta), \overline{AB}) \leq \epsilon |\pi_{A}(\zeta)|$,
for some $0 \leq \epsilon < 1$.  Let $h_{\overline{AB}}:\pi_{A}(\zeta)
\rightarrow \overline{AB}$ be the map that sends every point of $\pi_{A}(\zeta)$ to the closest point on $\overline{AB}$ by
orthogonal projection.  We say that $\zeta$ is  \begin{itemize}
\item $\delta$-monotone, if $ 1 > \delta \gg 2\hbar\epsilon$ and
\[ h_{\overline{AB}}(\pi_{A} \circ \zeta(t_{1}))=h_{\overline{AB}}(\pi_{A} \circ \zeta(t_{2}))
\Longrightarrow d(\zeta(t_{1}),\zeta(t_{2})) \leq \delta d(\zeta(0), \zeta(L)) \]

\item $(\nu,C_{1})$ weakly monotone if for $1 > \nu \gg  2\epsilon \hbar (2\kappa)^{2}$, $t_{1} > t_{2}$
\[h_{\overline{AB}}(\pi_{A} \circ \zeta(t_{1}))=h_{\overline{AB}}(\pi_{A} \circ \zeta(t_{2}))
\Longrightarrow d(\zeta(t_{1}), \zeta(t_{2})) \leq \nu d(\zeta(t_{1}), \zeta(0)) + C_{1} \]

\noindent Note that the definition of weakly monotone is not
symmetrical to both end points: it's biased towards the starting point $\zeta(0)$. \end{itemize} \end{definition}

The following says that in the case of a non-degenerate group, a monotone quasi-geodesic is close to a geodesic segment.
\begin{proposition} \label{close to being straight}
Let $G$ be a non-degenerate, split abelian-by-abelian group, and $\lambda:[0,L] \rightarrow G$ a $(\kappa, C)$
quasi-geodesic whose $\pi_{A}$ image is $\epsilon$-efficient.  Suppose that with respect to $\lambda(0)$, $\lambda$ lies outside of the
$\frac{3}{\delta d(\lambda(0), \lambda(L))}$-linear $+C$ neighborhood of the set of walls based at $\lambda(0)$.
Then \begin{enumerate}
\item $\lambda$ is within $O(\delta L)$ \footnote{ $O(\delta L)$ here can be taken as $2
|\triangle|( \hbar \sqrt{\delta^{2} + 4 \epsilon^{2}} |\overline{AB}| + \delta d(\zeta(0), \zeta(L)))$ } Hausdorff neighborhood of a
straight geodesic segment when $\lambda$ is $\delta$ monotone.

\item $\lambda$ is in $|\triangle| \eta$-linear $+O(1)$ \footnote{ the constant
$O(1)$ can be taken as $|\triangle|( \hbar \sqrt{\delta^{2} + 4 \epsilon^{2}}|\overline{AB}| + C_{1} )$}  neighborhood of a straight
geodesic when $\lambda$ is $(\eta, C_{1})$ weakly monotone.  Recall that $\triangle$ is the set of roots of $G$. \end{enumerate}\end{proposition}


Note that a monotone path is efficient by definition.  So being close to a geodesic
segment is the same as asking that the movement of the path along $\mathbf{H}$ direction
is not too big.  We will prove Proposition \ref{close to being straight} by using the observation that for a
monotone path in $G$, admitting a geodesic approximation is the same as saying that for any root $\alpha \in \triangle$,
its $\pi_{\alpha}$ image admits a (vertical) geodesic approximations.  The next lemma sets out one
scenario where we have (vertical) geodesic approximation in $\mathit{H}_{n+1} = \mathbb{R}^{n} \rtimes_{\psi} \mathbb{R}$
Recall that $\psi(t)$ is $e^{t}N(t)$, where $N(t)$ is a nilpotent matrix with polynomial entries.  We coordinatize
points in $\mathit{H}_{n+1}$ as $(x, t)$, where $x \in \mathbb{R}^{n}$, and $t \in
\mathbb{R}$.  Let $h$ denotes for the projection $(x,t) \mapsto t$.

\begin{lemma}\label{lemma A1-2}
Let $\{p_{i} \}_{i=-s}^{t}$, where $s, t \in \mathbb{Z}^{+}$, be points in
$\mathit{H}_{n+1}$ such that for some $h_{0} >2$, $h(p_{j})=h(p_{j-1})+h_{0}$, $\forall j$.  For $i > 0$, let
$d_{i}$ denote the distance between $p_{i}$ and the vertical geodesic passing through $p_{i-1}$; for
$i <0$, let $d_{i}$ denote for the distance between $p_{i}$ and the vertical geodesic passing $p_{i+1}$.

\begin{enumerate}
\item If for all $j$, $d_{j} \leq r$, and $2r \ll h_{0}$, then there is a geodesic $\gamma_{0}$ such that
$d(\gamma_{0}, p_{j}) \leq 2 r$, for all $j$.

\item If for all $j$, $d_{j} \leq \eta |j| + C_{1}$, where $\eta \ll 1$ and $2C_{1} \leq h_{0}$, then there is a
geodesic $\gamma_{0}$ such that $d(\gamma_{0}, p_{j}) \leq 2 \eta |j| + 2C_{1}$. \end{enumerate} \end{lemma}

\begin{proof}
We first produce geodesic $\gamma^{+}$ and $\gamma^{-}$ that stay close to $\{ p_{i}, i \geq 0 \}$ and
$\{ p_{i}, i \leq 0 \}$ respectively. Then we show that $\gamma^{+}$ and $\gamma^{-}$ meet at some $p_{j}$,
$j \geq 0$ and set $\gamma_{0}$ to be the union between $(\gamma^{+} \cap \gamma^{-})$ and
$\gamma^{-} - (\gamma^{-} \cap \gamma^{+})$. \smallskip

Write $p_{j}=(x_{j},t_{j})$.  We can assume without the loss of generality that $p_{0}=(0,0)$.
Note that the distance between a point $(x_{1},t_{1})$ and the vertical geodesic passing through $(x_{2},t_{2})$
is $U(|x_{1}-x_{2}|)-t_{2}$ by equation (\ref{distance}).

\begin{itemize}
\item Then, by equation (\ref{property of U}), for $j >0$,
\begin{equation*}
\ln |x_{j}-x_{j-1}| - j h_{0} \leq d_{j} \end{equation*} \noindent Hence for all $k \geq 0$,
\begin{equation*}
| x_{k} | \leq \sum_{j=1}^{k} |x_{j} - x_{j-1}| \leq \sum_{j} e^{d_{j} + j h_{0}} \end{equation*}

\noindent Let $\gamma^{+}$ be the geodesic passing through $p_{0}$.
Then for $k \geq 0$,  \begin{equation*}
d(p_{k}, \gamma^{+}) \leq 2\ln \left( \sum_{j}^{k} e^{d_{j}+ j h_{0}} \right) - 2 k h_{0}
= 2 \ln \left( \sum_{j=1}^{k}  e^{d_{j}+(j-k) h_{0}} \right) \end{equation*}

\item For $j < 0 $, again by equation (\ref{property of U})
\begin{equation*} \ln |x_{j+1} - x_{j}| - j h_{0} \leq d_{j} \end{equation*} \noindent Hence
\begin{equation*}
| x_{j+1} - x_{j} | \leq e^{d_{j}+ j h_{0}} \end{equation*} Note that under the assumptions of (i) or
(ii), $x_{-\infty}=\lim_{j \rightarrow -\infty}x_{j}$ exists. So for all $k<0 $,
\begin{equation*}
|x_{k} - x_{-\infty} | \leq \sum_{j=k-1}^{\infty} e^{d_{j} + j h_{0}} \end{equation*}
\noindent Let $\gamma^{-}$ be the vertical geodesic passing through $(x_{-\infty},0)$. Then for $k <0$,
\begin{equation*}
d(p_{k}, \gamma^{-}) \leq 2 \ln \left( \sum_{k-1}^{-\infty} e^{d_{j}+j h_{0}} \right)-2 k h_{0}
= 2 \ln \left( \sum_{j=k-1}^{-\infty} e^{d_{j}+(j-k)h_{0}} \right) \end{equation*} \end{itemize}

\begin{enumerate}
\item In this case, $d(p_{k}, \gamma^{+}) \leq 2r$ for all $k \geq
0$; $d(p_{k'}, \gamma^{-}) \leq 2r$ for all $k' \leq 0$. In
particular, $d(p_{0}, \gamma^{-}) \leq 2r$. Since $\gamma^{+} \ni
p_{0}$, the height at which $\gamma^{+}$ and $\gamma^{-}$ come
together is at most $h(p_{0}) + 2r < h(p_{1})$ by assumption,
therefore $\gamma_{0}$ as defined above satisfies the required
condition.

\item In this case, $d(p_{k}, \gamma^{+}) \leq (2 \eta) k + 2C_{1}$
for $k \geq 0$; $d(p_{k}, \gamma^{-}) \leq (2 \eta) (-k) + 2C_{1}$
for $k \leq 0$. In particular, $d(p_{0}, \gamma^{-}) \leq 2C_{1}$,
so the height at which $\gamma^{+}$ and $\gamma^{-}$ come together
occurs no higher than $h(p_{0})+ 2C_{1}$.  Since $p_{0} \in
\gamma^{+}$, $\gamma_{0}$ therefore satisfies the required
condition. \end{enumerate} \end{proof}

We now proceed to prove Proposition \ref{close to being straight} by showing that if a
path is monotone, then for any root $\alpha$, its $\pi_{\alpha}$ image satisfies the
hypothesis of Lemma \ref{lemma A1-2}.  \begin{proof} \textit{of Proposition \ref{close to being
straight}}

Set \begin{itemize} \item $s=\delta |\overline{AB}|$
\item $t_{j}= \max \{ t \mbox{    }| \mbox{    } h_{\overline{AB}} \circ \pi_{A} \circ \lambda(t)=js \}$

\item $t^{'}_{j}= \min\{ t \in [t_{j-1},t_{j}]  \mbox{     }| \mbox{   }
h_{\overline{AB}} \circ \pi_{A} \circ \lambda(t) = js\}$  \end{itemize}

Therefore for $t \in [t_{j-1},t^{'}_{j}]$, we must have
$h_{\overline{AB}} \circ \pi_{A} \circ \lambda(t) \in [(j-1)s, js]$.  Since
$d(\pi_{A}(\lambda),\overline{AB}) \leq \epsilon \overline{AB}$, the set $\{ \pi_{A} (\lambda(t)), t \in
[t_{j-1},t^{'}_{j}] \}$ lies in a ball of diameter at most $\tilde{s}=\sqrt{s^{2} +
(2 \epsilon |AB|)^{2}}=\sqrt{\delta^{2}+4 \epsilon^{2}}|AB|$, which means
$d(\lambda(t_{j-1}),\lambda(t^{'}_{j})) \leq \hbar \tilde{s}$ by Proposition \ref{can't move far in R2R3}.

\begin{enumerate}
\item  In the case that $\lambda$ is $\delta$ monotone,
\begin{equation*} h_{\overline{AB}} ( \pi_{A} \circ
\lambda(t_{j})) = h_{\overline{AB}} ( \pi_{A} \circ \lambda(t'_{j}))
\Longrightarrow d(\lambda(t_{j}), \lambda(t'_{j})) \leq \delta
d(\zeta(0), \zeta(L)) \end{equation*}

\item If $\lambda$ is $(\eta,C_{1})$ weakly monotone
\begin{equation*} h_{\overline{AB}} ( \pi_{A} \circ \lambda(t_{j})) =
h_{\overline{AB}} ( \pi_{A} \circ \lambda(t^{'}_{j}))
\Longrightarrow d(\zeta(t_{j}), \zeta(t'_{j})) \leq \eta
d(\zeta(t_{j}), \zeta(0)) + C_{1} \end{equation*} \end{enumerate}

\noindent Therefore \[ d(\lambda(t_{j-1}),\lambda(t_{j})) \leq
d(\lambda(t_{j-1}),\lambda(t^{'}_{j})) +
d(\lambda(t^{'}_{j}),\lambda(t_{j})) \leq   \Upsilon \]

\noindent where $\Upsilon= \hbar \tilde{s} + \delta d(\zeta(0),
\zeta(L)) $ when $\lambda$ is $\delta$-monotone; and
\newline $\Upsilon = \hbar \tilde{s} + \eta d(\zeta(t_{j}), \zeta(0))
+ C_{1}$ when $\lambda$ is $(\eta,C_{1})$ weakly monotone. \medskip

By assumption, $\lambda$ lies outside of $\frac{3}{C}$-linear $+C$
neighborhood of the set of walls based at $\lambda(0)$. Since
$d_{H}(\pi_{A}(\lambda), \overline{AB}) \leq \epsilon
|\overline{AB}|$, $h(\pi_{\Xi} \circ \lambda(t_{j})) - h(\pi_{\Xi}
\circ \lambda(t_{j-1})) >2$ for any root $\Xi$.   The claims now
follow from application of Lemma \ref{lemma A1-2} to $\{
\pi_{\Xi}(\lambda(t_{j})) \}_{j}$ in the $\Xi$ weight hyperbolic
space for each root $\Xi$. \end{proof}

We now prove the main lemma in this subsection which roughly says that given $\delta >0$,
a sufficiently long quasi-geodesics whose $\pi_{A}$ image is $\epsilon$-efficient, is
$\delta$-monotone at some scale.

\begin{lemma} \label{monotone scale}
Let $G$ be a non-degenerate, split abelian-by-abelian group.  For any $N \gg 2$, $L_{a} \geq 2\kappa (C)$,
$0< \delta <1$, and $\epsilon>0$, if $\zeta:[0,L] \rightarrow G$ is a $(\kappa,C)$ quasi-geodesic satisfying

\begin{enumerate}
\item $\pi_{A} \circ \zeta$ is $\epsilon$-efficient at scale $\frac{1}{2}\epsilon^{\frac{1}{4}}$, where
$\epsilon \leq \min \{ \left( \frac{\delta}{2\hbar} \right)^{4}, \left( \frac{\delta}{3.01 \hbar}\right)^{8}, (0.01)^{8} \}$

\item \footnote{this long expressions just says that $L$ is sufficiently big with respect to given data. }
\[ \frac{\frac{2L_{a}}{3 \epsilon^{1/8}}} {(\delta)^{\frac{(2\kappa)^{2}\hbar (2N)}{(1-\epsilon^{1/2}\hbar)\delta}}}
\leq 2\kappa L \]  \end{enumerate}

\noindent then there are scales $\rho_{I+1} < \rho_{I} \ll 1$ such
that for $i=I,I+1$,
\begin{equation*} \frac{ |\mathcal{S}(\zeta,\rho_{i}L, \mathbf{P})| } {| \mathcal{S}(\zeta,\rho_{i}L)|} \leq \frac{1}{N} \end{equation*}

\noindent where $\mathbf{P}$ is the statement 'either not $\delta$-monotone, or is monotone but of opposite direction to the $\delta$-monotone
segment in $\mathcal{S}(\zeta,\rho_{i-1}L)$ to which it is a subset of. \end{lemma}

\begin{proof}
The idea of the proof is similar to that of Lemma \ref{efficiency scale}.  Suppose the $\pi_{A}$ image of a segment is efficient but
the segment itself fails to be $\delta$ monotone. Then we can find two points whose $\pi_{A}$ images are close to each other, but the
distance between the two points is very large. By Proposition \ref{can't move far in R2R3}, this means there must be some point in
between those two points whose $\pi_{A}$ image is far away from the $\pi_{A}$ images of those two points. This means that after a
subdivision to the $\pi_{A}$ of the segment, the sum of the distance between consecutive points exceeds the distance of its end points by
some pre-determined amount. In other words, not monotone gains length. But the length of the $\pi_{A}$ image is bounded, a
quasi-geodesic cannot fail to be monotone at smaller and smaller scales. \\


First we note that the conditions on $L$ in relation to $L_{a}$, $\epsilon$, $\delta$ and $N$ is the same as

\begin{equation}\label{Larret}
\frac{\ln \left( \frac{2}{3 \epsilon^{1/8}} \frac{L_{a}}{2\kappa L} \right) }{\ln(\delta)} \geq
\frac{(2\kappa)^{2}\hbar}{(1-\epsilon^{1/2}\hbar)\delta} (2N) \end{equation}  \medskip

\noindent The conditions on $\epsilon$ means that we have
\begin{enumerate} \renewcommand{\labelenumi}{(\Roman{enumi})}
\item  $2 \epsilon^{1/4} \hbar \ll \delta$
\item  $\epsilon^{1/4} \leq 0.01 \epsilon^{1/8}$
\item $3.01 \epsilon^{1/8} \leq \frac{\delta}{\hbar}$ \end{enumerate}

Write $L_{a}=d(\pi_{A} \circ \zeta(0), \pi_{A} \circ \zeta(L))$.  If
$\zeta:[0,L] \rightarrow G$ itself is $\delta$ monotone, we are
done. Otherwise let $\{p_{j}^{0}\}_{j=0}^{n_{0}} \subset \{p_{j}^{1}
\}_{j=0}^{n_{1}} \subset \{p_{j}^{2} \}_{j=0}^{n_{2}} \subset \cdots
\subset \{ p_{j}^{D} \}_{j=0}^{n_{D}}$ be an increasing sets of
points on $\zeta$ such that

\begin{enumerate}

\item $\{ p_{j}^{0} \}= \hat{\mathcal{S}}(\zeta,L_{1})$,
$L_{1}=\delta d(\zeta(0), \zeta(L))$

\item For $i \geq 1$, $\{p_{j}^{i} \}_{j=0}^{n_{i}}=\hat{\mathcal{S}}(\zeta,
L_{i+1})$, where $L_{i+1}=\delta L_{i}$. Note $L_{i+1}< L_{i}$.

\item $1.5(\epsilon^{1/2})^{1/4} L_{D}=L_{a}$ \end{enumerate} \medskip

Let $0 \leq i \leq D$.  $\{ \pi_{A}(p^{i}_{j}) \}_{j=0}^{n_{j}}$ is
a subdivision of $\pi_{A}(\zeta)$. The distance between consecutive
points satisfies $\frac{L_{i+1}}{\hbar} \leq d(\pi_{A}(p^{i}_{j}),
\pi_{A}(p^{i}_{j+1})) \leq L_{i+1}$. We also have
$\frac{L_{i}}{\hbar}, L_{i} \in [\epsilon^{1/4}L_{a},L_{a}]$.
 Therefore by Lemma \ref{lots of efficient segments}, there is a subset $\mathcal{G}_{i} \subset \{
 \pi_{A}(\zeta)_{[\pi_{A}(p^{i}_{j}),\pi_{A}(p^{i}_{j+1})]} \}$, with
$|\mathcal{G}_{i} | \geq (1-\epsilon^{1/2}\hbar) |\{
\pi_{A}(\zeta)_{[\pi_{A}(p^{i}_{j}),\pi_{A}(p^{i}_{j+1})]} \}|$,
such that whenever
$\pi_{A}(\zeta)_{[\pi_{A}(p^{i}_{j'}),\pi_{A}(p^{i}_{j'+1})]}  \in
\mathcal{G}_{i}$, it is $\epsilon^{1/2}$ efficient at scale
$\frac{1}{2}\epsilon^{1/4}$. \bigskip

We define the following : \begin{itemize}
\item $\mathcal{C}_{i}=\{1 \leq j \leq n_{i} \mbox{  } | \mbox{  }
\pi_{A}(\zeta)_{[\pi_{A}(p_{j}^{i}),\pi_{A}(p_{j+1}^{i})]} \in \mathcal{G}_{i} \}$ is the set of subsegments produced by $\{p^{i}_{j}\}$
whose $\pi_{A}$ images are $\epsilon^{1/2}$ efficient at scale $\frac{1}{2} \epsilon^{1/4}$.

\item $\mathcal{NC}_{i}=\{ 1 \leq j \leq n_{i} \mbox{  } | \mbox{  }  j \not \in \mathcal{C}_{i} \}$, is those
subsegments whose $\pi_{A}$ images are not $\epsilon^{1/2}$ efficient at scale $\frac{1}{2} \epsilon^{1/4}$.
Note $\frac{|\mathcal{C}_{i}|}{|\mathcal{C}_{i}| + |\mathcal{NC}_{i}|} \geq 1-\epsilon^{1/2}\hbar$

\item $\mathcal{B}_{i}=\{1 \leq j \leq n_{i} \mbox{ }  | \mbox{  } j \in \mathcal{C}_{i}, \mbox{  } \zeta_{[(p_{j}^{i}),
p_{j+1}^{i}]} \mbox{ \textit{is not} } \delta \mbox{ \textit{monotone}} \} $ is those segments whose $\pi_{A}$
images are $\epsilon^{1/2}$ efficient at scale $\frac{1}{2} \epsilon^{1/4}$ but fails to be $\delta$ monotone.

\item $\flat_{i}=\frac{|\mathcal{B}_{i}|}{|\mathcal{C}_{i}|}$ be the proportion of
subsegments that are $\epsilon^{1/2}$-efficient at scale $1/2 \epsilon^{1/4}$ but fails
to be $\delta$ monotone.

\item For $J \in \mathcal{C}_{i} - \mathcal{B}_{i}$, \newline
$\Psi_{i+1,J} = \{  j' \in \mathcal{C}_{i+1} -\mathcal{B}_{i+1} \mbox{ }  | \mbox{ } \zeta_{[p_{j'}^{i+1},p_{j'+1}^{i+1}]} \subset
\zeta_{[p_{J}^{i},p_{J+1}^{i}]},\newline \mbox{ \textit{but those two have opposite orientations} } \}$ are basically those
subsegments produced by $\{p^{i+1}_{j} \}_{j=0}^{n_{i+1}}$ that are $\delta$ monotone and belong to a $\delta$ monotone subsegment
produced by $\{p^{i}_{j}\}_{j=0}^{n_{i}}$ but their orientations do not agree.

\item $\mathcal{R}_{i+1}=\bigcup_{J \in \mathcal{C}_{i}-\mathcal{B}_{i}} \Psi_{i+1,J}$

\item $\natural_{i+1}=\frac{|\mathcal{R}_{i+1}|}{|\mathcal{C}_{i+1}|}$ be the proportion
of subsegments that are $\epsilon^{1/2}$-efficient at scale $1/2 \epsilon^{1/4}$ and
$\delta$ monotone but of wrong orientation.

\item Write $\hat{L}=(1-\epsilon^{1/2}\hbar)L$ and note that
$|\mathcal{C}_{i} | \geq \frac{\hat{L}}{2 \kappa L_{i+1}}$
\end{itemize} \bigskip

Since $\zeta$ is not $\delta$-monotone, there are two points
$t_{1},t_{2} \in [0,L]$ such that

\begin{enumerate}
\item $h_{\overline{\pi_{A} \circ \zeta(0) \pi_{A} \circ
\zeta(L)}}(\pi_{A} \circ \zeta(t_{1}))=h_{\overline{\pi_{A} \circ
\zeta(0) \pi_{A} \circ \zeta(L)}}(\pi_{A} \circ \zeta(t_{2}))$. This
means \begin{equation}\label{A} d(\pi_{A} \circ \zeta(t_{1}),\pi_{A}
\circ \zeta(t_{2})) \leq 4 \epsilon^{1/4} L_{a}
\end{equation} because $\pi_{A} \circ \zeta$ is
$\epsilon$-efficient on scale $\frac{1}{2} \epsilon^{1/4}$, which
means the Hausdorff distance between $\pi_{A} \circ \zeta$ and
$\overline{\pi_{A} \circ \zeta(0) \pi_{A} \circ \zeta(L)}$ is at
most $2 \epsilon^{1/4}L_{a}$. \medskip

AND

\item  $d(\zeta(t_{1}),\zeta(t_{2})) \geq  \delta d(\zeta(0),
\zeta(L))$.   By Proposition \ref{can't move far in R2R3}, this
means $\exists t \in [t_{1},t_{2}]$ such that
\begin{equation}\label{B}
d(\pi_{A} \circ \zeta(t),\pi_{A} \circ \zeta(t_{i})) \geq
 \frac{\delta}{\hbar} d(\zeta(0), \zeta(L)) \end{equation}
 for $i=1,2$ in light of(\ref{A}) \end{enumerate}

\noindent Equations (\ref{A}) and(\ref{B}) together means
\begin{equation*}
\sum_{j=0}^{n_{0}} d(\pi_{A}(p_{j}^{0})b, \pi_{A}(p_{j+1}^{0})) -
d(\pi_{A}(p_{0}^{0}),\pi_{A}(p_{n_{0}}^{0}))  \geq  \left(2
\frac{\delta}{\hbar}d(\zeta(0), \zeta(L)) - 4 \epsilon^{1/4}L_{a}
\right) \end{equation*}

i.e.  \begin{equation*} \sum_{j=0}^{n_{0}} d(\pi_{A}(p_{j}^{0})b,
\pi_{A}(p_{j+1}^{0})) - L_{a}  \geq  \left( \frac{2
\delta}{\hbar}-4\epsilon^{1/4} \right)L_{a}  \gg 0
\end{equation*}

\noindent where we used property (I) for the last inequality and recalled that $L_{a}=d(\pi_{A}(p^{0}_{0}), \pi_{A}(p^{0}_{n_{0}}))$. \medskip

Now for each $D \geq i \geq 1$, $ 1 \leq j \leq n_{i}$,
\begin{itemize}

\item EITHER $j \in \mathcal{B}_{i}$.  In this case
$\pi_{A} \circ \zeta_{[p_{j}^{i}, p_{j+1}^{i}]}$ is $\epsilon^{1/2}$
efficient at scale $\frac{1}{2} \epsilon^{1/4}$ but not
$\delta$-monotone.

Then there are two points $t_{1},t_{2} \in
[\zeta^{-1}(p_{j}^{i}),\zeta^{-1}(p_{j+1}^{i})]$ such that

\begin{enumerate}
\item \begin{equation*}
h_{\overline{\pi_{A}(p_{j}^{i}),\pi_{A}(p_{j+1}^{i})}}( \pi_{A}
\circ \zeta(t_{1})) =
h_{\overline{\pi_{A}(p_{j}^{i}),\pi_{A}(p_{j+1}^{i})}}( \pi_{A}
\circ \zeta(t_{2})) \end{equation*}

This means that \begin{equation*} d(\pi_{A} \circ
\zeta(t_{1}),\pi_{A} \circ \zeta(t_{2})) \leq 2 \left(\frac{3}{2}(
\epsilon^{\frac{1}{2}})^{\frac{1}{4}} d(\pi_{A}(p^{i}_{j}),
\pi_{A}(p^{i}_{j+1})) + \frac{1}{2}
\epsilon^{1/4}d(\pi_{A}(p^{i}_{j}), \pi_{A}(p^{i}_{j+1})) \right)
\end{equation*}  by Lemma \ref{efficient in Rn}. \smallskip

Therefore by property (II) in the hypothesis

\begin{equation}\label{AA} d(\pi_{A} \circ \zeta(t_{1}), \pi_{A} \circ \zeta(t_{2})) \leq
3.01\epsilon^{1/8}L_{i+1} \end{equation}
\medskip

AND

\item $d(\zeta(t_{1}), \zeta(t_{2})) > \delta L_{i+1}$. By Proposition \ref{can't move far in R2R3},
this means $\exists t \in [t_{1},t_{2}]$ such that
\begin{equation}\label{BB}
d(\pi_{A}(\zeta(t)),\pi_{A}(\zeta(t_{i}))) \geq \frac{\delta}{\hbar} L_{i+1}
\end{equation} for $i=1,2$ in light of (\ref{AA}) \end{enumerate}

say $p_{j}^{i}=p_{s1_{j}}^{i+1}$, $p_{j+1}^{i}=p_{s2_{j}}^{i+2}$,
then (\ref{AA}) together with \ref{BB} imply that
\begin{eqnarray*}
\sum_{t=s1_{j}}^{s2_{j}-1} d(\pi_{A}
(p_{t}^{i+1}),\pi_{A}(p_{t+1}^{i+1})) & \geq &
d(\pi_{A}(p_{j}^{i}),\pi_{A}(p_{j+1}^{i})) + \left( 2\frac{\delta
L_{i+1}}{\hbar} \right) - 3.01 \epsilon^{1/8}L_{i+1} \\
& \geq & d(\pi_{A}(p_{j}^{i}),\pi_{A}(p_{j+1}^{i})) + H_{i+1} \delta
L_{i+1} \end{eqnarray*} \noindent  where we have set constants
$H_{i+1}$ to satisfy

\begin{equation} \label{property of Hi}
\frac{2 \delta}{\hbar} - 3.01 \epsilon^{1/8} \geq H_{i+1} \delta
\end{equation} \smallskip

Summing over all $j \in \mathcal{B}_{i}$ we have

\begin{eqnarray*}
\sum_{j \in \mathcal{B}_{i}} \sum_{t=s1_{j}}^{s2_{j}-1} d(\pi_{A}
(p_{t}^{i+1}),\pi_{A}(p_{t+1}^{i+1})) & \geq & \sum_{j \in
\mathcal{B}_{i}} d(\pi_{A}(p_{j}^{i}),\pi_{A}(p_{j+1}^{i})) + |
\mathcal{B}_{i} | H_{i+1} \delta L_{i+1} \\
& \geq &  \sum_{j \in \mathcal{B}_{i}}
d(\pi_{A}(p_{j}^{i}),\pi_{A}(p_{j+1}^{i})) + \flat_{i}
|\mathcal{C}_{i} | H_{i+1} \delta L_{i+1} \\ & \geq & \sum_{j \in
\mathcal{B}_{i}} d(\pi_{A}(p_{j}^{i}),\pi_{A}(p_{j+1}^{i})) +
\flat_{i} \frac{\hat{L}}{2 \kappa  L_{i+1}} H_{i+1} \delta L_{i+1}
\end{eqnarray*}

That is, \begin{equation}\label{those in C}
\sum_{j \in \mathcal{B}_{i}} \sum_{t=s1_{j}}^{s2_{j}-1} d(\pi_{A}
(p_{t}^{i+1}),\pi_{A}(p_{t+1}^{i+1}))  \geq \sum_{j \in
\mathcal{B}_{i}} d(\pi_{A}(p_{j}^{i}),\pi_{A}(p_{j+1}^{i})) +
 \frac{\flat_{i}\hat{L}}{2 \kappa} H_{i+1} \delta \end{equation}

\item  OR $j \in \mathcal{C}_{i} - \mathcal{B}_{i}$.  In this case
$\zeta_{[p_{j}^{i},p_{j+1}^{i}]}$ is $\delta$-monotone. Write
$p_{j}^{i}=p_{s1_{j}}^{i+1}$, $p_{j+1}^{i}=p_{s2_{j}}^{i+1}$, then

\begin{equation*}
\sum_{t=s1_{j}}^{s2_{j}-1}
d(\pi_{A}(p_{t}^{i+1}),\pi_{A}(p_{t}^{i+1})) - \sum_{z \in
\Psi_{i+1,j}} \frac{L_{i+2}}{\hbar} \geq
d(\pi_{A}(p_{j}^{i}),\pi_{A}(p_{j+1}^{i})) \end{equation*}
\smallskip Summing over all $j \in \mathcal{C}_{i} - \mathcal{B}_{i}$
we have

\begin{eqnarray*}
\sum_{j \in \mathcal{C}_{i} - \mathcal{B}_{i}}
\sum_{t=s1_{j}}^{s2_{j}-1}
d(\pi_{A}(p_{t}^{i+1}),\pi_{A}(p_{t}^{i+1})) & \geq & \sum_{j \in
\mathcal{C}_{i}-\mathcal{B}_{i}}d(\pi_{A}(p_{j}^{i}),\pi_{A}(p_{j+1}^{i}))
+ \sum_{j \in \mathcal{C}_{i} - \mathcal{B}_{i}} \sum_{z \in
\Psi_{i+1,j}} \frac{L_{i+2}}{\hbar} \\
& =&  \sum_{j \in
\mathcal{C}_{i}-\mathcal{B}_{i}}d(\pi_{A}(p_{j}^{i}),\pi_{A}(p_{j+1}^{i}))+
|\mathcal{R}_{i+1}| \frac{L_{i+2}}{\hbar} \\
& \geq & \sum_{j \in
\mathcal{C}_{i}-\mathcal{B}_{i}}d(\pi_{A}(p_{j}^{i}),\pi_{A}(p_{j+1}^{i}))+
\natural_{i+1} |\mathcal{C}_{i+1}| \frac{L_{i+2}}{\hbar} \\
& \geq & \sum_{j \in
\mathcal{C}_{i}-\mathcal{B}_{i}}d(\pi_{A}(p_{j}^{i}),\pi_{A}(p_{j+1}^{i}))
+ \natural_{i+1} \frac{\hat{L}}{2 \kappa
L_{i+2}}\frac{L_{i+2}}{\hbar}\end{eqnarray*}

That is,
\begin{equation}\label{those in C-B}
\sum_{j \in \mathcal{C}_{i} - \mathcal{B}_{i}}
\sum_{t=s1_{j}}^{s2_{j}-1}
d(\pi_{A}(p_{t}^{i+1}),\pi_{A}(p_{t}^{i+1}))  \geq  \sum_{j \in
\mathcal{C}_{i}-\mathcal{B}_{i}}d(\pi_{A}(p_{j}^{i}),\pi_{A}(p_{j+1}^{i}))+
 \frac{\natural_{i+1}\hat{L}}{2 \kappa \hbar} \end{equation}

\item OR $j \in \mathcal{NC}_{i}$. Write $p_{j}^{i}=p_{s1_{j}}^{i+1}$,
$p_{j+1}^{i}=p_{s2_{j}}^{i+1}$, then

\begin{equation*}
\sum_{t=s1_{j}}^{s2_{j}-1}
d(\pi_{A}(p_{t}^{i+1}),\pi_{A}(p_{t}^{i+1})) \geq
d(\pi_{A}(p_{j}^{i}),\pi_{A}(p_{j+1}^{i})) \end{equation*}
\smallskip Summing over all $j \in \mathcal{NC}_{i}$ we have

\begin{equation}\label{those in NC}
\sum_{j \in \mathcal{NC}_{i}} \sum_{t=s1_{j}}^{s2_{j}-1}
d(\pi_{A}(p_{t}^{i+1}),\pi_{A}(p_{t}^{i+1})) \geq \sum_{j \in
\mathcal{NC}_{i}} d(\pi_{A}(p_{j}^{i}),\pi_{A}(p_{j+1}^{i}))
\end{equation} \end{itemize}

Putting (\ref{those in C}), (\ref{those in C-B}) and (\ref{those in
NC}) together, we have
\begin{eqnarray*} \sum_{1 \leq w \leq n_{i+1}}
d(\pi_{A}(p_{w}^{i+1}),\pi_{A}(p_{w+1}^{i+1})) & \geq & \sum_{ 1
\leq z \leq n_{i}} d(\pi_{A}(p_{z}^{i}),\pi_{A}(p_{z+1}^{i})) +
\frac{\flat_{i}\hat{L}}{2 \kappa} H_{i+1} \delta +
\frac{\natural_{i+1} \hat{L}}{2 \kappa} \end{eqnarray*} \bigskip

By equation (\ref{property of Hi}), $H_{i}$ satisfies
\[ \frac{2 \delta }{\hbar} - 3.01 \epsilon^{1/8}  \geq H_{i+1} \delta  \] by property (III)
in the hypothesis on $\epsilon$ and $\delta$, $3.01 \epsilon^{1/8}
\leq \frac{\delta}{ \hbar}$, so we can take $H_{i}=\frac{1}{\hbar}$,
hence

\begin{equation}\label{gaining length}
\sum_{1 \leq w \leq n_{i+1} }
d(\pi_{A}(p_{w}^{i+1}),\pi_{A}(p_{w+1}^{i+1})) \geq  \sum_{1 \leq z
\leq n_{i}} d(\pi_{A}(p_{z}^{i}),\pi_{A}(p_{z+1}^{i})) + \flat_{i}
\frac{\delta \hat{L}}{2\kappa \hbar} + \natural_{i+1} \frac{\hat{L}}{2 \kappa} \end{equation}\\

Write $\lambda_{i}= \sum_{ 1 \leq j \leq n_{i} }
d(\pi_{A}(p_{j}^{i}),\pi_{A}(p_{j+1}^{i}))$, and using (\ref{gaining
length}) we have

\begin{eqnarray*}
2 \kappa L \geq \lambda_{D}-L_{a} = \left(\sum_{i=0}^{D-1}
\lambda_{i+1}-\lambda_{i} \right) + \lambda_{0}-L_{a} & \geq &
\sum_{i=0}^{D-1} \lambda_{i+1}-\lambda_{i} \\
& \geq & \sum_{i=0}^{D-1} \frac{\flat_{i} \delta
\hat{L}}{2\kappa \hbar} + \frac{\natural_{i+1} \hat{L}}{2 \kappa} \\
&=& \frac{\delta \hat{L}}{2\kappa \hbar} \sum_{i=0}^{D-1} \flat_{i}
+ \frac{\hat{L}}{2 \kappa} \sum_{i=0}^{D-1} \natural_{i+1}
\end{eqnarray*} Divide both sides by $\hat{L}=(1 - \epsilon^{1/2}\hbar)L$

\begin{equation}\label{sum of scales for monotone}
\frac{2 \kappa}{1-\epsilon^{1/2}\hbar}  \geq \frac{\delta }{2\kappa
\hbar} \sum_{i=0}^{D-1} \flat_{i}  +  \frac{1}{2 \kappa}
\sum_{i=0}^{D-1} \natural_{i+1} \geq \frac{\delta }{2\kappa \hbar}
\sum_{i=0}^{D-1} (\flat_{i} + \natural_{i+1}) \end{equation}

\noindent Since $L_{i+1}=\delta L_{i}$ for $i \geq 1$, $L_{1}=\delta
d(\zeta(0), \zeta(L))$, $L_{i} = \delta^{i} d(\zeta(0), \zeta(L))$.
The condition on $L_{D}$ means
\begin{eqnarray*}
1.5 (\epsilon^{1/2})^{1/4} L_{D} &=& L_{a} \\
1.5 \epsilon^{1/8} \delta^{D} d(\zeta(0),
\zeta(L))=L_{a}\\
\frac{2}{3 \epsilon^{1/8}}
\frac{L_{a}}{d(\zeta(0),\zeta(L))} &=& \delta^{D} \\
D &=& \frac{\ln(\frac{2}{3\epsilon^{1/8}}
\frac{L_{a}}{d(\zeta(0), \zeta(L))} ) }{\ln(\delta) }
\end{eqnarray*}

\noindent By equation(\ref{Larret}), we have $D \geq 2N
\frac{(2\kappa)^{2} \hbar}{\delta (1-\epsilon^{1/2}\hbar)}$.  So
equation (\ref{sum of scales for monotone}) implies that for some $1
\leq I \leq D-1$ we must have
\begin{equation*} \flat_{I-1} + \natural_{I-1} \leq \frac{1}{N}, \mbox{        }  \flat_{I} + \natural_{I} \leq \frac{1}{N}
\end{equation*} Recall that $\flat_{I'}$ is the proportion of efficient
subsegments produced by $\{p^{I'}_{j} \}_{j=0}^{n_{I'}}$ that are
not monotone, and $\natural_{I'}$ is the proportion that are
monotone but of the wrong orientation.  The desired
$\rho_{I}=\frac{L_{I}}{L}$, $\rho_{I+1}=\frac{L_{I+1}}{L}$
\end{proof}

\begin{corollary} \label{monotone scale multiple}
Let $G$ be a non-degenerate, split abelian-by-abelian group.  Take any $2 \ll N_{0} < N/2 $, $L_{0} \geq 2\kappa (C)$,
$0< \delta <1$, and $\epsilon>0$, and let $\mathcal{F}=\{ \zeta_{j} \}$ be a finite set of $(\kappa, C)$ quasi-geodesics.
If every element of $\mathcal{F}$, $\zeta_{j}: [0,L_{j}] \rightarrow G$ satisfies the
following:
\begin{enumerate}
\item $\pi_{A} \circ \zeta_{j}$ is $\epsilon$-efficient at scale $\frac{1}{2}\epsilon^{\frac{1}{4}}$,
where $\epsilon \leq \min \{ \left( \frac{\delta}{2\hbar} \right)^{4}, \left( \frac{\delta}{3.01 \hbar}\right)^{8}, (0.01)^{8} \}$

\item \[ \frac{\frac{2L_{0}}{3 \epsilon^{1/8}}} {(\delta)^{\frac{(2\kappa)^{2}\hbar (2N)}{(1-\epsilon^{1/2}\hbar)\delta}}}
\leq 2\kappa L_{j} \]  \end{enumerate}

\noindent then there are scales $\rho_{I+1} < \rho_{I} \ll 1$, and a subset $\mathcal{F}_{0} \subset \mathcal{F}$ such
that \begin{enumerate}
\item $|\mathcal{F}_{0}| \geq \left( 1-\frac{2N_{0}}{N} \right) |\mathcal{F}|$

\item for every $\zeta \in \mathcal{F}_{0}$, and $i=I,I+1$,
\begin{equation*} \frac{ |\mathcal{S}(\zeta,\rho_{i}L, \mathbf{P})| } {| \mathcal{S}(\zeta,\rho_{i}L)|} \leq \frac{1}{N} \end{equation*}

\noindent where $\mathbf{P}$ is the statement 'either not $\delta$-monotone, or is monotone but of opposite direction to the $\delta$-monotone
segment in $\mathcal{S}(\zeta,\rho_{i-1}L)$ to which it is a subset of. \end{enumerate} \end{corollary}

\begin{proof}
We apply Lemma \ref{monotone scale} to each element of $\gothic{G}$ since its elements are all $(\kappa, C)$
quasi-geodesics.  We arrive at equation (\ref{sum of scales for monotone}) for each element of
$\mathcal{F}$.  That is,

\begin{equation*}
\frac{2 \kappa}{1-\epsilon^{1/2}\hbar}  \geq  \frac{\delta }{2\kappa \hbar}
\sum_{i=0}^{D-1} (\flat_{i}(\zeta_{j}) + \natural_{i+1}(\zeta_{j})) \end{equation*}

\noindent therefore

\[ \frac{2 \kappa}{1-\epsilon^{1/2}\hbar} \geq \frac{1}{|\mathcal{F}|} \sum_{\zeta_{j} \in \mathcal{F}} \frac{\delta }{2\kappa \hbar}
\sum_{i=0}^{D-1} (\flat_{i}(\zeta_{j}) + \natural_{i+1}(\zeta_{j})) = \frac{\delta }{2\kappa \hbar}
\sum_{i=0}^{D-1} \frac{1}{|\mathcal{F}|} \sum_{\zeta_{j} \in \mathcal{F}} (\flat_{i}(\zeta_{j}) + \natural_{i+1}(\zeta_{j})) \]

\noindent Counting the number of terms on the right hand side means that for some $1 \leq I \leq
D-1$,
\begin{eqnarray*}
\frac{1}{|\mathcal{F}|} \sum_{\zeta_{j} \in \mathcal{F}} (\flat_{I-1}(\zeta_{j}) + \natural_{I-1}(\zeta_{j}))
& \leq & \frac{1}{N} \\
\frac{1}{|\mathcal{F}|} \sum_{\zeta_{j} \in \mathcal{F}} (\flat_{I}(\zeta_{j}) + \natural_{I}(\zeta_{j}))
& \leq & \frac{1}{N} \end{eqnarray*}

Let $\mathcal{F}_{b}$ consist of those $\zeta$ whose $\flat_{I}+ \natural_{I}$ or $\flat_{I-1}+
\natural_{I-1}$ values is more than $\frac{1}{N_{0}}$.  The desired claim is obtained
after applying Chebyshev inequality and setting $\mathcal{F}_{0}$ as the complement of
$\mathcal{F}_{b}$, and $\rho_{I}=\delta^{I}$, $\rho_{I+1}=\delta^{I+1}$.  \end{proof}

\subsection{Occurrence of weakly monotone segments}
In the previous subsection we showed the existence of a $\delta$-monotone scale.  In this
subsection, we write $G$ for a non-degenerate, split abelian-by-abelian group, and
we will see that by chaining a lot of $\delta$-monotone segments together, we
end up with a path that is weakly monotone.

\begin{definition}\label{uniform set}
Let $G$ be a non-degenerate, split abelian-by-abelian group.  Let $\zeta:[0,L] \rightarrow G$ be a $(\kappa,C)$ quasi-geodesic
segment that is $\delta$-monotone. Suppose for some $L_{s} \gg 2\kappa  C$,
\[ \frac{|\mathcal{S}(\zeta,L_{s},\mathbf{P})|}{|\mathcal{S}(\zeta,L_{s})|}\leq \frac{1}{N}  \]

\noindent where $\mathbf{P}$ is the statement "not $\delta$ monotone, or is $\delta$ monotone but with opposite
orientation from $\zeta$". \smallskip

\noindent For a point $x \in \zeta$, define
\begin{equation*}
P(x,\zeta,T) = | B(x,T) \cap \Delta|, \mbox{ where } \Delta = \bigcup_{\lambda \in \mathcal{S}(\zeta,L_{s},\mathbf{P})} \lambda \end{equation*}

\noindent We say $x$ is (M) \emph{uniform along $\zeta$} if for
all $T \geq 0 $, \[ P(x,\zeta,T) \leq M \left( \frac{|\mathcal{S}(\zeta,L_{s},\mathbf{P})|}{|\mathcal{S}(\zeta,L_{s})|}
\right) T \] where $M$ satisfies $\frac{M}{N} \ll 1$. \end{definition}

The main lemma of this subsection is:

\begin{lemma} \label{how to show weakly monotone}
Let $\zeta$ be a $\delta$-monotone $(\kappa, C)$ quasi-geodesic in $G$. Suppose $x \in \zeta$ is a uniform point,
with $P(x,\zeta,T) \leq \nu T$, $\nu \ll 1$. Then $\zeta$ consider as a $(\kappa,C)$ quasi-geodesic leaving $x$ at $T=0$
is $(\nu (1+ \hbar), 2\kappa L_{s} )$ weakly monotone. \end{lemma}

\begin{proof}
let $h$ denote the projection of $\pi_{A}(\zeta)$ onto the straight line joining the end points of $\pi_{A}(\zeta)$.

Then up to time $T$, provided $T > L_{s}$, at most $\nu$ proportion of segments in $\mathcal{S}(\zeta,L_{s})$ belong
to $\mathcal{S}(\zeta,L_{s}, \mathbf{P})$, and at least $1-\nu$ proportion are $\delta$ monotone.  Therefore
\begin{equation} \label{moving linearly}
h(\zeta(T))-h(x) \geq \frac{(1-\nu)T}{\hbar} - \nu T =(1-\nu - \hbar \nu)\frac{T}{\hbar} \end{equation} so
$\pi_{A}(\zeta)$ moves at a linear rate. \smallskip

For any $\hat{s}>0$, let $t_{1}, t_{2}$ be the smallest and largest number $t$ such that $h(\zeta(t))=\hat{s}$.
Let $b$ denotes the proportion of $\mathcal{S}(\zeta, L_{s})$ in between $\zeta(t_{1})$ and $\zeta(t_{2})$ that
belongs to $\mathcal{S}(\zeta,L_{s},\mathbf{P})$. Either $\zeta(t_{2})-\zeta(t_{1}) \leq L_{s}$, in which
case $t_{2}-t_{1} \leq 2\kappa L_{s}$; OR $\zeta(t_{2})-\zeta(t_{1}) > L_{s}$, in which case we have
\begin{equation*}
0= h(\zeta(t_{2}))-h(\zeta(t_{1})) \geq \frac{(1-b)(t_{2}-t_{1})}{\hbar} - b (t_{2}-t_{1}) \end{equation*}
\noindent which means $ b > \frac{1}{1+ \hbar} $.  On the other hand, we also know that $b(t_{2}-t_{1}) \leq \nu t_{2}$,
therefore \begin{equation*}
t_{2}-t_{1} \leq \frac{\nu t_{2}}{b} \leq (1+ \hbar) \nu t_{2}
\end{equation*}

\noindent That is, whenever $h(t_{2})=h(t_{1})$, we must have
$t_{2}-t_{1} \leq (1+ \hbar) \nu t_{2} + 2\kappa L_{s}$. \end{proof}

The following lemma provides us with abundant supply of uniform points.
\begin{lemma}\label{big uniform set}
Let $\zeta$ be a $\delta$-monotone $(\kappa, C)$ quasi-geodesic in $G$.  Then the proportion of
non-$M$ uniform points in $\hat{\mathcal{S}}(\zeta,L_{s})$ is at most $\frac{2}{M}$. \end{lemma}
\begin{proof}
Write $\Delta$ as the union of all segments in $\mathcal{S}(\zeta,
L_{s}, \mathbf{P})$, $N$ as the measure of the union of all the
segments in $ \mathcal{S}(\zeta,L_{s})$, and $\mu=\frac{|\Delta
|}{N}$. For every non-uniform point $x$, we can find an interval
$I_{x}$, such that \[ |I_{x} \cap \Delta| \geq M \mu |I_{x}| \]

\noindent Then the collection of all such interval $\{I_{x}\}$ forms
a cover for the set of non-uniform points.  Choose a subcover so
that $\sum |I_{x} \cap \Delta| \leq 2 |\Delta|$. Then
\[ \left| \bigcup I_{x} \right| \leq \sum |I_{x}| \leq \sum \frac{1}{M \mu} \left| I_{x} \cap
\Delta \right| \leq \frac{2 |\Delta|}{M \mu} \]  \noindent now divide both sides
by $N$. \end{proof}

\begin{remark} \label{general uniform}
Let $\zeta$ be a quasi-geodesic segment that lies within $(\nu,c)$-linear neighborhood of a
geodesic segment, where $c \ll \nu |\zeta|$.  In light of Lemma \ref{how to show weakly monotone}, we may call a
point $p \in \zeta$ as a $\nu$-uniform point if the subsegments of $\zeta$ of length $\gg \nu |\zeta|$, viewed as
quasi-geodesics starting from $p$, lies in $(\nu, c')$-linear neighborhood of geodesic
segments for some $c' \ll \nu |\zeta|$.  \end{remark}

\begin{remark}
By abuse of notation, from now on, when we say a point $p$ is $M$ uniform with respect to a
quasi-geodesic segment for some $M \gg 1$, we mean definition \ref{uniform set}; if we
say $p$ is $\nu$ uniform, where $\nu < 1$, we mean remark \ref{general uniform}.  \end{remark}

\subsection{Proof of Theorem \ref{existence of good boxes}}

So far our results from previous sections only require the group to be non-degenerate and
split abelian-by-abelian.  From now on, we will require all our groups to be unimodular.


\begin{proposition}\label{finite number of qigeodesics close to geodesics}
Let $G, G'$ be non-degenerate, unimodular, split abelian-by-abelian Lie groups, and $\phi: G \rightarrow G'$ be a $(\kappa, C)$
quasi-isometry.  Then, to any $0< \delta, \eta< \tilde{\eta} <1 $, there are numbers $L_{0}$, $m>1$, and $0< \rho_{s} < \rho_{b'} < \rho_{b} \leq 1$
depending on $\delta, \eta$, and $\kappa, C$, with the following properties: \smallskip

If $\Omega \subset \mathbf{A}$ is a product of intervals of equal size at least $mL_{0}$, by writing
\begin{itemize}
\item $\mathcal{P}=\phi ( \mathcal{P}(\Omega))$ as the $\phi$ images of points in
$\mathbf{B}(\Omega)$,
\item $\mathbf{L}=\bigcup_{\zeta \in \phi(\mathcal{L}(\Omega)[m])} \mathcal{S}(\zeta, \hbar \rho_{s} \rho_{b} |\zeta|)$
as the union of subsegments obtained by dividing each $\zeta \in \mathcal{L}(\Omega)[m]$ at scale $\hbar
\rho_{s}\rho_{b}$. \end{itemize}  Then,
\begin{enumerate}
\item there is a subset $\mathbf{L}_{0} \subset \mathbf{L}$, with $|\mathbf{L}_{0}| \geq (1-\tilde{q}) |\mathbf{L}|$,
\item there is also a subset $\tilde{\mathcal{P}} \subset \mathcal{P}$, with $|\tilde{\mathcal{P}}| \geq (1-Q)|\mathcal{P}|$
\end{enumerate}

\noindent such that for every $p \in \tilde{\mathcal{P}}$, amongst all elements in $\mathbf{L}$ containing $p$, at least $1-\tilde{Q}$
proportion of them belong to $\mathbf{L}_{0}$, and of those, a further $1-\hat{Q}$ proportion admit geodesic approximation.  That is,
if $\gamma$ is in this set, then it is within $(\eta, (\delta+ \frac{ \rho_{b'}}{\rho_{b}} ) |\gamma|)$-linear neighborhood of a
geodesic segment that makes an angle of at least $\sin^{-1}(\tilde{\eta})$ with root kernel directions.  Here $\tilde{\eta}$, $\tilde{q}$,
$Q, \tilde{Q}, \hat{Q} \rightarrow 0$ as $\eta$, $\delta$, $\tilde{\eta}$ approach zero. \end{proposition}

\begin{proof}
Recall that $G=\mathbf{H} \rtimes \mathbf{A}$, $G' = \mathbf{H}' \rtimes \mathbf{A}'$, where $\mathbf{H}'$, $\mathbf{H}$,
$\mathbf{A}'$, $\mathbf{A}$ are all are abelian, and $\triangle'$ is the set of roots of $G'$.  First, we choose the following constants:
\begin{itemize}
\item $N \gg dim(\mathbf{A}')$ such that $\frac{|\triangle'|}{\sqrt{N}}(1+\hbar) < \eta$,
\item $m = N^{1/3}$

\item $\epsilon \leq \tilde{\eta} \min \{ \left( \frac{\delta}{2\hbar} \right)^{4},
\left( \frac{\delta}{3.01 \hbar} \right)^{8}, (0.01)^{8} \}$
\item $L_{a} \geq \frac{3}{\varepsilon}$

\item $L_{stop}$ such that  \begin{equation*}
\frac{\frac{2L_{a}}{3 \epsilon^{1/8}}} {(\delta)^{\frac{(2\kappa)^{2}\hbar (2N)}{(1-\epsilon^{1/2}\hbar)\delta}}} \leq 2\kappa L_{stop} \end{equation*}

\item $L_{0}$ such that \begin{equation*}
\frac{L_{stop}}{\left( \frac{1}{2}\epsilon^{1/4} \right)^{\frac{ \hbar (2\kappa)^{2} N + \epsilon}{\epsilon}}}  \leq 2\kappa L_{0} \end{equation*}

\item $M = \eta \sqrt{N}$.  \end{itemize}

Let $\Omega \subset \mathbf{A}$ be a product of intervals of equal size at least $m
L_{0}$.  We will build $\tilde{\mathcal{P}}$ from the $\phi$ images of $\mathcal{P}(\Omega) =
\bigcup_{\zeta \in \mathcal{L}(\Omega)[m]} \mathcal{P}(\zeta)$.

Let $\mathcal{F}= \phi(\mathcal{L}(\Omega)[m])$, and apply Lemma \ref{efficien scale sum multiple} to $\mathcal{F}$ and $N_{0}=\sqrt{N}$ to obtain
a scale $\rho_{J}$ and subset $\mathcal{F}_{0}$ such that
\[ \left|  \bigcup_{\zeta \in \mathcal{F}_{0}} \mathcal{S}(\pi_{A}(\zeta), \rho_{J}|\pi_{A} \circ \zeta|,
\mbox{$\epsilon$-efficient at scale $1/2 \epsilon^{1/4}$}) \right| \geq \left( 1-\frac{1}{\sqrt{N}} \right)^{2}
\left| \bigcup_{\zeta \in \mathcal{F}} \mathcal{S}(\pi_{A}(\zeta), \rho_{J}|\pi_{A} \circ \zeta|) \right| \]

Write $\gothic{M}$ for the union of $\mathcal{S}(\pi_{A}(\zeta), \rho_{J}|\pi_{A} \circ
\zeta)$ as $\zeta$ ranges over $\mathcal{F}$, and $\gothic{M}_{0}$ for the subset of
$\gothic{M}$ that are $\epsilon$-efficient at scale $1/2 \epsilon^{1/4}$. The above
equation says $|\gothic{M}_{0}| \geq (1-1/\sqrt{N})^{2} |\gothic{M}|$.

Each element of $\gothic{M}_{0}$ is the $\pi_{A}$ image of a subsegment of $\phi(\mathcal{L}(\Omega)[m])$.  Let
$\mathcal{G}$ be the $\pi_{A}$ pre-images of $\gothic{M}_{0}$.  That is, $\zeta \in
\mathcal{G}$ means $\pi_{A}(\zeta) \in \gothic{M}_{0}$.  We now apply Lemma \ref{monotone scale multiple} to
$\mathcal{G}$, and again taking $N_{0}=\sqrt{N}$, to obtain scales $\rho_{I}$, $\rho_{I+1}$ and a subset $\mathcal{G}_{0}$ such that

\[ \left| \bigcup_{\gamma \in \mathcal{G}_{0}} \mathcal{S}(\gamma, \rho_{I} |\gamma|,
\mbox{$\delta$-monotone} ) \right| \geq \left( 1-\frac{1}{\sqrt{N}} \right)^{2}
\left| \bigcup_{\gamma \in \mathcal{G}} \mathcal{S}(\gamma, \rho_{I}|\gamma|) \right| \]

In other words, setting $\mathbf{L}$ as the union of $\mathcal{S}(\zeta, \hbar \rho_{I}
\rho_{J} |\zeta|)$ where $\zeta$ ranges over all $\mathcal{F}$, we have obtained a subset
$\mathbf{L}_{g}$ whose measure is at least $(1-1/\sqrt{N})^{4}$ that of $\mathbf{L}$, and
each element in $\mathbf{L}_{g}$ is $\delta$ monotone.

Recall that $\mathcal{P}=\phi(\mathcal{P}(\Omega))=\bigcup_{\zeta \in \mathbf{L}}
\mathcal{P}(\zeta)$, and for those $p \in \mathcal{P}(\Omega)$ such that $d(p, \partial
\mathbf{B}(\Omega)) \geq L_{a}/(2\kappa )$, the intersection between the union of elements in $\mathcal{F}$
and $B(\phi(p), L_{a})$ has full measure.  Since $\mathbf{B}(\Omega)$ has small boundary
area compared to its volume, and the ratio of $L_{a}$ to $L_{0}$ is a function of $\delta$ that goes to zero as $\delta$
approaches zero, we have a subset $\mathbf{L}_{0} \subset \mathbf{L}_{g}$ with relative
measure at least $1-\vartheta$ whose elements make an angle at least $\sin^{-1}(\tilde{\eta})$ with root
kernels.  Here $\vartheta$ goes to zero as $\tilde{\eta}$ and $\delta$ approach zero.

Let $\mathcal{P}_{g} \subset \mathcal{P}$ be those images coming form a point in
$\mathcal{P}(\Omega)$ at least $L_{a}/(2\kappa)$ away from $\partial \mathbf{B}(\Omega)$.
We will extract those points of $\mathcal{P}_{g}$ that are $M$ uniform with respect to at
least $s$ proportion of those elements in $\mathbf{L}_{0}$ that contain it.  We will then
choose $s$ appropriately so that the relative proportion of $\mathcal{P} -
\mathcal{P}_{0}$ is small and depends on our input data.

To begin, we note that the incident relation between $\mathcal{P}_{g}$ and $\mathbf{L}_{0}$ is symmetrical.  Moreover we know that
for any two points in $\mathcal{P}_{g}$, the ratio of numbers of elements in $\mathbf{L}_{0}$ containing each of them is bounded by
$2^{dim(\mathbf{A})}$\symbolfootnote[2]{because $\Omega$ is a product of intervals}.  For any two elements of $\mathbf{L}$, the ratio of numbers of
points in $\mathcal{P}_{g}$ lying on each of them is bounded by $m$.

For $p \in \mathcal{P}_{g}$ (resp. $\zeta \in \mathbf{L}_{0}$) write $\mathbf{Y}(p)$ (resp. $\mathcal{P}(\zeta)$) for the set of
elements in $\mathbf{L}_{0}$ (resp. $\mathcal{P}_{g}$) incident with $p$ (resp. $\zeta$).  Let $\mathcal{BP} \subset \mathcal{P}_{g}$
consisting of points that fails to be $M$-uniform with respect to at least $s$ proportion of elements in $\mathbf{Y}(p)$.

We know that for $\zeta \in \mathbf{L}_{0}$, the proportion of non $M$-uniform points is at most $\frac{2}{M}$.
Let $\chi$ denote for the characteristic function of the subset of $\{(p,\zeta): p \in \mathcal{P}_{g}, \zeta \in \mathbf{L}_{0},
p \in \zeta \}$ consisting of pairs $(p, \zeta)$ such that $p$ fails to be $M$-uniform
of $\zeta$.  Then, starting from
\[  \sum_{p \in \mathcal{P}_{g}} \mbox{      } \sum_{\zeta \in \mathbf{Y}(p)} \chi
= \sum_{\zeta \in \mathbf{L}_{0}} \mbox{        }   \sum_{p \in \zeta } \chi \]

\noindent we have

\[  s \left| \mathbf{Y}(p) \right|_{\min} |\mathcal{BP}| \leq  \sum_{p \in \mathcal{BP}} \mbox{       } \sum_{\zeta \in \mathbf{Y}(p)} \chi
= \sum_{p \in \mathcal{P}_{g}} \mbox{        }    \sum_{\zeta \in \mathbf{Y}(p)} \chi  \] \noindent and

\[ \sum_{\zeta \in \mathbf{L}_{0}} \mbox{       }   \sum_{p \in \zeta } \chi \leq \sum_{\zeta \in \mathbf{L}_{0}} \frac{2}{M} \left| \mathcal{P}(\zeta) \right|
\leq \frac{2}{M} \left| \mathbf{L}_{0}  \right| \left| \mathcal{P}(\zeta) \right|_{\max} \]

\noindent Therefore $|\mathcal{BP}| \leq \frac{2/M}{s} 2^{dim(\mathbf{A})} m |\mathcal{P}_{g}|$.  By choosing
$s=\frac{2}{N^{1/6}} 2^{dim(\mathbf{A})}m$, we have $\frac{|\mathcal{BP}|}{|\mathcal{P}|}
\leq  \frac{\tilde{\eta}}{N^{1/3}}$.  Setting $\mathcal{P}_{0}$ as
$\mathcal{P}_{g}-\mathcal{BP}$.  The desired claim now follows after Lemma \ref{how to show weakly
monotone}.    \end{proof}

We can now prove Theorem \ref{existence of good boxes}.

\begin{proof} \textit{Theorem \ref{existence of good boxes}}
We apply Proposition \ref{finite number of qigeodesics close to geodesics} to $\mathcal{L}(\Omega)$ to obtain two scales:
$\varrho_{1}=\hbar \rho_{s} \rho_{b'}$, $\varrho_{2}= \hbar \rho_{s} \rho_{b}$, and a subset $\tilde{\mathbf{L}}_{0} \subset
\mathbf{L}=\bigcup_{\gamma \in \mathcal{L}(\Omega)} \mathcal{S}(\phi(\gamma), \varrho_{2} |\phi(\gamma)|)$, such that if
$\zeta \in \tilde{\mathbf{L}}_{0}$, then $\zeta$ is within $(\eta, \frac{\varrho_{1}}{\varrho_{2}}|\zeta|)$-linear neighborhood of a
geodesic segment that makes an angle of at least $\sin^{-1}(\eta)$ with root kernel directions.

For each $\gamma \in \mathcal{L}(\Omega)$, the pre-images of $\mathcal{S}(\phi(\gamma), \varrho_{2} |\phi(\gamma)|)$ under
$\phi$ are subsegments $C_{\gamma}=\{\gamma_{i} \}$ whose union is $\gamma$, whose lengths lie between
$\frac{\varrho_{2}}{2\kappa}|\gamma|$ and $2\kappa \varrho_{2}|\gamma|$.  Furthermore, the subset $C^{0}_{\gamma}=\{
\zeta \in C_{\gamma}: \phi(\zeta) \in \tilde{\mathbf{L}}_{0} \}$ has large measure.  If for some $\gamma_{i} \in C^{0}_{\gamma}$, an
element $\zeta \in \mathcal{S}(\gamma, \frac{\varrho_{2}}{2\kappa }|\gamma|)$ satisfies $|\zeta \cap \gamma_{i}| \geq
(1-\frac{\varrho_{1}}{\varrho_{2} 2\kappa}) |\zeta|$, then $\phi(\zeta)$ is within $(\eta, \frac{\varrho_{1}}{\varrho_{2}}|\phi(\zeta)|)$-linear
neighborhood of another geodesic segment.  Since $|C^{0}_{\gamma}| \geq (1-\tilde{q}(\eta)) |C_{\gamma}|$, the subset
$D^{0}_{\gamma}=\{ \zeta \in \mathcal{S}(\gamma, \frac{\varrho_{2}}{2\kappa}): |\zeta \cap \gamma_{i}| \geq
(1-\frac{\varrho_{1}}{\varrho_{2} 2\kappa}) |\zeta|, \mbox{ for some } \gamma_{i} \in C^{0}_{\gamma} \}$ has relative measure of at least
$1-\dot{Q}(\eta)$, where $\dot{Q}(\eta) \rightarrow 0$ as $\eta \rightarrow 0$.

We now tile $\mathbf{B}(\Omega)$ by $\mathbf{B}(\frac{\varrho_{2}}{2\kappa} \Omega)$:
\[ \mathbf{B}(\Omega) = \bigsqcup_{j \in \mathbf{J}} \mathbf{B}(\Omega_{j}) \sqcup \Upsilon \]

\noindent where $\Omega_{j}=\frac{\varrho_{2}}{2\kappa} \Omega$.  Note that the union of $\bigsqcup_{j} \mathcal{L}(\Omega_{j})$ with
the subset of $\mathcal{L}(\Omega)$ consists of elements lying in $\Upsilon$ is $\mathcal{L}(\Omega)$.

Set $\mathcal{L}_{0}=\bigcup_{\gamma \in \mathcal{L}(\Omega)} D^{0}_{\gamma}$.  The
`favourable' boxes are going to be those tiling boxes that have most of their geodesics
belonging to $\mathcal{L}_{0}$.  That is, we set $\mathbf{J}_{0} =\{ j \in \mathbf{J}:
|\mathcal{L}(\Omega_{j}) \cap (\mathcal{L}(\Omega)-\mathcal{L}_{0})| \leq \tilde{q}^{1/2}
|\mathcal{L}(\Omega_{j})|\}$. \smallskip

Then,
\begin{equation*}
\tilde{q}^{1/2} |\mathcal{L}(\Omega_{j})| |\mathbf{J}-\mathbf{J}_{0}| \leq \sum_{j \in \mathbf{J}-\mathbf{J}_{0}} \tilde{q}^{1/2}
|\mathcal{L}(\Omega_{j})|  \leq \sum_{j \in \mathbf{J}} \sum_{\zeta \in \mathcal{L}(\Omega_{j})}
\chi_{\mathcal{L}(\Omega)-\mathcal{L}_{0}} = |\mathcal{L}-\mathcal{L}_{0}| \leq \theta |\mathcal{L}|
\end{equation*} \noindent Hence
\begin{equation*}
|\mathbf{J}-\mathbf{J}_{0}| \leq \tilde{q}^{1/2} \frac{|\mathcal{L}(\Omega)|}{|\mathcal{L}(\Omega_{j})|} \leq
\tilde{q}^{1/2}  |\mathbf{J}| O(\varrho_{2}) \end{equation*} \end{proof}

\section{Inside of a box}
In this section we explore the consequences of having geodesic
approximations to a large percentage of geodesic segments in a box,
and extend Theorem \ref{existence of good boxes} to the following: \smallskip

\noindent \textbf{Theorem  \ref{exisence of standard maps in small boxes}}
\textit{Let $G$, $G'$ be non-degenerate, unimodular, split abelian-by-abelian Lie groups, and
$\phi:G \rightarrow G'$ be a $(\kappa, C)$ quasi-isometry.  Given $0< \delta, \eta < \tilde{\eta} < 1$, there exist numbers $L_{0}$, $m > 1$, $\varrho, \hat{\eta} < 1$ depending on $\delta$,
$\eta$, $\tilde{\eta}$ and $\kappa, C$ with the following properties:} \smallskip

\textit{If $\Omega \subset \mathbf{A}$ is a product of intervals of equal size at least $mL_{0}$,
then a tiling of $\mathbf{B}(\Omega)$ by isometric copies of $\mathbf{B}(\varrho
\Omega)$}

\[ \mathbf{B}(\Omega)= \bigsqcup_{i \in \mathbf{I}} \mathbf{B}(\omega_{i}) \sqcup \Upsilon \]

\noindent \textit{contains a subset $\mathbf{I}_{0}$ of $\mathbf{I}$ with relative measure at
least $1-\nu$ such that}

\begin{enumerate}

\item  \textit{For every $i \in \mathbf{I}_{0}$, there is a subset $\mathcal{P}^{0}(\omega_{i}) \subset
\mathcal{P}(\omega_{i})$ of relative measure at least $1-\nu'$}

\item \textit{The restriction $\phi|_{\mathcal{P}^{0}(\omega_{i})}$ is within $\hat{\eta}
diam(\mathbf{B}(\omega_{i}))$ Hausdorff neighborhood of a standard map $g_{i} \times
f_{i}$.} \end{enumerate}

\textit{Here, $\nu$, $\nu'$ and $\hat{\eta}$ all approach zero as $\tilde{\eta}$, $\delta$ go to
zero.}

\subsection{Geometry of flats}

We now observe those geometric properties of non-degenerate, unimodular, split abelian-by-abelian groups relevant to Theorem
\ref{exisence of standard maps in small boxes}.  Specifically, this subsection explores some implications of our knowledge that a large percentage of geodesics in a box admit geodesic approximations to
its $\phi$ images.  \smallskip

\begin{lemma} \label{how to tell they are on the same flat}
Let $G$ be a non-degenerate, split abelian-by-abelian group, and $\gamma$, $\zeta$ are geodesic segments in $G$ making an angle of at least
$\sin^{-1}(\tilde{\eta})$ with root kernels such that for some $\tilde{\eta} \ll \eta < 1$, $d_{H}(\gamma, \zeta)=\eta (|\gamma|+
|\zeta|)$.  Then, $\gamma$ and $\zeta$ lie on a common flat for all but $\frac{\eta}{\tilde{\eta}}$ proportion of their lengths.
\end{lemma}

\begin{proof}
If not, then there is a root $\alpha$ such that $\pi_{\alpha}(\gamma)$ and
$\pi_{\alpha}(\zeta)$ disagrees for more than $\frac{\eta}{\tilde{\eta}}$ of their
length.  But this means that
\[ d_{H}(\pi_{\alpha}(\gamma), \pi_{\alpha}(\zeta)) > \frac{\eta}{\tilde{\eta}}
\left( \left| \pi_{\alpha}(\gamma) \right| + \left| \pi_{\alpha}(\zeta) \right| \right) \geq \eta \left( |\gamma| + |\zeta| \right) \]

\noindent which is a contradiction because $d_{H}(\gamma, \zeta) \geq
d_{H}(\pi_{\alpha}(\gamma), \pi_{\alpha}(\zeta))$.  \end{proof}

\begin{definition}
Let $G$ be a non-degenerate, split abelian-by-abelian Lie group.  We define the following
objects in $G$.
\begin{enumerate}
\item A \textbf{2-simplex} $\Delta$ is a set of three geodesic segments that intersect pair-wisely.  This includes the degenerate case of
a geodesic segment and two subsegments of it.  Elements of $\Delta$ are called \textbf{edges} of $\Delta$\symbolfootnote[2]{A 2-simplex is just a triangle.  The term `2-simplex' is used here only because it
is more convenient to describe inductive argument later on}.

\item A \textbf{filled 2-simplex} $\tilde{\Delta}$ is a set of 2-simplicies $\{ \Delta \} \cup \{ \delta_{i}\}$ such that for every
$i$, every two edges of $\delta_{i}$ are subsegments of two edges of $\Delta$. The edges of $\Delta$ are called \textbf{faces} or
\textbf{edges} of $\tilde{\Delta}$.

\begin{figure}[h]
  \centering
  \includegraphics{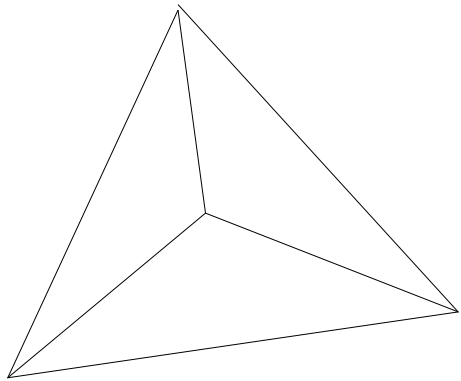}
  \caption{The big triangle together with the three little ones inside qualifies as (degenerate) 3-simplex}
\end{figure}

For $I \geq 3$, we define
\item A \textbf{I-simplex} $\Delta$ as a set of $I+1$ many filled $I-1$-simplicies such that they intersect pair-wisely at their at
their $I-2$ faces.  This includes the degenerate case of a set of $I+1$ many $I-1$-simplicies.  Elements of $\Delta$ are called
\textbf{($I-1$)-face}s of $\Delta$.

\item A \textbf{filled I-simplex} $\tilde{\Delta}$ is a collection of $I$-simplicies $\{\Delta \} \cup \{ \delta_{i} \}$ such that for
every $i$, $I$ many faces of $\delta_{i}$ are subsets of $I$ many faces of $\Delta$.  \textbf{Faces of $\tilde{\Delta}$} refers to
faces of $\Delta$.

\begin{figure}[h]
  \centering
  \includegraphics{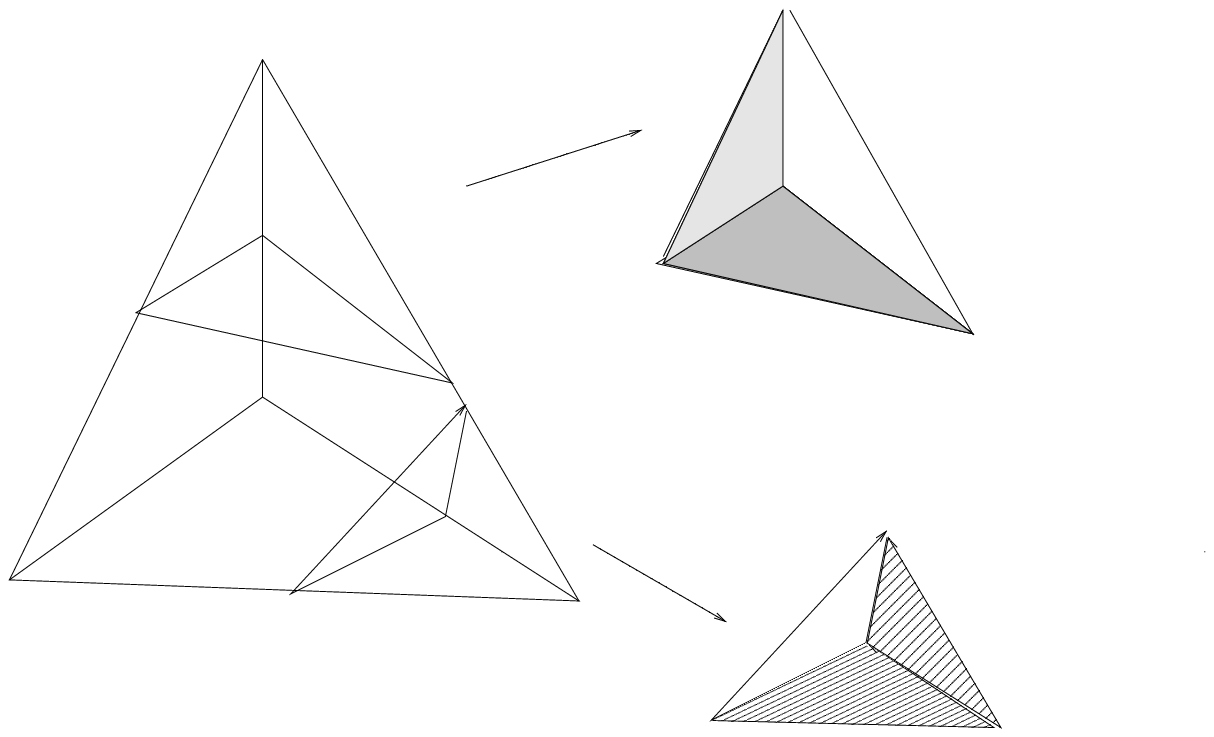}
  \caption{The big tetrahedron together with the two shaded ones qualifies as a filled 3-simplex.  Note
that a filled simplex cannot contain a degenerate simplex of higher
dimension.} \end{figure}
\end{enumerate} \end{definition}

If the faces of a simplex behaves well under the quasi-isometry $\phi$, that is, if $\phi$ images of those faces
admit approximations by hyperplanes of appropriate dimensions, then we can approximate $\phi$ image of the simplex.  This is the content
of the next lemma, which deals with one instance where simplex approximations of a quasi-simplex (image of a simplex under a
quasi-isometry) is possible. \medskip

\begin{lemma} \label{simplicies made out of good lines}
Let $G$ be a non-degenerate, split abelian-by-abelian Lie group and $\mathcal{B}$ a family of geodesic segments such that
\begin{itemize}
\item $\max \{ |\zeta|, \zeta \in \mathcal{B}\}=M \ll \infty$

\item For $\zeta \in \mathcal{B}$, $\phi(\zeta)$ is within $\eta |\phi(\zeta)|$ Hausdorff neighborhood of another geodesic
segment $\hat{\zeta}$.  We call $\hat{\zeta}$ a geodesic approximation of $\phi(\zeta)$.

\item For some $\tilde{\eta} \gg \eta$, the direction of any two geodesic approximation makes an angle
of at most $\sin^{-1}(\eta)$, and their angles each makes an angle at least $\sin^{-1}(\tilde{\eta}))$ root kernels.

\item If $\zeta, \gamma \in \mathcal{B}$, $\gamma \subset \zeta$, and let $\hat{\zeta}$, $\hat{\gamma}$ be geodesic approximations
of $\phi(\zeta)$ and $\phi(\gamma)$. Then there is a subsegment $\tilde{\zeta} \subset \hat{\zeta}$ such that $d_{H}(\tilde{\zeta},
\hat{\gamma}) \leq 2\eta |\phi(\gamma)|$.  \end{itemize}

Then for $I \leq n$, the $\phi$ images of any $I$-simplex or filled
$I$-simplex made out of elements of $\mathcal{B}$ is within $O(\eta
M)$ Hausdorff neighborhood of another simplex or filled simplex of
the same dimension lying on a flat.\end{lemma}

\begin{proof}\noindent
We prove the claims by induction on $I$, starting with a 2-simplex,
then filled 2-simplex followed by 3-simplex, filled 3-simplex etc.
\medskip

\noindent \underline{\textbf{Base step.}}

\bold{2-simplex } Fix three geodesic approximations for $\phi$
images of edges of $\Delta$, and for each weight $\Xi$, look at the
images of those geodesic approximations under $\pi_{\Xi}$. There are
six possible configurations shown in figure~\ref{fig:6config} below.  To specify a
2-simplex on a flat that is close to these three geodesics, it is enough to specify the
root space coordinates of this flat, and this is given by the root space coordinate of
the dotted line in each configuration.


\begin{figure}[h]
  \centering
  \includegraphics{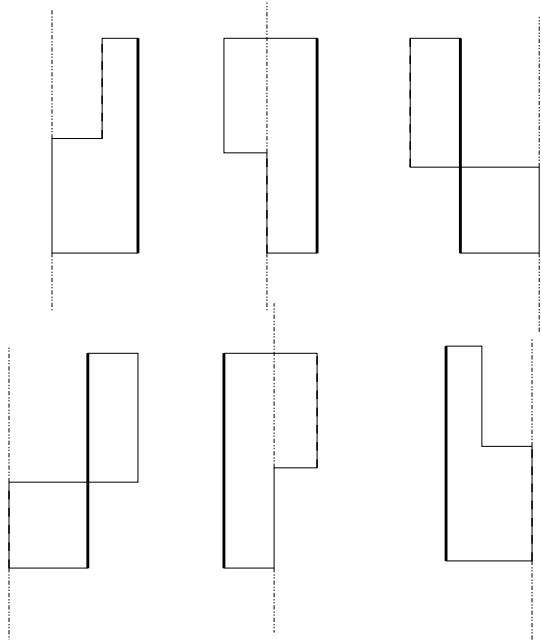}
  \caption{The six configurations in the Base step, 2-simplex case of proof to Lemma \ref{simplicies made out of good lines}.
  The dotted line is the image of a 2-simplex close to the quasi-2-simplex. }
  \label{fig:6config}
\end{figure}

\bold{filled 2-simplex.}  Let $\tilde{\Delta}=\{ \Delta \} \cup \{
\delta_{i} \}_{i}$, and $\hat{\Delta}$, $\hat{\delta}_{i}$'s denote for
the 2-simplex approximation of $\phi(\Delta)$, $\phi(\delta_{i})$'s,
as given by \textbf{2-simplex} case above.  Then for every two edges of
$\hat{\delta}_{i}$, there are subsegments of two edges of
$\hat{\Delta}$ such that each pair satisfy the hypothesis of Lemma
\ref{how to tell they are on the same flat}. This means the flats
housing $\hat{\Delta}$ and $\hat{\delta}_{i}$ must come together
(because the conclusion of Lemma \ref{how to tell they are on the same flat} says that they lie on a common flat).
Since the the set where two flats come together is convex, we conclude
therefore that there is a 2-simplex $\grave{\delta}_{i}$ lying on
the flat that houses $\hat{\Delta}$ such that
$d_{H}(\hat{\delta}_{i}, \grave{\delta}_{i}) \leq \eta M$, and two
edges of $\grave{\delta}_{i}$ are subsegments of two edges of
$\hat{\Delta}$. Then $\check{\Delta}=\{\hat{\Delta}\} \cup
\{\grave{\delta}_{i} \}_{i}$ has the desired property. \medskip

\noindent \underline{\textbf{Induction step.}}

\bold{$I$-simplex } Let $\Delta=\{ \tilde{\delta}_{i} \}_{i=0}^{I}$
where each $\tilde{\delta}_{i}$ is a filled $I-1$-simplex, and
$\check{\delta}_{i}$ be their filled $I-1$ simplex approximations as
yielded by the inductive hypothesis.  Then we know for each weight
$\Xi$, $\pi_{\Xi}(\check{\delta}_{i})$ is a vertical geodesic
segment, and for any $\check{\delta}_{i}, \check{\delta}_{j}$,
$\pi_{\Xi}(\check{\delta}_{i})$, $\pi_{\Xi}(\check{\delta}_{j})$
come together at some subsegment.  If modulo
$\frac{\eta}{\tilde{\eta}}$ proportion of the ends, $\pi_{\Xi}(\check{\delta})$'s do not lie on a common vertical geodesic
segment, then the relationship between
$\pi_{\Xi}(\check{\delta}_{i})$, $\pi_{\Xi}(\check{\delta}_{j})$ is
that of a forking $Y$, see Figure~\ref{fig:fork} below.

\begin{figure}[h]
  \centering
  \includegraphics{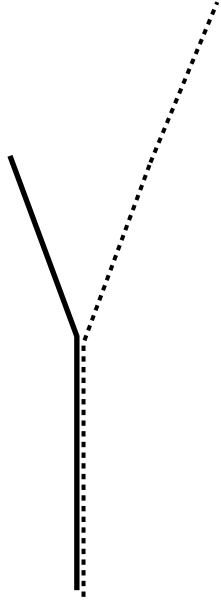}
  \caption{Inductive step, $I$-simplex case in the proof of Lemma \ref{simplicies made out of good lines}:
  the solid and dotted lines represent $\pi_{\Xi}(\check{\delta}_{i})$ and
    $\pi_{\Xi}(\check{\delta}_{j})$. }
  \label{fig:fork}
    \end{figure}

But this contradicts the existence of another $\check{\delta}_{k}$
that shares a face with $\check{\delta}_{i}$ and another face with
$\check{\delta}_{j}$.  So modulo the $\frac{\eta}{\tilde{\eta}}$
proportion of their ends, $\pi_{\Xi}(\check{\delta}_{i})$,
$\pi_{\Xi}(\check{\delta}_{j})$ must lie on a common vertical
geodesic.  The same argument applied to every other weights means
that we can translate each $\check{\delta}_{i}$ to
$\grave{\delta}_{i}$ so that $\grave{\delta}_{i}$,
$\grave{\delta}_{j}$ share a common face. The collection of all
$\grave{\delta}_{i}$'s forms our desired $\hat{\Delta}$ $I$-simplex. \footnote{The
inductive step is not a replacement of the 2-simplex case in the Base step because here,
faces intersects at filled simplex of dimension $I-2$, which has diameter compatible to
that of the diameter of the $I$-simplicies of concern, whereas in the 2-simplex case, the
pair-wise intersection of edges consist of just one point for each pair, so the same forking argument
wouldn't work there. }

\bold{filled $I$-simplex }  Let $\tilde{\Delta}=\{ \Delta \} \cup \{
\delta_{i} \}$ where each of $\Delta$ and $\delta_{i}$ is a
$I$-simplex , and let $\hat{\Delta}$ and $\hat{\delta}_{i}$'s denote
for $I$-simplex approximations of $\phi(\Delta)$ and
$\phi(\delta_{i})$'s as yielded above.  Then for every $I$ faces of
$\hat{\delta}_{i}$ there are $I$ many corresponding faces of
$\hat{\Delta}$ to which they are a subset of, and this means the
corresponding subsegments of edges of faces of $\hat{\Delta}$ and
the edges of faces of $\hat{\delta}_{i}$ satisfy the hypothesis of
Lemma \ref{how to tell they are on the same flat}, so they lie on a
common flat.  This means the flats housing $\hat{\Delta}$ and
$\hat{\delta}_{i}$ respectively must come together and since the set
where two flats come together is a convex set, we conclude therefore
that there is a $I$-simplex $\grave{\delta}_{i}$ in the flat
containing $\hat{\Delta}$ such that $d_{H}(\grave{\delta}_{i},
\hat{\delta}_{i}) \leq \eta M$, and $\grave{\delta}$ share $I$ of
its faces with faces of $\hat{\Delta}$.  Then $\check{\Delta}=\{
\hat{\Delta} \} \cup \{ \grave{\delta}_{i} \}$ has the desired
property. \end{proof} \medskip

\begin{definition} \label{quadrilateral}
Let $\eta < 1$. A \boldmath $\eta$ \unboldmath
\textbf{quadrilateral} \boldmath $Q=\{\mathbf{T}_{i} \}_{i=0}^{3}$
\unboldmath in $G$ is a set of 4 oriented geodesic segments
$\mathbf{T}_{i}$'s satisfying the following:

\begin{enumerate}
\item $\exists \vec{v} \in \mathbf{A}$ for which $W_{\vec{v}}^{0}=\{0\}$ such that
the directions of $\mathbf{T}_{i}$'s are all parallel to $\vec{v}$

\item $\forall i$, $|\mathbf{T}_{i}| > 2 \eta \sum_{j=0}^{3}
|\mathbf{T}_{j}|$

\item for all $i$, \begin{itemize} \item $d(e_{i},b_{i+1}) \leq \eta (|\mathbf{T}_{i}|
+ |\mathbf{T}_{i+1}|)$, \item  $d(b_{i}, e_{i+1}) \geq
(|\mathbf{T}_{i}| + |\mathbf{T}_{i+1}|)$  \end{itemize} where
$b_{i}, e_{i}$ are the beginning and end points of $\mathbf{T}_{i}$.
\end{enumerate} We will often refer to $\mathbf{T}_{i}$'s as edges of $Q$, and write diam$(Q)$ for
the maximum length of its edges. \end{definition}

\textbf{Example} Suppose the rank of $G$ is 1.  Let $V_{+},V_{-}$ denote for the two root
class horocycles based at the identity element.  Let $x \in V_{+}$, $y \in V_{-}$, and
the word $x y x^{-1} y^{-1}$ represents a loop in $\mathbf{H}=V_{+} \oplus V_{-}$.  If we
replace $x$ by $t \tilde{x} t^{-1}$, and $y$ by $t^{-1}\tilde{y} t$ for some small $\tilde{x}
\in V_{+}$ and $\tilde{y} \in V_{-}$, we obtain a loop representing a quadrilateral.
Note that the same construction works if $G$ is rank 1 and non-unimodular, as long as
there are two root classes.

\begin{remark} The first requirement of a quadrilateral means a quadrilateral exists in the subgroup $\langle \vec{v}
\rangle \ltimes \mathbf{H}$ (or a left translate of it).  Since $\vec{v}$ does not act trivially on any proper subspace,
quadrilaterals exist when rank of $G$ is 2 or higher for the same reason that they exist rank $1$ spaces as illustrated
by the previous example. \end{remark}

\begin{lemma} \label{orientation of a quadrilateral}
Let $Q=\{\mathbf{T}_{i}\}_{i=0}^{3}$ be a $\eta$ quadrilateral. Then
the direction of $\mathbf{T}_{i}$ and $\mathbf{T}_{i+2}$ are
positive multiple of each other, and that of $\mathbf{T}_{i}$ and
$\mathbf{T}_{i+1}$ are negative multiple of each other. \end{lemma}

\begin{proof} There are 16 possibilities to the relationship among
directions of all the $\mathbf{T}_{i}$'s (being positive or negative
multiples of each other).  One checks that only the combination
stated above is allowed.  An argument is given in the Appendix.
\end{proof} \medskip

Let $A(t)$ be a 1-parameter matrix consisting of blocks of the form $e^{\alpha t}N(t)$ where $\alpha \not= 0$, $N(t)$
a nilpotent matrix with polynomial entries, and $\mathbb{R} \ltimes_{A} \mathbb{R}^{m}$ be a
semidirect product for which $r \in \mathbb{R}$ acts on $\mathbb{R}^{m}$ by linear map $A(r)$.  Write an element of
$\mathbb{R} \ltimes_{A} \mathbb{R}^{m}$ as $(r, \mathbf{x})$, where $r \in \mathbb{R}$, $\mathbf{x} \in \mathbb{R}^{m}$, and $W^{+}$
(resp. $W^{-}$) for the direct sum of positive (resp. negative ) eigenspaces of $A$.

\begin{lemma} \label{quadrilateral word in rank 1}
In $\mathbb{R} \ltimes_{A} \mathbb{R}^{m}$, suppose for some $\eta \ll 1$, we have $r_{0},r_{1},r_{2},r_{3}>0$, $u_{0}, u_{2} \in W^{+}$,
$u_{1}, u_{3} \in W^{-}$ satisfying
\begin{itemize}
\item  $d(u_{j},e) \leq \eta (r_{j} + r_{j+1})$, $\forall j$
\item $r_{j} > 2 \eta \sum_{\iota=0}^{3} r_{\iota}$, $\forall j$
\item The word $(r_{0},0)u_{0}(-r_{1},0)u_{1}(r_{2},0)u_{2}(-r_{3},0)u_{3}$ is
trivial. \end{itemize} Then $|r_{i} - r_{i+1}| \leq d(e,u_{i+1}) + d(e,u_{i+3})$.
In particular this implies that the sizes of $r_{i}$'s are equal up
to an error of at most $\eta \sum_{i=0}^{3} r_{i}$. \end{lemma}
\begin{proof} See Appendix \end{proof}

\begin{lemma}\label{structure of a quadrilateral}
Let $Q=\{\mathbf{T}_{i}\}_{i=0}^{3}$ be a $\eta$ quadrilateral. Then

\begin{enumerate}
\item $|\mathbf{T}_{i}| -|\mathbf{T}_{j}| \leq \eta \left( \sum_{i=0}^{3} |\mathbf{T}_{i}| \right)$, $\forall i, j$

\item $\forall i$, $| \pi_{\vec{v}} \circ \Pi_{\vec{v}}(e_{i}) -
\pi_{\vec{v}} \circ \Pi_{\vec{v}}(b_{i-1}) | \leq d(e_{i},b_{i+1}) +
d(e_{i+2},b_{i+3})$

\item $\{ \Pi_{\vec{v}}(b_{i}), \Pi_{\vec{v}}(e_{i+1}), \Pi_{\vec{v}}(b_{i+2}), \Pi_{\vec{v}}(e_{i+3}) \}$ are within $\eta
\left( \sum_{i=0}^{3} |\mathbf{T}_{i}| \right)$ neighborhood of a
coset of $W^{+}_{\vec{v}}$ (or $W^{-}_{\vec{v}}$) if $i=0$(mod 2),
and of a coset of $W^{-}_{\vec{v}}$ (or $W^{+}_{\vec{v}}$)
otherwise.  \end{enumerate} \end{lemma}

\begin{proof}
Modifying $\mathbf{T}_{i}$'s by an amount of at most $\eta \sum_{j}
|\mathbf{T}_{j}|$, we can assume $\pi_{A}(e_{i})=\pi_{A}(b_{i+1})$
for all $i$.  Furthermore, the divergent assumption between $b_{i}$
and $e_{i+1}$ means that $e_{i}^{-1}(b_{i+1}) \in W^{+}_{\vec{v}}$
(resp. $W^{-}_{\vec{v}}$) if the direction of $\mathbf{T}_{i}$ is
positive (resp. negative) multiples of $\vec{v}$.  The result now
follows from Lemma \ref{quadrilateral word in rank 1}.\end{proof} \medskip

A schematic illustration for a quadrilateral with the correct
orientation and lengths for its edges is given in Figure~\ref{fig:quadrilateral} below.
\begin{figure}[htbp]
    \begin{center}

    \input{quadrilateral.pstex_t}
    \caption{A schematic illustration of a quadrilateral}
    \label{fig:quadrilateral}
    \end{center}
\end{figure}
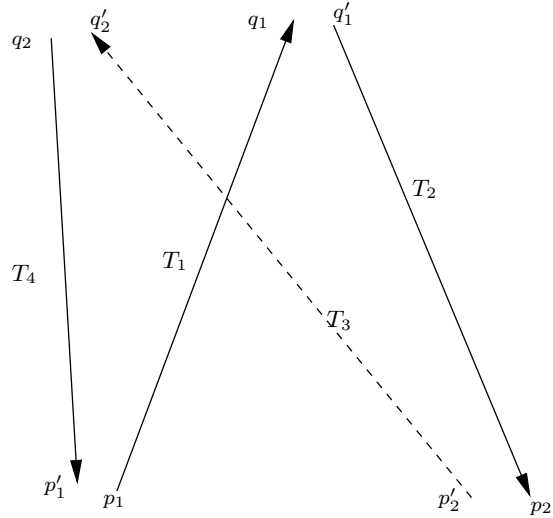

\begin{lemma} \label{close to a quadrilateral}
Let $Q=\{ \gamma_{j} \}_{j=0}^{3}$ be a $0$-quadrilateral in $G$,
such that each $\gamma_{j}$ is properly contained in a geodesic
segment $\tilde{\gamma}_{j}$, whose $\phi$ image is within $\eta
|\tilde{\gamma}|$ neighborhood of another geodesic segment whose
direction is parallel to $\vec{v}_{j} \in \mathbf{A}$ with
$W^{0}_{\vec{v}_{j}}=\{0\}$.  Suppose further that each
$|\gamma_{j}| > 2 \eta \sum_{\iota} |\tilde{\gamma}_{\iota}|$. Then,
there is a $\hat{\eta}$$(=\max\{ \eta
\frac{\tilde{\gamma}_{j}}{\gamma_{j}}\}_{j})$-quadrilateral
$\hat{Q}$ satisfying $d_{H}(\phi(Q), \hat{Q}) \leq \hat{\eta}
diam(Q)$. \end{lemma}

\begin{proof}
For each $j$, let $\tilde{T}_{j}$ be an geodesic approximation of
$\phi(\tilde{\gamma}_{j})$.  Since $Q$ is a $0$-quadrilateral,
$\tilde{\gamma}_{j} \cap \tilde{\gamma}_{j+1}$ is a geodesic segment
with positive length, therefore $\angle (\vec{v}_{j}, \vec{v}_{j+1})
\leq \sin^{-1}(\eta)$, and $d(\tilde{T}_{j}, \tilde{T}_{j+1}) \leq
\eta (|\tilde{T}_{j}| + |\tilde{T}_{j+1}| )$.  By moving each
$\tilde{T}_{j}$ by an amount at most $\sum_{\iota}
|\tilde{T}_{\iota}|$, we can assume the directions of
$\tilde{T}_{j}$'s are all parallel to some $\vec{v}$ with
$W^{0}_{\vec{v}}=\{0\}$, and $\tilde{T}_{j} \cap \tilde{T}_{j+1}$ is
a geodesic segment of positive length.  Let $\mathbf{T}_{j} \subset
\tilde{T}_{j}$ be the subsegment closest to $\phi(\gamma_{j})$. Then
$\hat{Q}=\{ \mathbf{T}_{j} \}$ is a $\hat{\eta}$-quadrilateral.
\end{proof}

%

\subsection{Averaging}
In this subsection, we put together some of the observations in the
last two subsections to show that if a large percentage of geodesic
segments in a box admit geodesic approximations to their $\phi$
images, then for $i \geq 2$, a large percentage of $i$-hyperplanes
in the box also admit $i$-hyperplane approximations to their $\phi$
images. In particular, there is a large subset of flats in the box
whose $\phi$ images are close flats. \\

The following averaging lemma that will be used repeatedly for the remaining of this section.

\begin{lemma} \label{ping-pong averaging}
Let $(A, \mu_{\alpha})$,$(B,\mu_{\beta})$ be two finite measure space, and $\sim$ is a symmetric relation between them.
For $a \in A$, write $B_{a}=\{ b \in B, b \sim a\}$ as the subset of $B$ consisted of elements related to $a$, and
$A_{b}$, for $b \in B$, as the subset of elements of $A$ related to $b$.  \smallskip
Suppose $\frac{\mu_{\beta}(B_{a})}{\mu_{\beta}(B_{a'})} \leq M_{A}$ for any $a, a' \in A$, and
$\frac{\mu_{\alpha}(A_{b})}{\mu_{\alpha}(A_{b'})} \leq M_{B}$ for any $b, b' \in B$. \smallskip

If for some $s \leq \frac{1}{M_{A}M_{B}}$, $A_{s} \subset A$ with $\mu_{\alpha}(A_{s}) \leq s \mu_{\alpha}(A)$,
then the subset $B^{s,t}=\{ b \in B: \mu_{\alpha}(A_{b} \cap A_{s}) \geq t \mu_{\alpha}(A_{b}) \}$, satisfies
$\mu_{\beta}(B^{s,t}) \leq \frac{s}{t}M_{A}M_{B} \mu_{\beta}(B)$. \end{lemma}
\begin{proof} See Appendix. \end{proof}

\begin{remark} \label{comment on ping-pong}
Lemma \ref{ping-pong averaging} will often be used to show that for subset $A_{0} \subset
A$ of relative large measure, the subset of $B$ consisting of elements $b \in B$ such
that the measure of $A_{b} \cap A_{0}$ is large relative to that of $A_{b}$, is large.  \end{remark}

%

\begin{lemma} \label{lines of big angle}
Let $\mu$ be a $\mathbb{O}(k+1)$-invariant measure on $\mathbb{O}(k+1)/
\mathbb{O}(k)$.  Let $\{ e_{\iota} \}_{\iota=1}^{k+1}$ be an orthonormal basis of $\mathbb{R}^{k+1}$ and $H_{j}$ be linear span
of $\{ e_{\iota} \}_{\iota \not= j}$, $M_{k}$ as the common value of $d(H_{j}, H_{j'})$ in $\mathbb{O}(k+1)/\mathbb{O}(k)$.
Suppose for some $\upsilon \ll 1$, $\Omega$ is a subset with $\mu(\Omega) \geq (1-\upsilon) \mu(\mathbb{O}(k+1)/\mathbb{O}(k)$.
Then we can find $k+1$ points $x_{i} \in \Omega$ such that $d(x_{i},x_{j}) \geq M_{k}-W(\upsilon)$, where $W(\upsilon)
\rightarrow 0 $ as $\upsilon \rightarrow 0$.\end{lemma}

\begin{proof}
Equip $\mathbb{O}(k+1)/\mathbb{O}(k) \times \mathbb{O}(k+1)/ \mathbb{O}(k)$ with the product measure $\mu \times \mu$.  Then
the relative measure of $\Omega \times \Omega$ is at least $(1-\upsilon)^{2}$.  If the claim was not true, then $\Omega$
is contained in a ball of radius $M_{k}$, and this would create a contradiction to the measure of $\Omega \times \Omega$
when $\upsilon$ is sufficiently small. \end{proof}

\begin{lemma} \label{bare minimum for flats to flats}
Let $G=\mathbf{H} \rtimes \mathbf{A}$, $G'=\mathbf{H}' \rtimes \mathbf{A}'$ be non-degenerate, unimodular, split
abelian-by-abelian Lie groups, and $\phi: G \rightarrow G'$ be a $(\kappa, C)$ quasi-isometry.  Let
$\Omega \subset \mathbf{A}$ be a product of intervals of equal size.  Inside of the box $\mathbf{B}(\Omega) \subset G$,
suppose for some $\eta < 1$ there is a subset $\mathcal{L}^{0} \subset
\mathcal{L}(\Omega)[m]$, where $m \rightarrow \infty$ as $\eta \rightarrow 0$, and
$|\mathcal{L}^{0}|\geq (1-F_{1})|\mathcal{L}(\Omega)[m]|$, where $F_{1}$ is a function of $\eta$ and
approaches zero as $\eta \rightarrow 0$, such that for every $\mathit{l} \in \mathcal{L}^{0}$,

\begin{enumerate}
\item $\phi(\mathit{l})$ is within $\eta |\mathit{l}|$ Hausdorff neighborhood of a geodesic segment that makes an angle at least $\sin^{-1}(\tilde{\eta})$ with root kernel directions.
Here $\tilde{\eta}$ depends on $\eta$ and approaches zero as $\eta \rightarrow 0$.

\item For each $\mathit{l} \in \mathcal{L}^{0}$, the proportion of $\eta$ uniform points
is at least $1-F_{1}$. \end{enumerate}  Then, for $i=2,3, \cdots dim(\mathbf{A})$, there
are subsets $\mathcal{L}^{0}_{i} \subset \mathcal{L}_{i}(\Omega)[m]$ and functions $F_{i}$,
together with a subset $\mathcal{P}^{0} \subset \mathcal{P}(\Omega)$ and a function $F_{0}$ that satisfy the
the following properties.

\begin{enumerate}
\item $F_{i}$'s and $F_{0}$ are functions of $\eta$ and approach zero as $\eta \rightarrow 0$.

\item If $S \in \mathcal{L}_{i}^{0}$, then $\phi(S)$ is within $\eta \mbox{diam}(\mathbf{B}(\Omega))$ Hausdorff neighborhood of a
$i$-dimensional hyperplane.


\item For every $p \in \mathcal{P}^{0}$,
\[
\left| L(p) \cap \mathcal{L}(\Omega)[m] \cap \mathcal{L}^{0} \right|   \geq  (1-F_{1})^{2}
\left| L(p) \cap \mathcal{L}(\Omega)[m] \right| \] \noindent and $i=2,3, \cdots$,
\[
\left| L_{i}(p) \cap \mathcal{L}_{i}(\Omega)[m] \cap \mathcal{L} \right| \geq (1-F_{i})^{2}
\left| L_{i}(p) \cap \mathcal{L}_{i}(\Omega)[m] \right| \]

\item  The relative measure of $\mathcal{L}_{i}^{0}$ and $\mathcal{P}^{0}$ in
$\mathcal{L}_{i}(\Omega)[m]$ and $\mathcal{P}(\Omega)$ are at least $1-F_{i}$ and
$1-F_{0}$ respectively.  \end{enumerate} \end{lemma}

\begin{proof}
More precisely, we prove the following claims:

For $i=2,3, \cdots dim(\mathbf{A})$, there are subsets $\mathcal{L}^{0}_{i} \subset
\mathcal{L}(\Omega)[m]$, $\mathcal{P}_{i} \subset \hat{\mathcal{P}}_{i}$ of $\mathcal{P}(\Omega)$, all of relative
large measure such that

\begin{enumerate}\renewcommand{\labelenumi}{\alph{enumi}.}
\item elements of $\mathcal{L}^{0}_{i-1}$ admit $i-1$ hyperplane approximations.
$\mathcal{L}^{0}_{1}$ is defined to be $\mathcal{L}^{0}$.

\item if $S \in \mathcal{L}^{0}_{i}$, $p \in \mathcal{P}_{i}$, $p \in
S$, then a large proportion of elements in $L_{i-1}(p) \cap L_{i-1}(S)$ lies in
$\mathcal{L}^{0}_{i-1}$.  Here, `large' means closer to $1$ as $\eta \rightarrow 0$.

\item Elements of $\mathcal{L}^{0}_{i}$ admit $i$-hyperplane approximations.  \end{enumerate}

The proof proceeds in two steps.  First, we construct $\mathcal{L}_{i}^{0}(\Omega)$ and
subsets $\mathcal{P}_{i} \subset \mathcal{P}(\Omega)$ by induction.  Then we use
$\mathcal{P}_{i}$ to show that elements of $\mathcal{L}_{i}^{0}$ satisfies the desired properties.  The set
$\mathcal{P}^{0}$ will be the intersection of those $\mathcal{P}_{i}$ from the first step.

We start with the base case when $i=2$.

The incidence relation between $\mathcal{L}(\Omega)[m]$ and $\mathcal{L}_{2}(\Omega)[m]$ is
symmetrical.  Therefore, by Lemma \ref{ping-pong averaging} we can choose $s_{2}(\eta) \ll
1$ appropriately so that the set

\[ \mathcal{L}_{2}^{0}= \{ S \in \mathcal{L}_{2}(\Omega)[m]: \left|  L_{1}(S) \cap \mathcal{L}(\Omega)[m] \cap \mathcal{L}^{0} \right|
\geq (1-s_{2}) \left| L_{1}(S) \cap \mathcal{L}(\Omega)[m] \right| \} \]

\noindent satisfies

\[ \left| \mathcal{L}_{2}^{0} \right| \geq (1-F_{1}^{1/2}) \left| \mathcal{L}(\Omega)[m]
\right| \]

Fix a $S \in \mathcal{L}_{2}^{0}$.  Let $P(S)^{bad} \subset P(S)$ consisting of those points $p$ such that $p$ fails
to be uniform with respect to at least $s_{b}$ proportion of elements in $L(p) \cap L(S) \cap \mathcal{L}(\Omega)[m]$.
Note that this means if $\zeta$ is not an element of $L(p) \cap L(S) \cap
\mathcal{L}^{0}$, then $p$ is not uniform with respect to $\zeta$.  We obtain a bound on the relative size of $P(S)^{bad}$ as
follows.

Let $\chi$ be the characteristic function of the subset of $\{(p, \zeta): p \in P(S), \zeta \in L(S) \cap \mathcal{L}(\Omega)[m],
p \in \zeta \}$ consisting of pairs $(p, \zeta)$ such that either $\zeta \not \in \mathcal{L}^{0}$ or
$\zeta \in \mathcal{L}^{0}$ but $p$ fails to be a uniform point on it.

Then, starting from

\[ \sum_{x \in P(S)} \mbox{     }\sum_{\gamma \in L_(x) \cap L(S) \cap \mathcal{L}(\Omega)[m]} \chi
=   \sum_{\gamma \in L_(S) \cap \mathcal{L}(\Omega)[m]} \mbox{      } \sum_{p \in P(\gamma) } \chi \]

\noindent we have
\begin{eqnarray*}
\sum_{x \in P(S)} \mbox{      } \sum_{\gamma \in L(x) \cap L(S) \cap \mathcal{L}(\Omega)[m] } \chi & \geq &
\sum_{x \in P(S)^{bad}} s_{b} \left| L(x) \cap L(S) \cap \mathcal{L}(\Omega)[m] \right| \\
& \geq & s_{b} \left| P(S)^{bad} \right| \mbox{    } \left| L(x) \cap L(S) \cap \mathcal{L}(\Omega)[m] \right|_{\min, x \in S}  \end{eqnarray*}

\noindent and
\begin{eqnarray*}
\sum_{\gamma \in L(S) \cap \mathcal{L}(\Omega)[m]} \mbox{          }\sum_{p \in \gamma } \chi & = &
\sum_{\gamma \in L(S) \cap \left(\mathcal{L}(\Omega)[m]-\mathcal{L}^{0} \right) } \mbox{          } \sum_{p \in \gamma } \chi +
\sum_{\gamma \in L(S) \cap \mathcal{L}^{0}} \mbox{     } \sum_{p \in \gamma} \chi \\ & \leq &
\sum_{\zeta \in L(S) \cap \mathcal{L}(\Omega)[m]} F_{1} \left| P(\zeta) \right| +
\sum_{\gamma \in L(S) \cap \mathcal{L}^{0}} F_{1} \left| P(\gamma) \right| \\ & \leq &
\sum_{\zeta \in L(S) \cap \mathcal{L}(\Omega)[m]} F_{1} \left| P(\zeta) \right| +
\sum_{\gamma \in L(S) \cap \mathcal{L}(\Omega)[m]} F_{1} \left| P(\gamma) \right| \\ & \leq &
2F_{1} \sum_{\zeta \in L(S) \cap \mathcal{L}(\Omega)[m]} \left| P(\zeta) \right| \\
& \leq &  2F_{1} \left| L(S) \cap \mathcal{L}(\Omega)[m] \right| \mbox{  } \left| P(\zeta) \cap P(S) \right|_{\max, \zeta \in \mathcal{L}(\Omega)[m]} \end{eqnarray*}

\noindent which yields
\[ \left| P(S)^{bad}  \right|  \leq \frac{2F_{1}}{s} k \left|    P(s)  \right| \] \noindent where $k$ depends only on $G$.
By choosing $s_{b}=2 F_{1}^{1/2}$, we have the measure of $P(S)^{bad}$ is at least
$1-F_{1}^{1/2}$ times that of $P(S)$.

We now apply Lemma \ref{ping-pong averaging} to $P(S)$, $L(S) \cap
\mathcal{L}(\Omega)[m]$, and $P(S)^{bad}$ to conclude that for some $\nu_{2} \ll 1$, the subset

\[ L(S)^{bad}=\{ \zeta \in L(S) \cap \mathcal{L}(\Omega)[m]: \left| P(\zeta) \cap
P(S)^{bad} \right| \geq \nu_{2} \left| P(\zeta) \right| \} \] satisfies

\[ \left| L(S)^{bad} \right| \leq \nu'_{2} \left| L(S) \right| \] \noindent for some
$\nu'_{2} \ll 1$.

Now apply Lemma \ref{ping-pong averaging} again to $P(S)$, $L(S) \cap
\mathcal{L}(\Omega)[m]$, and $L(S)^{good}=(L(S)-L(S)^{bad}) \cap \mathcal{L}^{0}$ to conclude that for some $\hat{\nu}_{2} \ll
1$, the subset

\[ P(S)^{w}= \{ p \in P(S): \left| L(p) \cap \mathcal{L}(\Omega)[m] \cap L(S)^{good}  \right|
\leq (1-\hat{\nu}_{2}) \left| L(p) \cap \mathcal{L}(\Omega)[m] \right| \} \] \noindent satisfies

\[ \left| P(S)^{w} \right| \leq \tilde{\nu}_{2} \left| P(S) \right| \] \noindent for some
$\tilde{\nu}_{2} \ll 1$.  Now set $P(S)^{0}$ as $P(S)-P(S)^{bad} - P(S)^{w}$, and let
$\mathcal{P}_{2}$ as the union of $P(S)^{0}$ as $S$ ranges over $\mathcal{L}^{0}_{2}$.

Now run the same argument for $i=3$, replacing `uniform points' of an element of
$\mathcal{L}^{0}$ by $P(S)^{0}$, where $S \in \mathcal{L}_{2}^{0}$.  Repeat this
procedure inductively, to arrive at subsets $\mathcal{L}^{0}_{i}$, and $P(S)^{0} \subset
P(S)$ for every $S \in \mathcal{L}^{0}_{i}$, all of relative large measure.

For a $S \in \mathcal{L}^{0}_{i+1}$, we now show that $\phi(P(S)^{0})$ is close to a
$i+1$-dimensional hyperplane.  We will do this by induction.  The hypothesis furnishes the base
step.

Take $p, q \in P^{good}(S)$.  By construction, for $\hat{\nu}_{i}, \mu_{i} \ll 1$, the $\phi$ images of at least $1-\hat{\nu}_{i}$
proportion of elements in $L_{i}(p)$ and $L_{i}(q)$ have the properties that 1) spend at least $1-\mu_{i}$ proportion of their area/measure in
$P^{0}(S)$, and 2) belong to $\mathcal{L}_{i}^{0}$, so admit $i$-hyperplane approximations by inductive hypothesis.

There are two cases to consider.

\bold{Case I.} At least one of $p,q$ is at least $\eta diam(\mathbf{B}(\Omega))$ away from $\partial \mathbf{B}(\Omega)$, in which case
we can do one of the followings: (see also Figure~\ref{fig:Case1} below.)
\begin{itemize}
\item find $i$ many points $r_{\iota} \in P^{0}(S)$, $\iota=1,2, \cdots i$
and $Q_{p} \in L_{i}(p) \cap L_{i}(S) \cap \mathcal{L}_{i}^{0}$, $Q_{r_{\iota}} \in L_{i}(r_{\iota}) \cap L_{i}(S) \cap \mathcal{L}_{i}^{0}$
such that they intersect to form a $i+1$-simplex $\Delta$ with $p,q$ and $r_{\iota}$'s lying on its faces.

\item or pick an element $Q_{p} \in L_{i}(p) \cap L_{i}(S) \cap \mathcal{L}_{i}^{0}$.  Since $Q_{p}$ has codimension 1 in $S$,
most elements of $L_{i}(q) \cap L_{i}(S) \cap \mathcal{L}^{0}_{i}$ will intersect it and we can find $i$ many elements
$Q_{q,\iota} \in  L_{i}(q) \cap L_{i}(S) \cap \mathcal{L}^{0}_{i}$, $\iota=1,2,\cdots i$ such that they intersect $Q_{p}$ to make a
$i+1$-simplex $\Delta$ with $q$ being one of its vertices and $p$ lying on the face opposite to $q$. \end{itemize}

\begin{figure}[htbp]
\begin{center}

\input{mak-Simplex1.pstex_t}
\caption{Case I: two ways of making a simplex when $p,q$ are far from the boundary of the box.}
\label{fig:Case1}
\end{center} \end{figure}
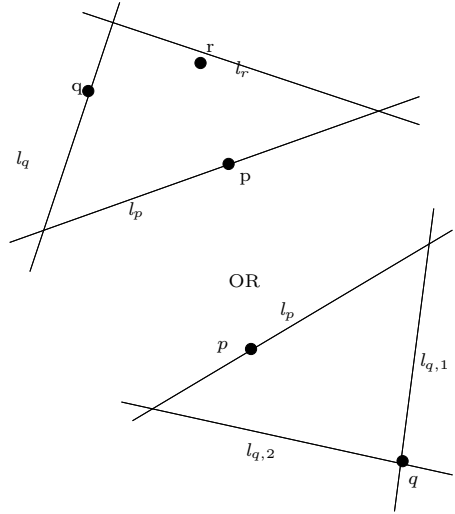

\noindent We now apply Lemma \ref{simplicies made out of good lines} to conclude that the $\phi$ images of $\Delta$ are within $\eta
\mbox{diam}(\mathbf{B}(\Omega))$ neighborhood of another $i+1$-simplex on a $i+1$-dimensional hyperplane.

\bold{Case II.}  Both $p$ and $q$ within $\eta diam(\mathbf{B}(\Omega))$ of $\partial \mathbf{B}(\Omega)$.


In this case we make a $i+1$-simplex with $p$ as one of its vertex as follows. (see also Figure~\ref{fig:Case2} below)
Apply Lemma \ref{lines of big angle} to subsets $L_{i}(p) \cap L_{i}(S) \cap \mathcal{L}_{i}^{0}$, which allows us to pick out $i+1$
elements $Q_{p,\iota} \in L_{i}(p) \cap L_{i}^{00}(S)$ that are almost equally spaced apart
(up to an error of $W(\eta)$ by Lemma \ref{lines of big angle}).

Since each $Q_{p,\iota}$ spends at least $1-\mu_{i}$ proportion of its measure in the set $P^{0}(S)$, we can certainly find
$x \in Q_{p,1} \cap P^{0}(S)$.  Furthermore we can assume $x$ is at most $O(\eta diam(\mathbf{B}(\Omega))$ away from
$\partial \mathbf{B}(\Omega)$.

The subset of $L_{i}(x)$ that intersect all of $Q_{p,\iota}$'s, for $\iota=1,2,\cdots i$ has large positive measure because elements
of $L_{i}(x)$ has codimension 1 in $S$.  So we can find $Q_{x} \in L_{i}(x) \cap L_{i}^{0}(S)$ that intersects all of $Q_{p,\iota}$'s,
thus making a $i+1$-simplex $\Delta$.  By choice, faces of $\Delta$: $Q_{p,1}, Q_{p,2} \cdots Q_{p,i},Q_{x}$, when considered as points
in $\mathbb{O}(i+1)/\mathbb{O}(i)$, have pair-wise distance at least $M_{i}-W(\eta)$, which means the volume of the set bounded by
$\Delta$ is at least $\frac{1}{2^{i}}$ proportion of the volume of $S$.

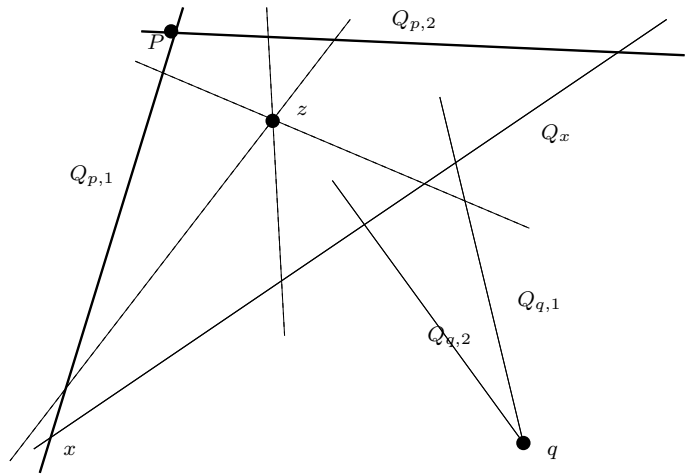
\begin{figure}[htbp]
\begin{center}

\input{mak-Simplex2.pstex_t}
\caption{Case II: when either $p$ or $q$ is too close to the boundary of the box, we first make a filled simplex using $p$ and look
at the intersection between $i$-hyperplanes in $L_{i}(q)$ that intersect this filled simplex.}

\label{fig:Case2}
\end{center}
\end{figure}

Let $z \in P^{0}(S)$ be a point that lies in the interior of the set bounded by $\Delta$.  Then most elements of
$L_{i}(z) \cap L_{i}(S) \cap \mathcal{L}_{i}^{0}$ are going to intersect $i+1$ faces of $\Delta$ thus making a smaller
$i+1$-simplicies, $i$ of its faces are subsets of faces of $\Delta$. We construct such $i+1$-simplicies $\delta_{i}$ for all points in
$P^{0}(S)$, and the collection of them together with $\Delta$ gives us $\tilde{\Delta}=\{ \Delta \} \cap \{ \delta_{i} \}$  a filled
$i+1$-simplex. \smallskip

By Lemma \ref{simplicies made out of good lines}, $\phi$ image $\tilde{\Delta}$ is within $\eta \mbox{diam}(\mathbf{B}(\Omega))$ Hausdorff
neighborhood from another filled $i+1$-simplex $\check{\Delta}$ on $i+1$ hyperplane, We are done if in addition, $q \in
\tilde{\Delta}$.  If not, then we can find two elements $Q_{q,1}, Q_{q,2} \in L_{i}(q) \cap L_{i}(S) \cap \mathcal{L}_{i}^{0}$ such that they both have no empty
intersection with $\tilde{\Delta}$, because the area of the set bounded by $\Delta$ to that of $S$ is at least $1/2^{i}$. \smallskip

Let $\hat{Q}_{q,1}$, $\hat{Q}_{q,2}$ be $i$-hyperplane approximations to $\phi(Q_{q,1})$ and $\phi(Q_{q,2})$.   Then for
any root $\Xi$, $\pi_{\Xi}$ images of $\hat{Q}_{q,1}$ and $\hat{Q}_{q,2}$ on the ends away from $\pi_{\Xi}(\phi(q))$ lie on a
common vertical geodesic segment because they both intersect $\check{\Delta}$ which lie on a $i+1$-hyperplane, and on the
$\pi_{\Xi}(\phi(q))$ end, the vertical geodesic segment containing them come together because both $Q_{q,i}$'s contain $q$.   Since any
two geodesic segments in a hyperbolic space come together at most one end, this means $\pi_{\Xi}(\hat{Q}_{q,1})$ and
$\pi_{\Xi}(\hat{Q}_{q,2})$ lie on a common vertical geodesic segment. As $\Xi$ ranges over all roots, this means that
$\hat{Q}_{q,1}$ and $\hat{Q}_{q,2}$ lie on the same flat as $\check{\Delta}$.  Lastly, as $\check{\Delta}$ lie on a $i+1$-hyperplane and each of $\hat{Q}_{q,i}$'s is a $i$-hyperplane, this gives us $\check{\Delta}
\cup \{ \hat{Q}_{q,i} \}$ lie on a common $i+1$-hyperplane within a flat.  \end{proof}


\subsection{Proof of Theorem \ref{exisence of standard maps in small boxes}}

\begin{proof} \textit{of Theorem \ref{exisence of standard maps in small boxes} }

Apply Theorem \ref{existence of good boxes} to $\mathbf{B}(\Omega)$.  Take a $\mathbf{B}(\Omega_{j})$, $j \in \mathbf{J}_{0}$
and apply Lemma \ref{bare minimum for flats to flats} to obtain subsets
$\mathcal{L}^{0} \subset \mathcal{L}(\Omega_{j})[m]$, $\mathcal{L}^{0}_{\iota}  \subset \mathcal{L}_{\iota}(\Omega_{j})[m]$ for
$\iota=2,3, \cdots rank(G)$, and $\mathcal{P}^{0} \subset \mathcal{P}(\Omega_{j})$, all
with relative measures approaching 1 as $\eta$, $\delta$ approach zero, such that if
$\zeta \in \mathcal{L}^{0}$, then $\phi(\zeta)$ is within $2\kappa \eta|\zeta|$ Hausdorff
neighborhood of a geodesic segment that makes an angle at most $\sin^{-1}(\tilde{\eta})$
with root angles. While when $S$ is an element of $\mathcal{L}^{0}_{\iota}$ for some $\iota = 2,3,
\cdots rank(G)$, $\phi$ images of the subset of $\mathcal{P}^{0}$ lying in $S$ are within
$\eta diam(\mathbf{B}(\Omega_{j}))$ of a hyperplane of appropriate dimension.  This means
that the restriction of $\phi|_{\mathbf{B}(\Omega_{j})}$ to the subset $\mathcal{P}^{0}$
sends left cosets of $\mathbf{A}$ to left cosets of $\mathbf{A}'$ up to an error of $\eta diam(\mathbf{B}(\Omega_{j}))$.
\smallskip

From now on we drop the subscript $j$. Let $\mu=(\tilde{\eta})^{1/2}$ and tile $\mathbf{B}(\Omega)$ by
$\mathbf{B}(\mu \Omega)$:
\[ \mathbf{B}(\Omega)=\bigsqcup_{i \in \mathbf{I}} \mathbf{B}(\omega_{i}) \cup \Upsilon
\]
\noindent By Lemma \ref{boxes are folner}, we can assume each of the tiles $\mathbf{B}(\omega_{i})$ is at least
$\mu diam(\mathbf{B}(\Omega))$ away from the boundary of $\mathbf{B}(\Omega)$, and the measure of $\Upsilon$ is at most $O(\tilde{\eta})$ times
that of $\mathbf{B}(\Omega)$.

By Chebyshev inequality and Lemma \ref{ping-pong averaging} we can obtain a subset $\mathbf{I}_{0} \subset \mathbf{I}$ with
$|\mathbf{I}_{0}| \geq (1-\varsigma) |\mathbf{I}|$ such that for every $i \in \mathbf{I}_{0}$, there are subsets $\mathcal{L}^{0}(\omega_{i})$,
$\mathcal{L}^{0}_{rank(G)}(\omega_{i})$ and $\mathcal{P}^{0}(\omega_{i})$ of $\mathcal{L}(\omega_{i})$, $\mathcal{L}_{rank(G)}(\omega_{i})$,
and $\mathcal{P}(\omega_{i})$, all of relative measure at least $1-\upsilon$ whose elements are restriction of $\mathcal{L}^{0}$, $\mathcal{L}^{0}_{rank(G)}$
and $\mathcal{P}^{0}$ to $\mathbf{B}(\omega_{i})$.  Here, $\varsigma$ and $\upsilon$ both go to zero as $\tilde{\eta} \rightarrow 0$.

Take a $\mathbf{B}(\omega_{i})$, $i \in \mathbf{I}_{0}$.  Then the restriction of
$\phi|_{\mathbf{B}(\omega_{i})}$ to $\mathcal{P}^{0}(\omega_{i})$ sends flats to within $\frac{\eta}{\mu} diam(\mathbf{B}(\omega_{i}))$
Hausdorff distance of a flat.  Note that $\eta < \tilde{\eta}< 1$, so $\frac{\eta}{\mu}
\ll 1$ and approaches zero when $\tilde{\eta} \rightarrow 0$.  Since two flats come
together at a convex set whose boundary is a union of hyperplanes parallel to root
kernels.

To obtain a product structure on $\mathcal{P}^{0}$, we proceed to show that $\phi|_{f}$ and $\phi|_{f'}$ for $f, f' \in
\mathcal{L}^{0}_{rank(G)}$ are identical up to a translational error of $\eta diam(\mathbf{B}(\Omega_{j}))$.  In the process of doing so,
we will also show that left cosets of $\mathbf{H}$ are sent to left cosets of $\mathbf{H}'$ up to
an error of the same order. \smallskip

First we show that the claim is true for two flats $f,f' \in \mathcal{L}^{0}_{rank(G)}(\omega_{i})$
that are at least $8 \frac{\eta}{\mu} diam(\mathbf{B}(\omega_{i}))$ units apart and
contains points $p \in f \cap \mathcal{P}^{0}(\omega_{i})$, $p' \in f' \cap
\mathcal{P}^{0}(\omega_{i})$ such that $p,p'$ lie on a common root class horocycle.
\smallskip

Since  $p, p' \in \mathcal{P}^{0}(\omega_{i}) \subset \mathcal{P}^{0}$, we can find geodesic
segments $\mathit{l}_{p,1}, \mathit{l}_{p,2} \in \mathcal{L}^{0}(\Omega)$ containing $p$,
$\mathit{l}_{q,1}, \mathit{l}_{q,2} \in \mathcal{L}^{0}(\Omega)$ containing $q$ such that for some subsegments
$\hat{\mathit{l}}_{*,\iota} \subset \mathit{l}_{*,\iota}$, $* =p,q$, $\iota=1,2$,
$Q=\{\hat{\mathit{l}}_{p,\iota}, \hat{\mathit{l}}_{q,\iota} \}_{\iota=1,2}$ is a $0$
quadrilateral. \medskip

As $d(p,p) \geq 8 \frac{\eta}{\mu} diam(\mathbf{B}(\omega_{i}))$, by Lemma \ref{close to a quadrilateral}, there is a $\eta$ quadrilateral
$\hat{Q}$ within $\eta diam(\mathbf{B}(\Omega_{j}))$ (i.e.$\frac{\eta}{\mu} diam(\mathbf{B}(\omega_{i}))$) Hausdorff distance away from
$\phi(Q)$.  Applying Lemma \ref{structure of a quadrilateral} to $\hat{Q}$, we see that $\phi(p)$ and $\phi(p')$ are within $\frac{\eta}{\mu} diam(\mathbf{B}(\omega_{i}))$ neighborhood of a left translate of
$W^{+}_{\vec{v}}$ or $W^{-}_{\vec{v}}$ where $\vec{v}$ is the direction of edges of $\hat{Q}$.  Since $p, p' \in \mathcal{P}^{0}(\omega_{i})$, we can
build quadrilaterals $Q_{1}, Q_{2}, \cdots Q_{k}$ for $k \leq n+2$, the edges of each are elements of $\mathcal{L}^{0}(\Omega)$ such that
their respective approximating quadrilaterals $\hat{Q}_{1}, \hat{Q}_{2}, \cdots \hat{Q}_{k}$, with edge directions
$\vec{v}_{1}, \vec{v}_{2}, \cdots \vec{v}_{k}$ satisfies $\cap_{\iota=1}^{k}
W^{\sigma(\iota)}_{\vec{v}_{\iota}}$ with $\sigma(\iota) \in \{+,- \}$, is $V_{[\alpha]}$ for some root class $[\alpha]$.  Argue as before,
we see that $\phi(p)$ and $\phi(q)$ lie within $\frac{\eta}{\mu} diam(\mathbf{B}(\omega_{i}))$ Hausdorff neighborhood of a
translate $W^{\sigma(\iota)}$ for $\iota=1,2, \cdots k$, therefore $\phi(p)$ and $\phi(q)$ lie within
$\frac{\eta}{\mu} diam(\mathbf{B}(\omega_{i}))$ Hausdorff neighborhood of a translate of $V_{[\alpha]}$.

By using more quadrilaterals, the argument above also shows that $\phi|_{f \cap
\mathcal{P}^{0}(\omega_{i})}$ are the same as $\phi|_{f' \cap \mathcal{P}^{0}(\omega_{i})}$ up to an
error of $\frac{\eta}{\mu} diam(\mathbf{B}(\omega_{i}))$.

In general, for two arbitrary points $p,p' \in \mathcal{P}^{0}_{i}$ in the same left coset of $\mathbf{H}$, we can find at most
$|\triangle|$ number of points $p_{0}=p, p_{1}, p_{2}, \cdots p_{l}=p'$, such that each pair of successive points lie on a common
root class horocycle.  The quadrilateral argument above then shows that $\phi(p)$ are $\phi(q)$ within
$|\triangle| \frac{\eta}{\mu} diam(\mathbf{B}(\omega_{i}))$ Hausdorff neighborhood of a translate
of $\mathbf{H}'$. \end{proof}



\section*{Appendix}

\begin{proof} \textit{of Lemma \ref{can't move far in H2} }
We will use the notations from equation (\ref{distance}).
Write $p=(x,t)$, $q=(x',t')$. By assumption, $|t - t'| \leq s$.  If $U(|x-x'|) \leq \min\{t,t'\}$, then assume $t \geq t'$
\begin{equation*} d((x,t),(x',t')) \leq d((x,t),(x',t)) +
d((x',t),(x',t')) \leq 2(t-t')+1 \leq 3s \end{equation*} and we are done. \smallskip

Now suppose $U |x-x'| \geq t,t'$, but $U |x-x'| \leq 4s$, then
\begin{eqnarray*}
d((x,t),(x',t')) &\leq & d((x,t),(x, U(|x-x'|))) + d((x, U(|x-x'|)), (x',U(|x-x'|))) \\
& & + d((x',U(|x-x'|)),(x',t')) \leq 2 U(|x-x'|) - (t+ t')+1 \\
   & & \leq 8s+1 \leq 12 \kappa s \end{eqnarray*} and we are done. \smallskip

Finally suppose $U(|x-x'|) \geq t,t'$, and $U(|x-x'|) \geq 4s$.  Since $\eta$ is continuous, we can find
$i_{0} \leq i_{1} \leq i_{2} \leq i_{3}..... i_{n} \in [a,b]$ and therefore points $\{p_{j}=\eta(i_{j})\}_{j=1}^{n}$ such
that $p=(x,t)=\eta(i_{0})=p_{0}$, $p_{n}=\eta(i_{n})=q=(x',t')$, and $U(|x_{j} - x_{j+1}|) =4s$, for all $j$ except maybe
the last one, where $U(|x_{n-1}-x_{n}|) \leq 4s$.

Then by equation (\ref{distance}) \begin{equation*}
\frac{ \sum_{j=0}^{n-1} \left(U(|x_{j}-x_{j+1}|)-(t_{j}+t_{j+1})\right)}{\left( U(|x_{0}-x_{n}|)-(t_{0}+t_{n}) \right)} \leq
\frac{ \sum_{j=0}^{n-1} d(p_{j},p_{j+1})}{d(p_{0},p_{n})} \leq 2 \kappa \end{equation*}

Simplifying using equation (\ref{property of U}) yields
\begin{equation*} \frac{(n-1) 2s}{2\ln(ne^{4s})} \leq 2 \kappa \end{equation*}

\noindent which means \begin{eqnarray*}
(n-1)s  &\leq& 2 \kappa \left( \ln(n) + 4s \right) \\
ns-2 \kappa \ln(n) &\leq& s  + 8 \kappa s \\
\frac{1}{2}n s \leq n s- 2s\ln(n) \leq n s- 2\kappa \ln(n)
&\leq& s + 8 \kappa s \leq 9 \kappa s \\
n &\leq& 20 \kappa \end{eqnarray*} So
\begin{eqnarray*}d(p_{0},q_{0}) \leq \sum_{j=0}^{n-1} d(p_{j},p_{j+1})
&\leq& \sum_{j=0}^{n-1} \left(U(|x_{j}-x_{j+1}|) -(t_{j}+t_{j+1}) \right) \\
&\leq & 20 \kappa s = 80 \kappa s  \end{eqnarray*}\end{proof}

\begin{proof} \textit{of Lemma \ref{mixing}  }
The claim is clear if $c_{\alpha}=1$. Otherwise we know
\begin{equation*} \frac{c_{\alpha}}{c_{\beta}} = \frac{\left| \frac{b}{B}
- \frac{a+b}{A+B} \right|}{\left| \frac{a}{A} - \frac{a+b}{A+B}
\right|} \end{equation*} $c_{\alpha} \geq c_{\beta}$ therefore gives
us that
\begin{equation*} \left| \frac{a}{A} - \frac{a+b}{A+B} \right| \leq \left| \frac{b}{B} - \frac{a+b}{A+B} \right|
\end{equation*}
\begin{itemize}
\item Suppose $\frac{b}{B} < \frac{a}{A}$. Writing $b=c_{1}a$, $B=c_{2}A$, we have
\begin{eqnarray*}
1- \frac{1+c_{1}}{1+c_{2}} &<& \frac{1+c_{1}}{1+c_{2}}
-\frac{c_{1}}{c_{2}} \\
1+ \frac{c_{1}}{c_{2}} & < & 2 \left( \frac{1+c_{1}}{1+c_{2}}
\right) \\
c_{2}(1+c_{2}) + c_{1}(1+c_{2}) &<& 2(1+c_{1})c_{2} \\
c_{2} + c_{2}^{2} + c_{1} + c_{1}c_{2} &\leq& 2c_{2} + 2c_{1}c_{2}
\\ c_{2}(c_{2}-1) &< & c_{1}(c_{2}-1) \end{eqnarray*} So if $A <
B=c_{2}A$, then $1< c_{2}$, and this gives us $c_{2} < c_{1}$, which
means $1 < \frac{c_{1}}{c_{2}}$. Multiplying both sides by
$\frac{a}{A}$ this means $\frac{a}{A} < \frac{b}{B}$, contradiction.
So $A \geq B$.

\item now suppose $\frac{a}{A} < \frac{b}{B}$. Then again, that
$\frac{a}{A}$ is closer to $\frac{a+b}{A+B}$ then $\frac{b}{B}$
means
\begin{eqnarray*}
\frac{a+b}{A+B} - \frac{a}{A} &<& \frac{b}{B} - \frac{a+b}{A+B} \\
\frac{1+c_{1}}{1+c_{2}} - 1 &<& \frac{c_{1}}{c_{2}} -
\frac{1+c_{1}}{1+c_{2}} \\
2 \left( \frac{1+c_{1}}{1+c_{2}} \right) &<& 1 + \frac{c_{1}}{c_{2}}
\\
2c_{2}(1+c_{1}) &<& c_{2}(1+c_{2}) + c_{1}(1+c_{2}) \\
2c_{2} + 2c_{1}c_{2} &< & c_{2}+c_{2}^{2} + c_{1} + c_{1}c_{2} \\
c_{1}(c_{2}-1) &<& c_{2}(c_{2}-1) \end{eqnarray*} If $A<B$, then
$c_{2}>1$, and this gives us $c_{1} < c_{2}$, which means
$\frac{c_{1}}{c_{2}}< 1$. Multiplying by $\frac{a}{A}$ this says
$\frac{b}{B} < \frac{a}{A}$, contradiction. So $A \geq B$.
\end{itemize} \end{proof}

\begin{lemma} \label{triangle in Rn}
Given a triangle in $\mathbb{R}^{2}$ with vertices A,B, C, and
opposites of length a,b,c, satisfying $\frac{a+b}{c} \leq 1+
\epsilon$ for some $\epsilon \in [0,0.5]$, then \begin{itemize}
\item $d(C, \overline{AB}) \leq 1.5 \epsilon^{1/4} \overline{AB}$
\item $\min \{A,B \} \leq \max \{ \pi - \cos^{-1}(-1 +
\sqrt{\frac{\epsilon}{1+\epsilon}}), \sin^{-1}(\frac{\sqrt{\frac{\epsilon}{1+\epsilon}}}{2}) \}$
\end{itemize} \end{lemma}
\begin{proof} the condition on the length means

\begin{equation*} 1 \geq \frac{c^{2}}{(a+b)^{2}} = \frac{(a+b)^{2}-2ab(1+\cos(C))}{(a+b)^{2}} \geq
\frac{1}{1+\epsilon} \end{equation*} Write $\frac{1}{1+\epsilon}=
1-\hat{\epsilon}$,(note that $\hat{\epsilon}=1-\frac{1}{1+\epsilon}
\leq \epsilon$) for some small $\hat{\epsilon}>0$, we have

\begin{equation*} 0 \leq \frac{2ab}{(a+b)^{2}} (1+ \cos(c)) \leq
\hat{\epsilon} \end{equation*} which means EITHER  \begin{itemize}
\item $(1+ \cos(C)) \leq \sqrt{\hat{\epsilon}}$.

In this case, $\cos(C) \leq -(1-\sqrt{\hat{\epsilon}})$, so
$\cos^{-1}(-1+\sqrt{\hat{\epsilon}}) \leq C \leq \pi$, leaving $A, B
< A+ B \leq \pi - \cos^{-1}(-1+\sqrt{\hat{\epsilon}})$ giving
\begin{equation*}
d(C, \overline{AB})= |\overline{AC}| \sin(A)  \leq  |\overline{AB}|
\sin(\pi - \cos^{-1}(-1+\sqrt{\hat{\epsilon}}))= |\overline{AB}|
\sin(\cos^{-1}(-1+\sqrt{\hat{\epsilon}})) \end{equation*} Hence
\begin{equation*}
d(C, \overline{AB}) \leq |\overline{AB}|
\sqrt{1-(1-\sqrt{\hat{\epsilon}})^{2}} \leq
|\overline{AB}|\sqrt{(1-1+\sqrt{\hat{\epsilon}})(1+1-\sqrt{\hat{\epsilon}})}
\leq |\overline{AB}| \sqrt{2\sqrt{\hat{\epsilon}}}
\end{equation*}

OR \item $\frac{2ab}{(a+b)^{2}} \leq \sqrt{\hat{\epsilon}}$. By Sine
rule, this is the same thing as
\begin{equation*} \frac{2 \sin(A) \sin(B)}{(\sin(A)+\sin(B))^{2}}
\leq \sqrt{\hat{\epsilon}} \end{equation*}

Divide top and bottom by $\sin(B)$ (if $\sin(A)=\sin(B)=0$ then we
are done, so assume one of them is not zero) so
\begin{equation*}
2 \sin(A) \leq 2 \frac{\sin(A)}{\sin(B)} \leq \frac{2
\frac{\sin(A)}{\sin(B)}}{\left( 1+\frac{\sin(A)}{\sin(B)}
\right)^{2}} \leq \sqrt{\hat{\epsilon}} \end{equation*} yields $A
\leq \sin^{-1}\left( \frac{\sqrt{\hat{\epsilon}}}{2} \right)$. Since
$\epsilon \leq 0.5$, $\hat{\epsilon}=1-\frac{1}{1+\epsilon} \leq
\frac{1}{3}$. So $\angle A \leq 16.78^{\circ}$. Since $C+B= \pi-A$,
WLOG $C \geq B$, $C \geq \frac{\pi-A}{2} \geq 45^{\circ}$ so
$\tan(C) \geq 1$. Therefore \begin{eqnarray*}
\frac{|\overline{AC}|}{|\overline{AB}|} = \frac{\sin(B)}{\sin(C)}&=&
\frac{\sin(\pi-C-A)}{\sin(C)}=\frac{\sin(\pi-C)\cos(A)}{\sin(C)} -
\frac{\sin(A)\cos(\pi-C)}{\sin(C)} \\
&=& \cos(A) + \frac{\sin(A)}{\tan(C)} \leq \cos(A) + \sin(A) \leq 2
\end{eqnarray*}
Hence \begin{equation*} d(C, \overline{AB})=\sin(A) |\overline{AC}|
\leq \sin(A) 2 |\overline{AB}| \leq \frac{\sqrt{\hat{\epsilon}}}{2}2
|\overline{AB}| = \sqrt{\hat{\epsilon}} |\overline{AB}|
\end{equation*} \end{itemize} \end{proof} \bigskip

\begin{proof} \textit{of Lemma \ref{orientation of a quadrilateral}}
The quadrilateral is the same as the loop below.
\begin{figure}[h]
\centering
\includegraphics{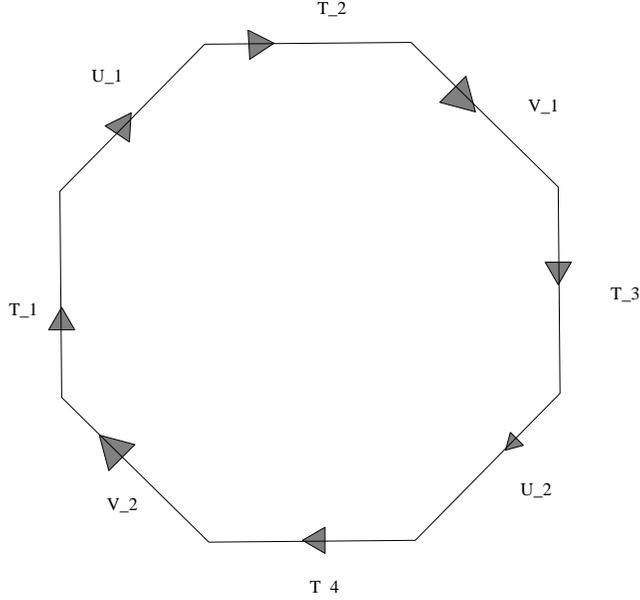}\\
\caption{The loop given by a quadrilateral} \end{figure}

Write $\mathbf{T}_{i}=T_{i}v$.  Since $|U_{1}|$, $|U_{2}|$,
$|V_{1}|$, $|V_{2}|$ are all less than $\eta(\sum |
\mathbf{T}_{i}|)$, the first claim that  $\sum_{i=1}^{4} T_{i} \leq
\eta(\sum_{i=1}^{4} |\mathbf{T}_{i}|)$ follows by walking around the
loop associated to $Q$.\medskip

So it cannot be the case that all the $T_{i}$'s are of the same
sign. WLOG we can assume $T_{2}>0$, and $T_{3} <0$. Furthermore,
regardless of the signs of the remaining $T_{i}$'s, there must be
another pair of adjacent $T_{i}$'s of opposite signs, and either
this pair involves one of $\{ T_{2}, T_{3} \}$, or that it doesn't.
In the latter case, $T_{1} >0$ and $T_{4}<0$, and the projection of
this quadrilateral into $\langle v \rangle \ltimes \mathbb{R}^{m}$
is a quadrilateral with two consecutive upward and two consecutive
downward edges, and such a quadrilaterals doesn't exist. \medskip

So either $T_{2}$ or $T_{3}$ is involved in a pair of oppositely
signed edges.  WLOG, we assume $T_{1} <0$. Then by (iv) in the
definition of a quadrilateral, we have that
$d(e,\Pi_{W^{+}_{v}}(U_{1})) \geq 1$, because $T_{1} <0$ and $T_{2}
>0$; and $d(e,\Pi_{W^{-}_{v}}(V_{1})) \geq 1$, because $T_{2} >0$
and $T_{3} <0$, where $\Pi_{W^{+}_{v}}:(x,t) \mapsto
\pi_{W^{+}_{v}}(x)$, $\pi_{W^{+}_{v}}$ is the usual projection from
$\mathbb{R}^{m}$ to $W^{+}_{v}$. $\Pi_{W^{-}_{v}}$ is defined
similarly. \medskip

Suppose $T_{4}<0$. Then $|T_{2}|=|T_{1}|+|T_{3}|+|T_{4}|$. Writing
the loop as: \begin{eqnarray*}
e &=&  \mathbf{T}_{2}V_{1}\mathbf{T}_{3}U_{2}\mathbf{T}_{4}V_{2}\mathbf{T}_{1}U_{1}  \\
& = &
(\mathbf{T}_{2}V_{1}\mathbf{T}_{2}^{-1})(\mathbf{T}_{2}\mathbf{T}_{3}U_{2}\mathbf{T}_{3}^{-1}\mathbf{T}_{2}^{-1})
(\mathbf{T}_{2}\mathbf{T}_{3}\mathbf{T}_{4}V_{2}\mathbf{T}_{1})U_{1}
\end{eqnarray*}

\noindent we see that only in the first bracket do we have a
coordinate of size $e^{|T_{2}|}$. So $T_{4}>0$, and again by (iv) in
the definition of quadrilateral, we conclude that for $i=1, 2$,
$d(e,\Pi_{W^{+}_{v}}(U_{i})) \geq 1$, $d(e,\Pi_{W^{-}_{v}}(V_{i}))
\geq 1$. \end{proof}

\begin{proof} \textit{of Lemma \ref{quadrilateral word in rank 1}}
Summing the $\mathbb{R}$ coordinates we see that
$r_{0}+r_{2}=r_{1}+r_{3}$.  The identity word can be written as
\begin{eqnarray*}
e &=& (r_{0},0)u_{0}(-r_{1},0)u_{1}(r_{2},0)u_{2}(-r_{3},0)u_{3} \\
& = &
((r_{0},0)u_{0}(-r_{0},0))((r_{0}-r_{1},0)u_{1}(r_{1}-r_{0},0))
((r_{3},0)u_{2}(-r_{3},0))u_{3} \end{eqnarray*} we see that
$|r_{0}-r_{3}| \leq d(e,u_{0})+d(e,u_{2})$, and $|r_{0}-r_{1}| \leq
d(e, u_{1})+ d(e, u_{3})$ by comparing the $W^{+}$ and $W^{-}$
coordinates. \smallskip

Similarly by looking at the word starting from $(-r_{1},0)$ we have
\begin{eqnarray*}
e &=& (-r_{1},0)u_{1}(r_{2},0)u_{2}(-r_{3},0)u_{3}(r_{0},0)u_{0} \\
& = &
((-r_{1},0)u_{1}(r_{1},0))((-r_{1}+r_{2},0)u_{2}(-r_{2}+r_{1},0))
((-r_{0},0)u_{3}(r_{0},0))u_{0} \end{eqnarray*} which gives us that
$| r_{1} -r_{0} |\leq d(e,u_{1})+d(e, u_{3})$, and $| r_{2}-r_{1} |
\leq d(e,u_{2})+d(e,u_{0})$.  We obtain the desired claim by writing
the word starting at $(r_{2},0)$ and $(-r_{3},0)$ and argue
similarly as above. \end{proof}

\begin{proof} \text{of Lemma \ref{ping-pong averaging}}
Equip the set $A \times B$ with the product measure $\mu=\mu_{\alpha} \times \mu_{\beta}$.  The measure of the set
$R=\{(a,b): a \sim b\}$ is therefore $\mu(R)=\int_{A} \mu_{\beta}(B_{a}) d \mu_{\alpha}=\int_{B} \mu_{\alpha}(A_{b}) d
\mu_{\beta}$.  Hence
\begin{equation} \label{bounds}
 \frac{1}{M_{B}} \frac{\mu(R)}{\mu_{\beta}(B)} \leq
\mu_{\alpha}(A_{b})_{\min}, \mbox{        }
\mu_{\beta}(B_{a})_{\max} \leq \frac{\mu(R)}{\mu_{\alpha}(A)} M_{A}
\end{equation} \noindent Let $\chi$ be the characteristic function of the set $\{ (a,b):
a \sim b, a \in A_{s} \}$.  Then
\begin{equation*}
\int_{B} \left( \int_{A_{b}} \chi d \mu_{\alpha} \right) d
\mu_{\beta} = \int_{A} \left( \int_{B_{a}} \chi d \mu_{\beta}
\right) d \mu_{\alpha} =\int_{A_{s}} \mu_{\beta}(B_{a}) d
\mu_{\alpha} \leq s \mu_{\alpha}(A) \mbox{
}\mu_{\beta}(B_{a})_{\max}
\end{equation*} \begin{equation*}
\int_{B} \left( \int_{A_{b}} \chi d \mu_{\alpha} \right) d
\mu_{\beta} \geq \int_{B^{s,t}} \left( \int_{A_{b}} \chi d
\mu_{\alpha} \right) d \mu_{\beta} \geq t \int_{B^{s,t}}
\mu_{\alpha}(A_{b}) d\mu_{\beta} \geq t \mu_{\alpha}(A_{b})_{\min}
\mbox{   }  \mu_{\beta}(B^{s,t}) \end{equation*}

\noindent Therefore
\begin{equation*} \mu_{\beta}(B^{s,t}) \leq \frac{s \mu_{\alpha}(A)
\mbox{   } \mu_{\beta}(B_{a})_{\max}}{t \mu_{\alpha}(A_{b})_{\min}}
\leq \frac{s}{t} M_{A}M_{B} \mbox{  } \mu_{\beta}(B) \end{equation*}
where the last inequality comes from quoting equation(\ref{bounds})
\end{proof}

\end{document}

%% file: quadrilateral.pstex_t
\begin{picture}(0,0)%
\includegraphics{quadrilateral.pstex}%
\end{picture}%
\setlength{\unitlength}{2171sp}%
\begingroup\makeatletter\ifx\SetFigFont\undefined%
\gdef\SetFigFont#1#2#3#4#5{%
  \reset@font\fontsize{#1}{#2pt}%
  \fontfamily{#3}\fontseries{#4}\fontshape{#5}%
  \selectfont}%
\fi\endgroup%
\begin{picture}(6678,5856)(3451,-6124)
\put(8326,-5986){\makebox(0,0)[lb]{\smash{{\SetFigFont{8}{9.6}{\familydefault}{\mddefault}{\updefault}{\color[rgb]{0,0,0}$p'_{2}$}%
}}}}
\put(8026,-2461){\makebox(0,0)[lb]{\smash{{\SetFigFont{9}{10.8}{\familydefault}{\mddefault}{\updefault}{\color[rgb]{0,0,0}$T_{2}$}%
}}}}
\put(7051,-3961){\makebox(0,0)[lb]{\smash{{\SetFigFont{9}{10.8}{\familydefault}{\mddefault}{\updefault}{\color[rgb]{0,0,0}$T_{3}$}%
}}}}
\put(4501,-5986){\makebox(0,0)[lb]{\smash{{\SetFigFont{8}{9.6}{\familydefault}{\mddefault}{\updefault}{\color[rgb]{0,0,0}$p_{1}$}%
}}}}
\put(3451,-3436){\makebox(0,0)[lb]{\smash{{\SetFigFont{9}{10.8}{\familydefault}{\mddefault}{\updefault}{\color[rgb]{0,0,0}$T_{4}$}%
}}}}
\put(6151,-511){\makebox(0,0)[lb]{\smash{{\SetFigFont{8}{9.6}{\familydefault}{\mddefault}{\updefault}{\color[rgb]{0,0,0}$q_{1}$}%
}}}}
\put(9376,-6061){\makebox(0,0)[lb]{\smash{{\SetFigFont{8}{9.6}{\familydefault}{\mddefault}{\updefault}{\color[rgb]{0,0,0}$p_{2}$}%
}}}}
\put(4351,-511){\makebox(0,0)[lb]{\smash{{\SetFigFont{8}{9.6}{\familydefault}{\mddefault}{\updefault}{\color[rgb]{0,0,0}$q'_{2}$}%
}}}}
\put(3451,-736){\makebox(0,0)[lb]{\smash{{\SetFigFont{8}{9.6}{\familydefault}{\mddefault}{\updefault}{\color[rgb]{0,0,0}$q_{2}$}%
}}}}
\put(3826,-5836){\makebox(0,0)[lb]{\smash{{\SetFigFont{8}{9.6}{\familydefault}{\mddefault}{\updefault}{\color[rgb]{0,0,0}$p'_{1}$}%
}}}}
\put(5176,-3286){\makebox(0,0)[lb]{\smash{{\SetFigFont{9}{10.8}{\familydefault}{\mddefault}{\updefault}{\color[rgb]{0,0,0}$T_{1}$}%
}}}}
\put(7126,-436){\makebox(0,0)[lb]{\smash{{\SetFigFont{8}{9.6}{\familydefault}{\mddefault}{\updefault}{\color[rgb]{0,0,0}$q'_{1}$}%
}}}}
\end{picture}%

%% file: mak-Simplex1.pstex_t
\begin{picture}(0,0)%
\includegraphics{mak-Simplex1.pstex}%
\end{picture}%
\setlength{\unitlength}{1855sp}%
\begingroup\makeatletter\ifx\SetFigFont\undefined%
\gdef\SetFigFont#1#2#3#4#5{%
  \reset@font\fontsize{#1}{#2pt}%
  \fontfamily{#3}\fontseries{#4}\fontshape{#5}%
  \selectfont}%
\fi\endgroup%
\begin{picture}(5942,6830)(814,-6369)
\put(3976,-5611){\makebox(0,0)[lb]{\smash{{\SetFigFont{7}{8.4}{\familydefault}{\mddefault}{\updefault}{\color[rgb]{0,0,0}$\mathit{l}_{q,2}$}%
}}}}
\put(3751,-3361){\makebox(0,0)[lb]{\smash{{\SetFigFont{7}{8.4}{\familydefault}{\mddefault}{\updefault}{\color[rgb]{0,0,0}OR}%
}}}}
\put(3601,-4186){\makebox(0,0)[lb]{\smash{{\SetFigFont{7}{8.4}{\familydefault}{\mddefault}{\updefault}{\color[rgb]{0,0,0}$p$}%
}}}}
\put(6151,-5986){\makebox(0,0)[lb]{\smash{{\SetFigFont{7}{8.4}{\familydefault}{\mddefault}{\updefault}{\color[rgb]{0,0,0}$q$}%
}}}}
\put(6301,-4411){\makebox(0,0)[lb]{\smash{{\SetFigFont{7}{8.4}{\familydefault}{\mddefault}{\updefault}{\color[rgb]{0,0,0}$\mathit{l}_{q,1}$}%
}}}}
\put(4426,-3736){\makebox(0,0)[lb]{\smash{{\SetFigFont{7}{8.4}{\familydefault}{\mddefault}{\updefault}{\color[rgb]{0,0,0}$\mathit{l}_{p}$}%
}}}}
\put(3901,-1936){\makebox(0,0)[lb]{\smash{{\SetFigFont{7}{8.4}{\familydefault}{\mddefault}{\updefault}{\color[rgb]{0,0,0}p}%
}}}}
\put(1651,-736){\makebox(0,0)[lb]{\smash{{\SetFigFont{7}{8.4}{\familydefault}{\mddefault}{\updefault}{\color[rgb]{0,0,0}q}%
}}}}
\put(3451,-211){\makebox(0,0)[lb]{\smash{{\SetFigFont{7}{8.4}{\familydefault}{\mddefault}{\updefault}{\color[rgb]{0,0,0}r}%
}}}}
\put(3826,-511){\makebox(0,0)[lb]{\smash{{\SetFigFont{7}{8.4}{\familydefault}{\mddefault}{\updefault}{\color[rgb]{0,0,0}$\mathit{l}_{r}$}%
}}}}
\put(2401,-2386){\makebox(0,0)[lb]{\smash{{\SetFigFont{7}{8.4}{\familydefault}{\mddefault}{\updefault}{\color[rgb]{0,0,0}$\mathit{l}_{p}$}%
}}}}
\put(901,-1711){\makebox(0,0)[lb]{\smash{{\SetFigFont{7}{8.4}{\familydefault}{\mddefault}{\updefault}{\color[rgb]{0,0,0}$\mathit{l}_{q}$}%
}}}}
\end{picture}%

%% file: mak-Simplex2.pstex_t
\begin{picture}(0,0)%
\includegraphics{mak-Simplex2.pstex}%
\end{picture}%
\setlength{\unitlength}{1973sp}%
\begingroup\makeatletter\ifx\SetFigFont\undefined%
\gdef\SetFigFont#1#2#3#4#5{%
  \reset@font\fontsize{#1}{#2pt}%
  \fontfamily{#3}\fontseries{#4}\fontshape{#5}%
  \selectfont}%
\fi\endgroup%
\begin{picture}(8520,5916)(739,-5419)
\put(4351,-886){\makebox(0,0)[lb]{\smash{{\SetFigFont{8}{9.6}{\familydefault}{\mddefault}{\updefault}{\color[rgb]{0,0,0}$z$}%
}}}}
\put(2476,-61){\makebox(0,0)[lb]{\smash{{\SetFigFont{8}{9.6}{\familydefault}{\mddefault}{\updefault}{\color[rgb]{0,0,0}$P$}%
}}}}
\put(1426,-5161){\makebox(0,0)[lb]{\smash{{\SetFigFont{8}{9.6}{\familydefault}{\mddefault}{\updefault}{\color[rgb]{0,0,0}$x$}%
}}}}
\put(7501,-5161){\makebox(0,0)[lb]{\smash{{\SetFigFont{8}{9.6}{\familydefault}{\mddefault}{\updefault}{\color[rgb]{0,0,0}$q$}%
}}}}
\put(7126,-3286){\makebox(0,0)[lb]{\smash{{\SetFigFont{8}{9.6}{\familydefault}{\mddefault}{\updefault}{\color[rgb]{0,0,0}$Q_{q,1}$}%
}}}}
\put(6001,-3736){\makebox(0,0)[lb]{\smash{{\SetFigFont{8}{9.6}{\familydefault}{\mddefault}{\updefault}{\color[rgb]{0,0,0}$Q_{q,2}$}%
}}}}
\put(1501,-1711){\makebox(0,0)[lb]{\smash{{\SetFigFont{8}{9.6}{\familydefault}{\mddefault}{\updefault}{\color[rgb]{0,0,0}$Q_{p,1}$}%
}}}}
\put(5551,239){\makebox(0,0)[lb]{\smash{{\SetFigFont{8}{9.6}{\familydefault}{\mddefault}{\updefault}{\color[rgb]{0,0,0}$Q_{p,2}$}%
}}}}
\put(7426,-1186){\makebox(0,0)[lb]{\smash{{\SetFigFont{8}{9.6}{\familydefault}{\mddefault}{\updefault}{\color[rgb]{0,0,0}$Q_{x}$}%
}}}}
\end{picture}%